\documentclass[letterpaper,11pt,oneside,reqno]{article}

\usepackage[pdftex,backref=page,colorlinks=true,linkcolor=blue,citecolor=red]{hyperref}
\usepackage[alphabetic,nobysame]{amsrefs}

\usepackage{amsmath,amssymb,amsthm,amsfonts}
\usepackage{graphicx,color}
\usepackage{upgreek}
\usepackage[mathscr]{euscript}
\usepackage{mathtools}

\allowdisplaybreaks
\numberwithin{equation}{section}

\usepackage{tikz}
\usetikzlibrary{shapes,arrows,positioning,decorations.markings}

\usepackage{array}
\usepackage{adjustbox}
\usepackage{cleveref}
\usepackage{enumerate}

\usepackage[DIV=12]{typearea}

\newcommand{\ssp}{\hspace{1pt}}

\newtheorem{proposition}{Proposition}[section]
\newtheorem{lemma}[proposition]{Lemma}

\newtheorem{theorem}[proposition]{Theorem}
\newtheorem*{theorem*}{Theorem}
\theoremstyle{definition}
\newtheorem{definition}[proposition]{Definition}
\newtheorem{remark}[proposition]{Remark}

\newcommand{\pcu}{\begin{tikzpicture}[baseline=(current bounding box.center)] 
\draw (0.6,1)--(1,1)--(1,1.5);
\draw (1,0.6)--(1,1.5);
\draw (1,1.5)--(1,1.9);
\draw (0.6,1.5)--(1.4,1.5);
\draw[ultra thick] (0.6,1)--(1.4,1);
\end{tikzpicture}}
\newcommand{\pau}{\begin{tikzpicture}[baseline=(current bounding box.center)] 
\draw (0.6,1)--(1,1)--(1,1.5);
\draw (1,1.5)--(1,1.9);
\draw (0.6,1.5)--(1.4,1.5);
\draw[ultra thick] (1,0.6)--(1,1)--(1.4,1);
\end{tikzpicture} }
\newcommand{\pbu}{ \begin{tikzpicture}[baseline=(current bounding box.center)] 
\draw (0.6,1)--(1,1)--(1,1.5);
\draw (1,0.6)--(1,1.5);
\draw (1,1.5)--(1,1.9);
\draw (0.6,1.5)--(1.4,1.5);
\draw[ultra thick] (0.6,1.5)--(1,1.5)--(1,1.9);
\draw[ultra thick] (0.6,1)--(1.4,1);
\end{tikzpicture}}
\newcommand{\pciu}{\begin{tikzpicture}[baseline=(current bounding box.center)] 
\draw (0.6,1)--(1,1)--(1,1.5);
\draw (1,0.6)--(1,1.5);
\draw (1,1.5)--(1,1.9);
\draw (0.6,1.5)--(1.4,1.5);
\draw[ultra thick] (0.6,1.5)--(1,1.5)--(1,1.9);
\draw[ultra thick] (1,0.6)--(1,1)--(1.4,1);
\end{tikzpicture}}

\newcommand{\pcd}{ \begin{tikzpicture}[baseline=(current bounding box.center)] 
\draw (0.6,1)--(1.4,1);
\draw (1,0.6)--(1,1.9);
\draw (0.6,1.5)--(1.4,1.5);
\draw[ultra thick] (0.6,1.5)--(1,1.5)--(1,1.9);
\draw[ultra thick] (1,0.6)--(1,1.5)--(1.4,1.5);
\end{tikzpicture} }
\newcommand{\pad}{ \begin{tikzpicture}[baseline=(current bounding box.center)] 
\draw (0.6,1)--(1.4,1);
\draw (1,0.6)--(1,1.9);
\draw (0.6,1.5)--(1.4,1.5);
\draw[ultra thick] (0.6,1.5)--(1,1.5)--(1,1.9);
\draw[ultra thick] (0.6,1)--(1,1)--(1,1.5)--(1.4,1.5);
\end{tikzpicture} }
\newcommand{\pbd}{\begin{tikzpicture}[baseline=(current bounding box.center)] 
\draw (0.6,1)--(1.4,1);
\draw (1,0.6)--(1,1.9);
\draw (0.6,1.5)--(1.4,1.5);
\draw[ultra thick] (1,0.6)--(1,1)--(1,1.5)--(1.4,1.5);
\end{tikzpicture} }

\newcommand{\pcid}{\begin{tikzpicture}[baseline=(current bounding box.center)] 
\draw (0.6,1)--(1.4,1);
\draw (1,0.6)--(1,1.9);
\draw (0.6,1.5)--(1.4,1.5);
\draw[ultra thick] (0.6,1)--(1,1)--(1,1.5)--(1.4,1.5);
\end{tikzpicture}}

\newcommand{\ru}{
	\begin{tikzpicture}
		[scale=.7,very thick]
		\def\dd{.3}
		\draw (0,0)--++(\dd,0)--++(0,\dd);
\end{tikzpicture}\hspace{.8pt}}

\begin{document}
\title{Irreversible Markov Dynamics and
Hydrodynamics for KPZ States in the Stochastic Six Vertex Model}


\author{Matthew Nicoletti and Leonid Petrov}

\date{}

\maketitle

\begin{abstract}
	We introduce a family of Markov growth 
	processes 
	on discrete height functions
	defined on the 2-dimensional square lattice. 
	Each height
	function corresponds to a configuration of the
	six vertex
	model on the infinite square lattice.
	We focus on the stochastic six vertex model corresponding to a particular
	two-parameter family of weights within the ferroelectric regime.
	It is believed (and partially proven, see Aggarwal
	\cite{aggarwal2020nonexistence})
	that the stochastic six vertex model
	displays nontrivial pure (i.e., translation invariant and
	ergodic) Gibbs states of two types,
	KPZ and liquid. These phases have
	very different long-range correlation structure.
	The Markov processes we construct 
	preserve the KPZ pure states in the full plane.
	We also show that the same processes put on the torus
	preserve
	arbitrary Gibbs measures 
	for generic six vertex weights (not necessarily in the ferroelectric
	regime).

	Our
	dynamics arise naturally from the Yang--Baxter equation
	for the six vertex model via its bijectivisation, a technique first used in
	Bufetov--Petrov \cite{BufetovPetrovYB2017}.
	The dynamics we construct
	are irreversible; in particular, the 
	height function has a nonzero average drift.
	In each KPZ pure state, we explicitly compute the 
	average drift (also known as the current)
	as a function of the slope. 
	We use this to analyze the
	hydrodynamics 
	of a non-stationary
	version of our process acting on quarter 
	plane stochastic six vertex
	configurations. The fixed-time limit shapes in the quarter plane
	model
	were obtained in 
	Borodin--Corwin--Gorin \cite{BCG6V}.
\end{abstract}

\setcounter{tocdepth}{1}
\tableofcontents
\setcounter{tocdepth}{3}

\section{Introduction}
\label{sec:introduction}

\subsection{Overview}
\label{sub:KPZ}

The study of models for surface growth is pervasive in many
areas of science and engineering. Some physical examples
include front propagation in combustion and crystal growth
\cite{damron2016random}.

A large collection of surface growth models falls under the
umbrella of the \emph{KPZ universality class}, 
named after the Kardar--Parisi--Zhang stochastic
partial differential equation
\cite{KPZ1986}. 
Models in
the KPZ class 
are formulated as randomly growing 
height functions $h_t(x)$,
where $x\in \mathbb{R}^d$ is the space variable, 
and $t\in \mathbb{R}_{\ge0}$ is time.
The models in the KPZ class
have the following three characteristic properties 
\cite{CorwinKPZ}, \cite{halpin2015kpzCocktail} \cite{QuastelSpohnKPZ2015}:
\begin{itemize}
	\item \textbf{Smoothing}. The dynamics tends to force
		large fluctuations in the height function back towards
		the mean.
	\item \textbf{Slope dependent growth speed}. The average
		velocity of growth of the height function at a point only
		depends on its average slope around that point.
	\item \textbf{Space-time uncorrelated noise}. 
		The randomness in the model comes from a random environment
		which is space-time uncorrelated, such as the white noise
		in the KPZ equation.
\end{itemize}
For models in the KPZ class, and for growth models in general,
it is often intractable to obtain exact expressions for 
various observables.
However, there exist \emph{integrable} models in which algebraic
structure allows for a precise analysis of observables. In
this work, we construct a new integrable random growth model in two space
dimensions,
based on the six vertex model. The latter is well-studied in
statistical mechanics by techniques 
of quantum integrability (in particular, Bethe Ansatz). 
We refer to the book 
\cite{baxter2007exactly} for an introduction, and also to
\cite{reshetikhin2010lectures} for a more recent survey of the six 
vertex model.

There are two 
basic questions in the study of a growth model:
the identification of the \emph{translation invariant stationary
measures} (that is, measures
which are invariant both under space translation
and under the stochastic dynamics), and the computation of the 
\emph{current} $J(\nabla h)$.
The
current is the velocity of the height
function growth as a function of a particular interface
\emph{slope} $\nabla h_t(x)$, where $\nabla$ is the 
gradient in space.

Further crucial questions include the identification of
the scaling exponents $\alpha, \beta$ such that 
in the translation invariant stationary situation,
the size of the fluctuations of $h_t(x) - h_t(y)$
is proportional to $|x - y|^\alpha$ for large $|x-y|$ (and fixed $t$), 
and the standard
deviation of $h_t(x) - h_0(x)$ is proportional to
$t^{\beta}$ for large $t$. 
Scaling
exponents 
$\alpha=\frac{1}{2}$, 
$\beta=\frac{1}{3}$
for 
models in the \emph{(1+1)-dimensional KPZ class} (corresponding
to space dimension $d=1$) are
well-understood at a heuristic level \cite{KPZ1986},
\cite{Spohn2012}, and have been verified rigorously in several
concrete situations, see \cite{BenassiFouque1987},
\cite{GwaSpohn1992},
\cite{bertiniGiacomin1997stochastic},
\cite{imamura2004fluctuations},
\cite{ferrari2006scaling},
\cite{ImamuraSasamoto2011current}, 
\cite{BorodinCorwinFerrariVeto2013},
\cite{Amol2016Stationary}, \cite{imamura2017fluctuations}.

\medskip

In the case of higher dimensions,
the mathematical study of KPZ growth models is much more limited
(for example, see
\cite{halpin2015kpzCocktail}
for a discussion of applicable
numerical and renormalization group methods).
In the (2+1)-dimensional situation (the case $d=2$),
the sign of the Hessian of the current 
distinguishes 
between two different subclasses, 
\emph{isotropic} (positive Hessian) and \emph{anisotropic}
(negative Hessian), of
the KPZ class, with very different large scale behavior.
While in the isotropic case there exist only numerical
estimates of the (nonzero) scaling exponents,
for anisotropic KPZ models (AKPZ, for short)
it is expected that $\alpha = \beta = 0$.
In particular, 
the average fluctuation of the height function
grows at a rate slower than any polynomial in $t$.
See also the recent results 
\cite{cannizzaro2020stationary}, \cite{cannizzaro2021weak}
on scaling exponents and renormalization in the anisotropic KPZ
partial differential equation.

The first (2+1)-dimensional anisotropic KPZ model
which was studied rigorously 
is related to the dimer model on the hexagonal lattice
(we refer to \cite{Kenyon2007Lecture} for details on this dimer model).
This Markov dynamics was introduced in 
\cite{BorFerr2008DF} in the non-stationary regime (that is, for a particular densely packed
initial condition).
After that, Toninelli
\cite{Toninelli2015-Gibbs}
showed that the Markov dynamics
on Gibbs measures on the full plane is well-defined
(that is, almost surely it
does not make infinitely many jumps in finite time in a finite region) and
preserves pure Gibbs states of an arbitrary slope $(\mathsf{s},\mathsf{t})$.
A~pure state is, by definition, a translation invariant
and ergodic Gibbs measure, and $\mathsf{s}$ and $\mathsf{t}$
may be interpreted as average densities of dimers
of any two orientations (out of three orientations on the 
hexagonal lattice).
The existence
of the dynamics and the preservation of pure Gibbs states
follows from a coupling of the dynamics on a large torus
$\mathbb{T}_L$ with that on $\mathbb{Z}^2$, and
requires nontrivial probabilistic arguments. 

An explicit formula for the current
$J(\mathsf{s},\mathsf{t})$ for this dynamics on $\mathbb{Z}^2$
was conjectured in \cite{BorFerr2008DF}, \cite{Toninelli2015-Gibbs}
based on particular cases, and then proven in
\cite{chhita2017combinatorial}.
The 
negative Hessian 
of
$J(\mathsf{s},\mathsf{t})$
puts the model into the 
anisotropic KPZ class.
Under this dynamics, 
the 
height
function fluctuations should grow at most as
$\sqrt{\log t}$ as $t\to+\infty$. This was shown in
\cite{BorFerr2008DF} for a particular initial condition, and in 
\cite{Toninelli2015-Gibbs} under some conditions on the slopes.
In a subsequent work
on the AKPZ class \cite{CFT2019}, an analogous dynamics on
dimer coverings of the square lattice is studied. The current
for this process is computed explicitly, and it has negative Hessian.
Moreover, \cite{CFT2019} presented a
new simplified argument for the $\sqrt{\log t}$
fluctuations for both
square and hexagonal cases, which in
particular removes the technical assumption of
\cite{Toninelli2015-Gibbs}.

It is known 
that the pure Gibbs states (with $(\mathsf{s},\mathsf{t})$
away from a finite number of points in the polygon of
allowed slopes) for the hexagonal dimer model
are all \emph{liquid} \cite{KOS2006}, that is, the 
variance of the height function difference 
$h_t(x)-h_t(y)$ 
grows logarithmically
with distance $|x-y|$ (with time $t$ fixed and $x,y \in \mathbb{R}^2$). Moreover,
the limiting fluctuation field is the conformally invariant
(massless) Gaussian Free Field \cite{Kenyon2004Height}, \cite{Petrov2012GFF}.
In other words, the Markov dynamics on 
the hexagonal dimer model indeed displays scaling exponents $\alpha=\beta=0$.

We remark 
that the celebrated
domino shuffling of 
\cite{elkies1992alternating},
\cite{propp2003generalized}
also fits into the framework of \cite{BorFerr2008DF}
and thus belongs to the anisotropic KPZ class.
We refer to 
\cite{borodin2015random},
\cite{chhita2019}, and
\cite{borodin2018two} for details.

\medskip

The stochastic six vertex model 
introduced in \cite{GwaSpohn1992}
may be viewed as a certain model of interacting dimers.
The presence of interaction introduces
a new type of pure Gibbs states, the \emph{KPZ pure states} 
$\pi({\mathsf{s}})$,
arising for a special one-parameter family of slopes
$(\mathsf{s},\mathsf{t})$ with
$\mathsf{t}=\varphi(\mathsf{s})$, where 
\begin{equation}
	\label{eq:intro_phi_definition_new_with_u}
	\varphi(\mathsf{s}) = \varphi(\mathsf{s}\mid u)
	\coloneqq
	\frac{\mathsf{s}}{\mathsf{s}+u-\mathsf{s}u}
	.
\end{equation}
Here $u\in(0,1)$
is a parameter of the six vertex
weights (i.e., it parametrizes the Gibbs property
which gives rise to various pure states)
which we will usually omit in the notation.
The KPZ pure states exhibit scaling exponent $\alpha=\frac{1}{3}$
along a single direction in the plane, and $\frac{1}{2}$
along all other directions \cite{Amol2016Stationary}. 
In particular, the limiting fluctuations cannot be described by a
conformally invariant field.
The name ``KPZ state'' comes from the fact that this stochastic
six vertex model configuration in the plane may be
viewed as a trajectory of a stationary
Markov chain on particle configurations on~$\mathbb{Z}$
(coming from the vertex model's
transfer matrix which is a stochastic matrix) belonging to the 
(1+1)-dimensional KPZ universality class \cite{BCG6V},
\cite{corwin2020stochastic}.

It is believed (but not yet proven)
that the stochastic six vertex model also displays 
liquid pure Gibbs states
arising for generic slopes $(\mathsf{s},\mathsf{t})$ outside a certain region 
$\mathfrak{H}$
in the two-dimensional space of slopes (see
\Cref{fig:lens} for an illustration). 
The slopes on the boundary of $\mathfrak{H}$ correspond to KPZ pure states,
and there do not exist pure states with slopes inside $\mathfrak{H}$. This was conjectured in
\cite{shore1994coexistence},
\cite{bukman1995conical},
and proven recently in
\cite{aggarwal2020nonexistence}.

\medskip

The \textbf{main results of our paper}
are:
\begin{itemize}
	\item 
		We construct an irreversible
		continuous time
		Markov dynamics
		$\mathcal{C}(t)$, $t\in \mathbb{R}_{\ge0}$,
		on six vertex configurations in the full plane $\mathbb{Z}^2$.
		The dynamics
		preserves the
		KPZ pure Gibbs state $\pi({\mathsf{s}})$.
		More precisely, we
		show that the full plane Markov dynamics
		is well-defined 
		(i.e., almost surely 
		does not make infinitely many jumps in finite time in a finite region) 
		when 
		started from $\pi({\mathsf{s}})$ (\Cref{thm:C_t_full_plane_well_defined}), 
		and that $\pi({\mathsf{s}})$
		is its stationary distribution (\Cref{thm:C_t_preserves_Gibbs}).
	\item We show that the current 
		(that is, the average drift of the height function)
		under the dynamics $\mathcal{C}$ 
		started from $\pi({\mathsf{s}})$
		is given by an explicit formula
		(\Cref{thm:current_computation}),
		\begin{equation}
			\label{eq:intro_current}
			J(\mathsf{s},\varphi(\mathsf{s}\mid u))
			=
			-\frac{\mathsf{s}\ssp(1 - \mathsf{s})}
			{(\mathsf{s} + u - \mathsf{s} u)^2} 
			=
			-\frac{(1-\mathsf{s})\bigl(\varphi(\mathsf{s}\mid u)\bigr)^2 }{\mathsf{s}}
			=
			\frac{\partial}{\partial u}\ssp \varphi(\mathsf{s}\mid u).
		\end{equation}
	\item 
		In \Cref{sec:torus}
		we present an analogue of the continuous time dynamics
		$\mathcal{C}(t)$ which lives 
		on six vertex configurations on a torus.
		This construction is independent of the full plane one.
		Moreover, we are able to 
		extend our Markov processes to the case of the most
		general six vertex weights (traditionally denoted by 
		$a_1,a_2,b_1,b_2,c_1,c_2$), and for
		pure states on the torus
		with arbitrary average slopes
		$(\mathsf{s},\mathsf{t})$
		(in particular, away from the one-parameter family of KPZ pure states).
		Our torus dynamics is structurally similar
		to the one of \cite{BorodinBufetov2015},
		but they are different.
\end{itemize}

\medskip

Let us make some general remarks on these results, and then in
\Cref{sub:dynamics_definition} below we formulate them in detail.

\medskip

Algebraically, both the construction of the full plane
dynamics and the dynamics on the torus 
come from the bijectivisation
of the Yang--Baxter equation
\cite{BufetovPetrovYB2017}, which
turns the equation into a Markov map. 
Its action maps the random six vertex configuration
into a random six vertex configuration with swapped
spectral parameters.
Composing these Markov maps and passing to a Poisson type 
continuous time
limit following the approach in \cite{PetrovSaenz2019backTASEP}
produces the jump rates in the full plane dynamics. 
In the torus case, we deform the six vertex graph on the 
torus by a certain twist to 
define the Markov maps, and then observe that in the continuous time
limit the twist trivializes with probability going to one.
We also present another proof of the fact that the dynamics on the 
torus preserves the Gibbs distributions using
an adaptation of \cite{BorodinBufetov2015},
which allows to 
generalize our
construction to arbitrary six vertex weights and arbitrary
average slopes.

\medskip

For the full plane model, as in the dimer case
considered in \cite{Toninelli2015-Gibbs},
we employ a nontrivial probabilistic
argument to prove that the Markov dynamics 
$\mathcal{C}(t)$ is well-defined.
There are two major obstacles which we overcome 
compared to the dimer case.
First, instead of coupling the full plane dynamics to 
the Hammersley process
\cite{hammersley1972few}, \cite{aldous1995hammersley}
to upper bound the probability of long jumps,
we need to consider a new particle system in which
particles can jump and annihilate one another.
For this Annihilation-Jump process we are able to produce the 
required estimates. 
Furthermore, unlike in the dimer case, the stationary measure
$\pi(\mathsf{s})$
does not possess a determinantal structure, 
so we have to restrict our analysis
to KPZ pure states admitting an explicit
description 
\cite{Amol2016Stationary}
as trajectories of a (1+1)-dimensional system, see
\Cref{sub:pure_states} below.

\medskip

We compute the current 
\eqref{eq:intro_current}
for all KPZ pure states of the stochastic
six vertex model. One in principle could employ our dynamics
on the torus 
to get the current for 
general six vertex weights and arbitrary slopes, 
but this computation seems out of reach with existing techniques.

\medskip

We believe that 
our Markov processes on the torus 
(extended to the general six vertex weights)
should
belong
to the anisotropic KPZ class,
at least for the weights
$a_1,a_2,b_1,b_2,c_1,c_2$ 
which are ``sufficiently close'' to the dimer situation
(cf. fluctuation universality for small perturbations away from the dimer 
case in \cite{Giuliani_Mastropietro_Toninelli2020nonintegrable}).
The six vertex model corresponds to a dimer model for several choices of parameters.
Let us mention two cases:
\begin{itemize}
	\item 
		Setting
		$a_1=a_2=b_1=b_2=1$ and $c_1=c_2=\sqrt 2$
		maps six vertex configurations into domino tilings
		\cite{elkies1992alternating}, \cite{zinn2000six}, \cite{ferrari2006domino}.
	\item 
		Setting $a_2=0$ and $b_1b_2=c_1c_2$ forbids the vertex 
		$(1,1;1,1)$ and maps six vertex configurations
		into Gibbs ensembles of nonintersecting lattice paths.
		Here ``Gibbs'' means that the measure 
		is invariant under uniform resampling.
		Such nonintersecting paths
		are readily identified with lozenge tilings, for example,
		see 
		\cite[Figures 2,3]{borodin2010gibbs}.
		In this case we relate our processes on the torus to the 
		ones studied in \cite{BorFerr2008DF}, \cite{Toninelli2015-Gibbs},
		\cite{CFT2019}, see \Cref{sub:lozenges}.
\end{itemize}
Both these degenerations to dimer models, however, 
cannot be obtained by directly specializing parameters 
in the stochastic six vertex
model. For the latter, 
we know the current $J(\mathsf{s},\varphi(\mathsf{s}))$
only along a one-dimensional curve of slopes, and so we cannot 
access the full Hessian of $J(\mathsf{s},\mathsf{t})$ to determine its sign.

\medskip

Above we have outlined and briefly discussed our main results. 
In \Cref{sub:dynamics_definition} below we define the Markov
dynamics $\mathcal{C}(t)$ on six vertex configurations and 
formulate the main theorems.

\subsection{Main results}
\label{sub:dynamics_definition}

Fix two parameters $0 < \delta_1 < \delta_2 < 1$
which
determine the weights of the stochastic six vertex model, 
see \Cref{fig:height_fn}, top, and also \Cref{sub:vertex_weights}
below for more detail.
Another useful parametrization is in terms of the
\emph{spectral parameter} $u =
\frac{1-\delta_1}{1-\delta_2}$, and the \emph{quantum
parameter} $q = \delta_1/\delta_2$, where $q,u\in (0,1)$. 
The parameters $(\delta_1,\delta_2)$ (or $(q,u)$)
define the Gibbs property.

We define a Markov process whose state space is formed by
collections of up-right lattice paths in $\mathbb{Z}^2$
which can meet at a vertex but cannot cross or share an
edge.
Allowed path configurations are in bijection with height
functions defined (up to a constant) on the faces of the
lattice such that around each vertex 
the differences of the height functions are of one of six types 
displayed in \Cref{fig:height_fn}, top.
An example of a state is given in
\Cref{fig:height_fn}, bottom. 

\begin{figure}[h]
\begin{center}
	\includegraphics[width=\textwidth]{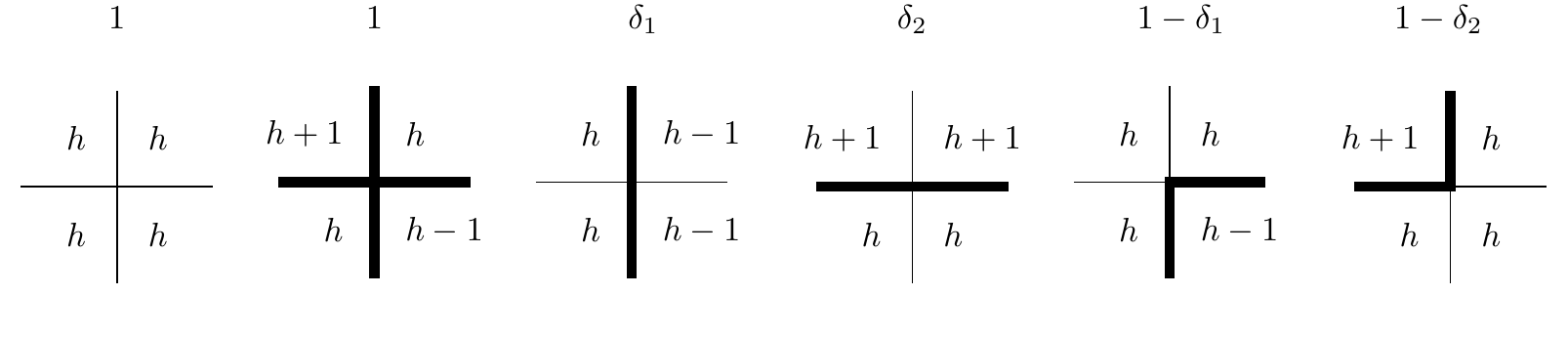}
	
	\includegraphics[width=.4\textwidth]{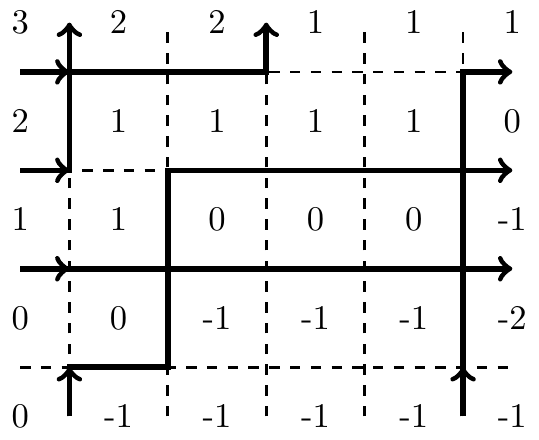}
 \end{center}
 \caption{Top: Local rules satisfied by a valid height
	 function, and vertex weights of each 
	local configuration. Bottom: A path configuration in the full plane
 and its corresponding height function (with some choice of
 the constant).}
 \label{fig:height_fn}
\end{figure}

The Markov process $\mathcal{C}(t)$, $t\in \mathbb{R}_{\ge0}$,
is driven by a collection of independent
Poisson clocks. Each vertical edge of $\mathbb{Z}^2$ has a
Poisson clock, and the rate of the clock at an edge
$(x,y)-(x,y+1)$, $x,y\in \mathbb{Z}$,
is a function of the local path
configuration around that edge, as shown in the table in
\Cref{fig:jumps_and_rates}. (Due to the memorylessness
of the Poisson process, this is the same as
attaching three independent Poisson clocks with rates
given in the table.)

\begin{figure}[htpb]
	\centering
	\begin{tabular}{ | m{5em} | m{2cm}| } 
\hline
Up jump & Rate \\ 
	\hline
	$\pcu $		& $\frac{1-q}{(1-u)\, (1 - q u)}$ \\[23pt]
	\hline
	$\pau $		& $\frac{(1 - u)\, q}{(1 - q)\,(1 - q u)\, u }$  \\[23pt]
		\hline
		$\pbu $		& $\frac{1 - q u}{(1 - u) \, (1 - q)\, u}$ \\[23pt]
		\hline
	\end{tabular}
	\qquad 
	\qquad 
	\begin{tabular}{ | m{5em} | m{2cm}|  } 
	\hline
	Down jump & Rate \\ 
		\hline
		$\pcd $		& $\frac{1-q}{(1-u)\, (1 - q u)}$ \\[23pt] 
		\hline
		$\pad $		& $\frac{(1 - u)\, q}{(1 - q)\,(1 - q u)\, u }$  \\[23pt] 
		\hline
	$\pbd $		& $\frac{1 - q u}{(1 - u) \, (1 - q)\, u}$ \\[23pt] 
		\hline
	\end{tabular}
	\caption{Jump rates for a given vertical edge, depending on the nearby path configuration. Denote by $(x,y)$
	the coordinates of the lower of the two vertices.}
\label{fig:jumps_and_rates}
\end{figure}

When the local configuration at some $(x,y)\in \mathbb{Z}^2$
corresponds to an 
up jump (\Cref{fig:jumps_and_rates}, left) and the clock rings, 
then the path passing through
the edge $(x,y)-(x+1,y)$
jumps up by $1$. 
Instantaneously, 
each path through $(x+i,y)-(x+i+1,y)$ jumps
up as well, for $i=1,2,\dots, c$, where $c$ is the largest
integer such that for each $i=1,2,\dots, c$,
$(x+i,y)-(x+i+1,y)$ is occupied and the horizontal edge
immediately above is unoccupied. 
Then $(x,y)-(x,y+1)$ becomes occupied, and we also remove the vertical
path from the 
vertical edge
$(x+c+1,y)-(x+c+1,y+1)$ which was previously occupied.
Similarly, when the local configuration
at $(x,y)$ corresponds to a down jump, 
then the path occupying the horizontal edge
$(x,y+1)-(x+1,y+1)$ jumps down, and again a maximal
sequence of adjacent occupied horizontal edges also jumps
down. 
In words, both for up and down jumps we restore the
path configuration to the allowed state using 
the most possible horizontal
edges to the right
of $(x,y)$.
See also \Cref{fig:jump_prop} in the text
for an illustration of a jump.

The KPZ pure Gibbs state $\pi(\mathsf{s})$
is a 
translation invariant ergodic Gibbs
measure
on full plane path configurations
under which
the 
average density of occupied
vertical edges is
$\mathsf{s}\in(0,1)$, and 
the average density of occupied horizontal edges
is $\varphi(\mathsf{s})$.

\begin{theorem*}[\Cref{thm:C_t_full_plane_well_defined,thm:C_t_preserves_Gibbs}
	in the text]
	Let $0 < \mathsf{s} < 1$. For a set $\Omega$
	of initial states 
	(\Cref{def:admissible_configurations} below)
	which is
	probability $1$ under the Gibbs measure $\pi(\mathsf{s})$, there
	exists a Markov process $\mathcal{C}(t)$
	whose generator acts according to
	the jump rates defined above. 
	Furthermore, $\pi(\mathsf{s})$ is
	stationary under $\mathcal{C}(t)$.
\end{theorem*}

We can
define the height change at a face 
under the Markov chain $\mathcal{C}(t)$
as follows.
If a path jumps up past a
face, then the height at this face height decreases by $1$, and if it jumps
down past a face, then the height increases by $1$. 
Denote the resulting randomly evolving height function by 
$h_t(x,y)$.
Define the
current in the KPZ pure state $\pi(\mathsf{s})$ by
\begin{equation}
	\label{eq:current_definition}
	J(\mathsf{s},\varphi(\mathsf{s})) \coloneqq 
	\frac{1}{t}
	\ssp
	\mathbb{E}_{\pi(\mathsf{s})}
	\left( h_t(0,0)-h_0(0,0) \right),
\end{equation}
where the dynamics starts in stationarity from $\pi(\mathsf{s})$.
Due to stationarity, the ratio in the right-hand
side of \eqref{eq:current_definition} is independent of $t$, so 
$J(\mathsf{s},\varphi(\mathsf{s}))$ is well-defined.
\begin{theorem*}
	[\Cref{thm:current_computation}
	in the text]
	The current is given by 
	formula \eqref{eq:intro_current}.
\end{theorem*}

Formula \eqref{eq:intro_current} for the current allows to predict a
hydrodynamic limit equation 
(in two space dimensions)
for the height function in a
non-stationary version 
of our dynamics acting in the quadrant
$\mathbb{Z}_{\ge0}\times \mathbb{Z}_{\ge1}$.
(In fact, we defined the full plane dynamics after looking 
at the bulk behavior of the quadrant dynamics.)
Under this non-stationary dynamics
which we denote by $\mathcal{Q}(\tau)$, the edge Poisson clocks
have rates depending on 
the lattice coordinate 
$y$ and time~$\tau$,
but otherwise the dynamics is the same as 
$\mathcal{C}(t)$, see \Cref{sub:q_cont_def} below.
The action of $\mathcal{Q}(\tau)$
changes (in distribution)
the Gibbs property of the stochastic six vertex model
by continuously increasing the spectral parameter 
from $u$ to some terminal value $u+\eta\in(0,1)$,
see \Cref{thm:quad_process_exists} in the text. 
We match
the current \eqref{eq:intro_current}
to a (heuristic) hydrodynamic limit of
$\mathcal{Q}(\tau)$. 
The latter should continuously transform the
explicit limit shapes
(obtained in 
\cite{BCG6V}, see also \cite{Reshetikhin2018LimitSO}, \cite{aggarwal2020limit})
of the stochastic six vertex model in the quadrant
with empty bottom and fully packed left boundary conditions.
See \Cref{fig:s6v_sim}, left, for a simulation of the stochastic six vertex
model with such boundary conditions.

\subsection{Outline}
\label{sub:outline}

In \Cref{sec:six_vertex_model} we define the
stochastic six vertex model,
describe Gibbs pure states,
and review the Yang--Baxter equation for the six stochastic six vertex model. 
In \Cref{sec:bijectivisation} we employ bijectivisation
to turn the Yang--Baxter equation 
into a Markov map, and introduce the resulting 
discrete time Markov dynamics on six 
vertex
configurations in the quadrant.
In
\Cref{sec:continuous_time_limit} we look at its continuous time
Poisson type limit leading to a Markov dynamics $\mathcal{Q}(\tau)$
in the quadrant, and state the measure mapping property.
In \Cref{sec:Markov_full_plane_process} we 
define the full plane dynamics $\mathcal{C}(t)$
and show that it exists and preserves the KPZ pure state
$\pi(\mathsf{s})$. The proof of the main estimate
is postponed to 
\Cref{sec:proof_lemma_S6}, where we describe the coupling
with the Annihilation-Jump particle system.
In \Cref{sec:hydrodynamics} we calculate the
current and discuss the hydrodynamic limit
of the dynamics $\mathcal{Q}(\tau)$ in the quadrant. In \Cref{sec:torus} we
present the construction of the dynamics on the torus, first
using bijectivisation, and then using an adaptation of the
proof in \cite{BorodinBufetov2015} to generalize our
Markov chain to general six vertex weights and arbitrary
average slopes.

\subsection{Notation}
\label{sub:notation}

Finite subsets of $\mathbb{Z}$ are denoted as
$\lambda=(\lambda_1>\ldots>\lambda_{\ell(\lambda)})$, where 
$\ell(\lambda)$ is the number of elements.
We also use letters like $\mu,\kappa,\nu$, and so on, for such finite subsets.
If all $\lambda_i\ge1$, the subset $\lambda$
may be viewed as a strict partition 
\cite[Example I.1.9]{Macdonald1995}
of the integer 
$|\lambda|\coloneqq\lambda_1+\ldots+\lambda_{\ell(\lambda)}$.

Arbitrary (possibly infinite) subsets of $\mathbb{Z}$ are denoted by 
similar letters but in the upright font: $\uplambda, \upkappa,\upmu,\upnu$, and so on.
For $h\ge1$ and 
$\uplambda\subseteq \mathbb{Z}_{\ge0}$, denote the
truncation $\uplambda^{[<h]}\coloneqq \uplambda\cap \left\{ 0,1,\ldots,h-1  \right\}$.

Each $\uplambda\subseteq \mathbb{Z}$ may be written in a multiplicative
notation $\uplambda=(\ldots(-1)^{m_{-1}}0^{m_0}1^{m_1}2^{m_2}\ldots  )$,
where $m_i=1$ if $i\in \uplambda$, and $m_i=0$ otherwise.

For an event or condition $A$,
$\mathbf{1}_{A}$ stands for the indicator of $A$.  

\subsection{Acknowledgments}

We are grateful to Amol Aggarwal, Alexei Borodin,
David Keating, and Fabio Toninelli
for helpful discussions.
The work was partially supported by the NSF 
grant DMS-1664617, and the 
Simons Collaboration Grant for Mathematicians 709055. 
This material is based upon work supported by the National
Science Foundation under Grant No. DMS-1928930 while LP
participated in program hosted by the Mathematical
Sciences Research institute in Berkeley, California, during
the Fall 2021 semester.

\section{Stochastic six vertex model}
\label{sec:six_vertex_model}

Here we define our 
main ``static'' objects --- certain families of
probability measures on configurations of up-right lattice paths
on edges of $\mathbb{Z}^2$
given by the stochastic six vertex model.

\subsection{Vertex weights}
\label{sub:vertex_weights}

Throughout the paper we fix the ``quantization'' parameter
$q\in[0,1)$.

The six vertex model is a probability measure on configurations
of up-right paths on the two-dimensional discrete lattice $\mathbb{Z}^2$,
such that there is at most one path allowed per edge. 
The paths are allowed to meet at a vertex.
See \Cref{fig:boundary_conditions} below for an example of a
path configuration.

For a single vertex, let $i_1,j_1,i_2,j_2\in \left\{ 0,1 \right\}$
be the number of paths entering and exiting the vertex as 
shown in \Cref{fig:6types}, left. 
Due to the path preservation, it must be $i_1+j_1=i_2+j_2$,
which leaves six possibilities for a vertex.
We assign the following weights to these six possible vertices:
	\label{eq:w_u_weights}
\begin{equation}
	\begin{split}
		w_u(0,0;0,0)&=
		w_u(1,1;1,1)=1,\\
		w_u(1,0;1,0) &= \frac{q (1 - u)}{1 - q u},
		\qquad 
		w_u(0,1;0,1)  =  \frac{1 - u }{1 - q u} ,
		\\
		w_u(1,0;0,1) &= \frac{1 - q}{1 - q u}
		,\qquad 
		w_u(0,1;1,0)  = \frac{u(1 -q) }{1 - q u}.
	\end{split}
\end{equation}
Here $u\in(0,1)$ is the \emph{spectral parameter} (also 
called the \emph{rapidity}) which may change from one vertex to another.

\begin{figure}[htpb]
	\centering
	\includegraphics[width=.9\textwidth]{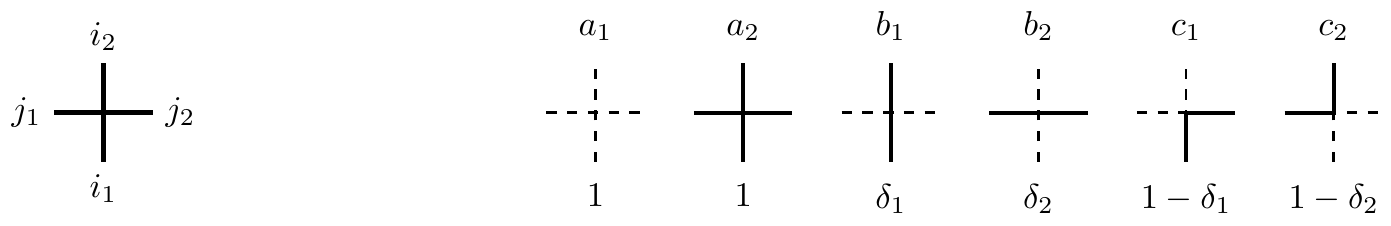}
	\caption{Stochastic six vertex weights arise as a specialization
		of the general asymmetric six vertex Boltzmann
		weights $(a_1,a_2,b_1,b_2,c_1,c_2)$
		as $a_1=a_2=1$, $b_1=1-c_1=\delta_1$, and 
		$b_2=1-c_2=\delta_2$.
		The general Boltzmann weights are
		listed above the vertices.}
	\label{fig:6types}
\end{figure}

We will use the parametrization 
of the vertex weights
by $(q,u)$ as in \eqref{eq:w_u_weights}
interchangeably with 
another one. This other parametrization 
involves two parameters $(\delta_1,\delta_2)$
with $0\le \delta_1<\delta_2\le 1$, see \Cref{fig:6types} for 
an illustration. The two parametrizations are related by
\begin{equation}
	\label{eq:delta_1_2_through_t_u}
	\delta_1=\delta_1(u)\coloneqq\frac{q (1 - u)}{1 - q u},\qquad 
	\delta_2=\delta_2(u)\coloneqq\frac{1 - u}{1 - q u},\qquad 
	q=\frac{\delta_1}{\delta_2}.
\end{equation}
In particular, the ratio $\delta_1/\delta_2$ is 
fixed throughout.

The weights \eqref{eq:w_u_weights}
satisfy the \emph{stochasticity} property at each vertex, namely,
\begin{equation}
	\label{eq:stochasticity_of_w_u}
	\sum_{i_2,j_2}w_u(i_1,j_1;i_2,j_2)=1
	\qquad 
	\textnormal{for all $i_1,j_1$}.
\end{equation}
Thus, our model is a particular case of the general asymmetric six vertex
model with weights $a_1,a_2,b_1,b_2,c_1,c_2$
specialized as 
$a_1=a_2=1$, $b_1=1-c_1=\delta_1$, and 
$b_2=1-c_2=\delta_2$
(cf. \Cref{fig:6types}).
The stochastic case of the six vertex model was introduced
by Gwa and Spohn
\cite{GwaSpohn1992} who 
studied its hydrodynamics 
in 1+1 dimension.
Asymptotic 
Tracy--Widom GUE fluctuations in this model
were obtained in Borodin--Corwin--Gorin
\cite{BCG6V}.

\subsection{Gibbs property}
\label{sub:gibbs_property}

Let us fix a finite rectangle $\Lambda=\{x,x+1,\ldots,x+R \}\times
\left\{ y,y+1,\ldots,y+R'  \right\}
\subset \mathbb{Z}^{2}$, and 
specify \emph{boundary conditions}, that is, the 
positions of entering paths at the bottom and the left boundaries, 
as well as the positions of exiting paths at the upper and right boundaries
(see \Cref{fig:boundary_conditions} for an example).
\begin{figure}[htpb]
	\centering
	\includegraphics[width=.4\textwidth]{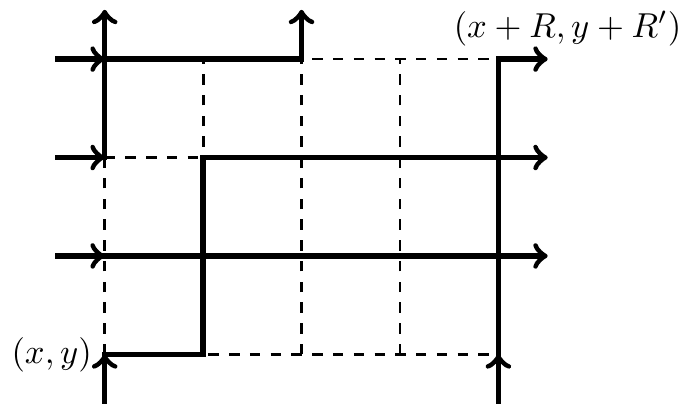}
	\caption{A rectangle with $R=4$, $R'=3$, and a configuration 
		of up-right paths with specified boundary conditions. 
		The boundary conditions are the positions of the 
		incoming and the outgoing paths, and are illustrated by arrows.}
	\label{fig:boundary_conditions}
\end{figure}

Observe that the set of all up-right path configurations
in $\Lambda$ with specified 
incoming and outgoing
boundary conditions is finite.
We consider a probability measure on these path configurations
depending on a sequence of spectral parameters
$\mathbf{u}=(u_y,u_{y+1},\ldots,u_{y+R'} )$ (with $u_l\in (0,1)$ for all $l$)
given by
\begin{equation}
	\label{eq:Gibbs_measures}
	\mathbb{P}_{\mathbf{u}}^{\mathrm{rect}}
	\left( \textnormal{path configuration} \right)
	\coloneqq\frac{1}{Z}
	\prod_{k=x}^{x+R}\,
	\prod_{l=y}^{y+R'}
	w_{u_l}\Biggl( 
		\begin{tikzpicture}
			[scale=.4,baseline=-3pt]
			\draw[fill] (0,0) circle (5pt);
			\draw[ultra thick] (-1,0)--++(2,0) node[right,yshift=3pt,xshift=-2pt] {\small$j_2$};
			\draw[ultra thick] (0,-1)--++(0,2) node[above,xshift=5pt,yshift=-4pt] {\small$i_2$};
			\node[below,xshift=-3pt,yshift=4pt] at (0,-1) {\small$i_1$};
			\node[left] at (-1,0) {\small$j_1$};
			\node at (1.4,-1) {$(k,l)$};
		\end{tikzpicture}
	\Biggr).
\end{equation}
The product in \eqref{eq:Gibbs_measures} is 
taken over all vertices $(k,l)$ in the lattice,
and for each vertex we pick one of the weights from
\eqref{eq:w_u_weights} (with the corresponding spectral parameter $u=u_l$)
depending on the occupation numbers $i_1,j_1,i_2,j_2$ 
of the edges adjacent to 
that vertex.
Here $Z$ is the \emph{partition function}, 
that is, the probability normalizing constant 
which ensures that the right-hand side of \eqref{eq:Gibbs_measures}
sums to $1$ over all possible path configurations with the given
boundary conditions.

One can also consider 
a probability measure 
$\mathbb{P}_{\mathbf{u}}^{\mathrm{free}}$
on path configurations
in the rectangle $\Lambda$ with \emph{free outgoing boundary conditions},
for which the locations and numbers of 
outgoing paths on the left and top boundaries of the rectangle
are not specified. Due to the 
stochasticity condition \eqref{eq:stochasticity_of_w_u},
the partition function of
$\mathbb{P}_{\mathbf{u}}^{\mathrm{free}}$
is simply equal to $1$,
and the measure $\mathbb{P}_{\mathbf{u}}^{\mathrm{free}}$
can be sampled by running a row-to-row Markov chain
based on the vertex weights $w_u$,
see \Cref{fig:P_u_free} for an illustration.
The conditional distributions of 
$\mathbb{P}_{\mathbf{u}}^{\mathrm{free}}$
with fixed locations of all outgoing paths
are precisely the measures $\mathbb{P}_{\mathbf{u}}^{\mathrm{rect}}$.

\begin{figure}[htpb]
	\centering
	\includegraphics[width=\textwidth]{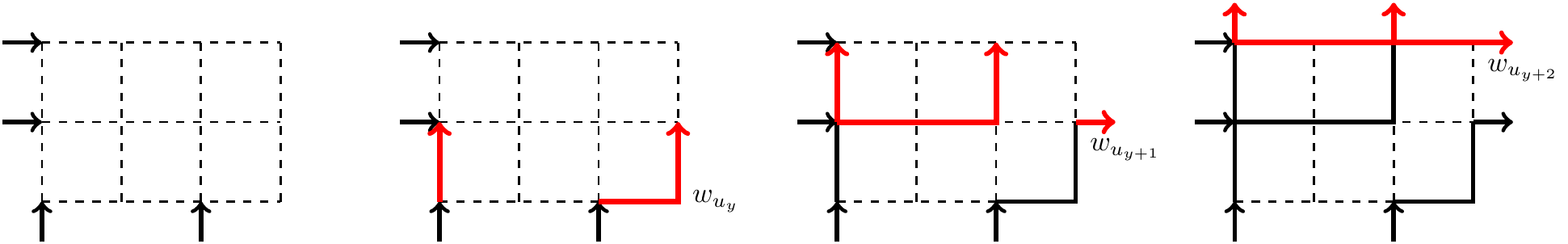}
	\caption{Sampling of $\mathbb{P}_{\mathbf{u}}^{\mathrm{free}}$ 
		by the row-to-row Markov chain. At each step, 
		we perform the sequential update (from left to right)
		in a single row, using the vertex weights $w_{u_{y+i}}$, $i=0,1,2$,
		and the incoming paths from the left and from the row below.
		For example, after one step the probability
		of getting the displayed configuration is
		equal to 
		$w_{u_y}(1,0;1,0)w_{u_y}(1,0;0,1)w_{u_y}(0,1;1,0)$.}
	\label{fig:P_u_free}
\end{figure}

We use the measures 
$\mathbb{P}_{\mathbf{u}}^{\mathrm{rect}}$
\eqref{eq:Gibbs_measures} 
as a building block for random ensembles of 
up-right paths in infinite regions of $\mathbb{Z}^2$.

\begin{definition}[Stochastic six vertex $\mathbf{u}$-Gibbs measures]
	\label{def:Gibbs}
	Let $\mathbf{u}=\{u_l\in(0,1)\colon l\in \mathbb{Z}\}$
	be a sequence of spectral parameters, and
	$\Uplambda\subseteq \mathbb{Z}^{2}$ a finite or infinite rectangular region.
	A probability measure $\mathbb{P}$ on configurations of 
	up-right paths in $\Uplambda$ is called 
	\emph{$\mathbf{u}$-Gibbs} if for any finite rectangle 
	$\Lambda=\{x,\ldots,x+R \}\times
	\left\{ y,\ldots,y+R'  \right\}\subset \Uplambda$,
	the conditional distribution of the up-right paths 
	in $\Lambda$ with arbitrary fixed incoming and outgoing boundary
	conditions 
	on all four sides of $\Lambda$ is 
	$\mathbb{P}^{\mathrm{rect}}_{\mathbf{u}\vert_\Lambda}$.
	Here
	$\mathbf{u}\vert_{\Lambda}\coloneqq(u_y,u_{y+1},\ldots,u_{y+R'} )$
	is the corresponding restriction of the spectral parameter sequence.
\end{definition}

\begin{remark}
	In \Cref{def:Gibbs}	
	it suffices to consider only the choices of boundary conditions
	for $\Lambda$
	having nonzero $\mathbb{P}$-probability.
\end{remark}

Let $\Uplambda$ be finite. In this case the $\mathbf{u}$-Gibbs measures
of \Cref{def:Gibbs}
on up-right path ensembles in $\Uplambda$
are mixtures of the measures 
$\mathbb{P}_{\mathbf{u}}^{\mathrm{rect}}$
corresponding to taking random boundary conditions
on all four sides of the finite rectangle $\Uplambda$.
Since the set of all possible boundary conditions
is finite, the mixture is also finite.
The measure $\mathbb{P}_{\mathbf{u}}^{\mathrm{free}}$
is a particular example of a $\mathbf{u}$-Gibbs measure
with random boundary conditions (on the right and top boundaries only).

\begin{definition}[Homogeneous stochastic six vertex Gibbs measures]
	\label{def:homogeneous_Gibbs}
	The above \Cref{def:Gibbs}
	deals with
	inhomogeneous (in the vertical direction)
	Gibbs measures for the stochastic six vertex model. 
	Setting 
	$u_l\equiv u\in(0,1)$ for all $l\in \mathbb{Z}$
	leads to the important subclass of
	\emph{homogeneous}
	stochastic six vertex Gibbs
	measures.
	The homogeneous Gibbs property is indexed by two parameters
	$(q,u)$ (equivalently, $(\delta_1,\delta_2)$, see \eqref{eq:delta_1_2_through_t_u}).
\end{definition}

\begin{remark}
	\label{rmk:ferroelectric}
	The space of homogeneous
	Gibbs measures coincides with the
	space of the Gibbs measures
	for the \emph{symmetric ferroelectric} six vertex model.
	The symmetry means that
	$a_1=a_2=a$, $b_1=b_2=b$, and $c_1=c_2=c$
	in the notation as in \Cref{fig:6types},
	and the ferroelectric condition reads
	$\Delta=(a^2+b^2-c^2)/(2ab)>1$.
	We refer to 
	\cite[Appendix A.1]{Amol2016Stationary} for details.
\end{remark}

Let us define a family of $\mathbf{u}$-Gibbs measures
on up-right paths
in the quadrant 
$\mathbb{Z}_{\ge0}\times \mathbb{Z}_{\ge 1}$
which are analogues of $\mathbb{P}_{\mathbf{u}}^{\mathrm{free}}$
and can also be sampled by running a row-to-row Markov chain:

\begin{definition}[Stochastic six vertex model in the quadrant]
	\label{def:s6v_quadrant}
	Fix a sequence of spectral parameters
	$u_l\in (0,1)$, $l=1,2,\ldots $,
	and take 
	$\Uplambda$ to be the quadrant $\mathbb{Z}_{\ge0}\times \mathbb{Z}_{\ge 1}$.
	The stochastic six vertex in the quadrant 
	with \emph{step}
	(also called \emph{half domain wall}) boundary conditions \cite{GwaSpohn1992}, \cite{BCG6V}
	has empty bottom boundary and an incoming
	path at each horizontal edge $(-1,l)-(0,l)$, $l\ge1$,
	along the left boundary.
	It is the unique $\mathbf{u}$-Gibbs measure
	for which for any $h\ge0$, $l\ge 1$
	the marginal distribution of the configuration
	in a finite rectangle of the form
	$\{0,1,\ldots,h \}\times\{1,\ldots,l\}$
	is $\mathbb{P}_\mathbf{u}^{\mathrm{free}}$
	(with empty and packed incoming conditions
	at the bottom, respectively, the left boundary).
	The uniqueness of thus described measure
	follows in a standard way from the Kolmogorov
	extension theorem, since the marginals
	$\mathbb{P}_{\mathbf{u}}^{\mathrm{free}}$
	in 
	finite rectangles are compatible.
	
	More generally, if the bottom boundary has 
	paths incoming at locations encoded by a 
	(finite or infinite)
	subset $\uplambda\subseteq \mathbb{Z}_{\ge0}$,
	and the left boundary is packed,
	then we call these the \emph{step-$\uplambda$ boundary conditions}.
	If the left boundary is empty and the bottom boundary is 
	encoded by $\uplambda$,
	we refer to this as the \emph{empty-$\uplambda$ boundary conditions}.

	We denote 
	by $\mathbb{P}_{\mathbf{u}}^{\mathrm{s6v}}$
	the distribution 
	of the stochastic six vertex model 
	in the quadrant
	(with step-$\uplambda$ or empty-$\uplambda$ boundary conditions,
	which is specified in each case separately).
\end{definition}

\subsection{Pure states}
\label{sub:pure_states}

Here we discuss homogeneous Gibbs measures 
on path ensembles in the full plane 
satisfying certain natural assumptions.

\begin{definition}
	\label{def:pure_state}
	Consider the homogeneous stochastic six vertex Gibbs property
	(\Cref{def:homogeneous_Gibbs}) with some parameters
	$(q,u)$.
	A 
	\emph{translation invariant, ergodic Gibbs measure}
	(also called a \emph{pure state}, for short)
	is a Gibbs probability measure on configurations of up-right
	paths in the whole plane $\mathbb{Z}^2$ which satisfies
	two additional properties:
	\begin{itemize}
		\item 
			The distribution of the path ensemble
			does not change under shifts
			of the underlying lattice $\mathbb{Z}^2$ by arbitrary 
			elements of $\mathbb{Z}^2$;
		\item The measure is ergodic, which, by definition, means that
			the probability of any 
			translation invariant event 
			(from the $\sigma$-algebra associated with path configurations on $\mathbb{Z}^2$)
			is either $0$ or $1$.
	\end{itemize}
\end{definition}

In particular, each pure state admits a \emph{slope}
$(\mathsf{s},\mathsf{t})\in[0,1]^2$, where
\begin{equation}
	\label{eq:slope}
	\begin{split}
		\mathsf{s}&\coloneqq
		\mathbb{P}\left( \textnormal{the vertical edge $(0,0)-(0,1)$ is 
		occupied by a path}\right),\\
		\mathsf{t}&\coloneqq
		\mathbb{P}\left( \textnormal{the horizontal edge $(0,0)-(1,0)$ is 
		occupied by a path}\right).
	\end{split}
\end{equation}
In other words, $\mathsf{s}$ and $\mathsf{t}$ are the 
average densities of the vertical and horizontal edges 
under the pure state. If a pure state has slope $(\mathsf{s},\mathsf{t})$,
we denote it by $\pi_{\mathsf{s},\mathsf{t}}$.

\begin{figure}[htpb]
	\centering
	\includegraphics[width=.4\textwidth]{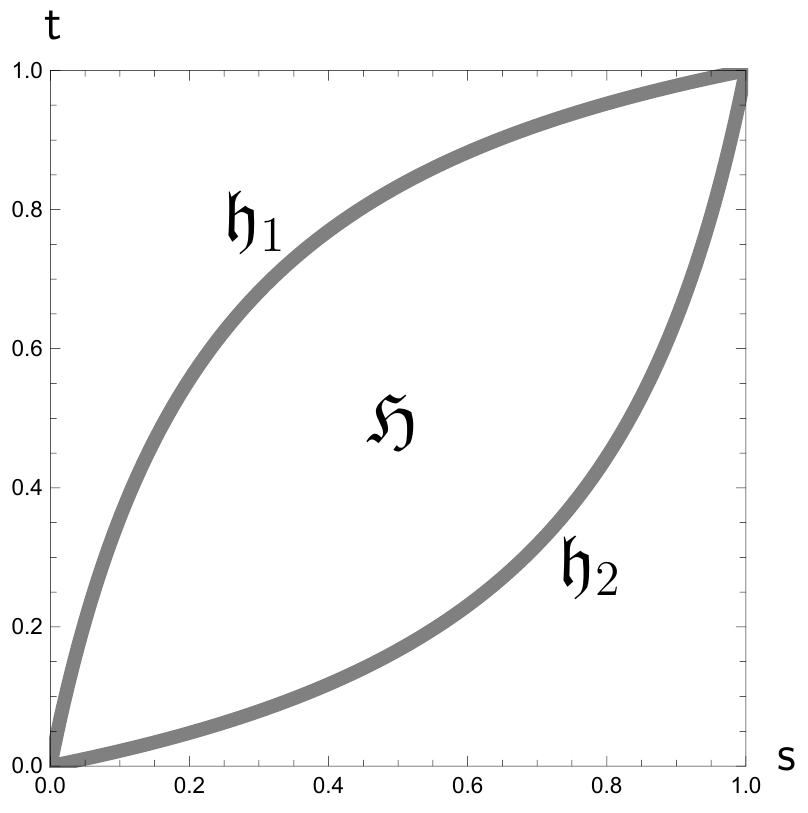}
	\caption{Phase diagram
	of pure states in the stochastic six vertex model for $u=\frac{1}{5}$.
	Remarkably, the phase diagram does not depend on $q$.}
	\label{fig:lens}
\end{figure}

Let us recall known results and predictions
about pure states $\pi_{\mathsf{s},\mathsf{t}}$.
We refer to 
\cite{aggarwal2020nonexistence}
for details, and follow the notation of that paper.
Let 
\begin{equation}
	\label{eq:phi_of_rho_define}
	\varphi(\mathsf{t})=\varphi(\mathsf{t}\mid u)\coloneqq
	\frac{\mathsf{t}}{\mathsf{t}+u-\mathsf{t}u},
	\qquad \mathsf{t}\in[0,1].
\end{equation}
We usually omit the spectral parameter $u$ in the notation.
Define the following subsets of $[0,1]^2$:
\begin{equation*}
	\mathfrak{h}_1\coloneqq
	\left\{ (\mathsf{s},\varphi(\mathsf{s}))\colon \mathsf{s}\in[0,1] \right\}
	,\qquad 
	\mathfrak{h}_2\coloneqq
	\left\{ (\varphi(\mathsf{t}),\mathsf{t})\colon \mathsf{t}\in[0,1] \right\}
	.
\end{equation*}
Let 
$\mathfrak{H}$ be the open subset of $[0,1]^2$ 
between $\mathfrak{h}_1$ and $\mathfrak{h}_2$,
so that $\partial\mathfrak{H}=\mathfrak{h}_1\cup \mathfrak{h}_2$.
See \Cref{fig:lens} for an illustration.

By \cite[Theorem 1.2]{aggarwal2020nonexistence},
for $(\mathsf{s},\mathsf{t})\in \mathfrak{H}$, there are 
no pure states with this slope,
and 
for each
$(\mathsf{s},\mathsf{t})\in \partial\mathfrak{H}$, 
there is exactly one pure state 
with this slope.
Since the phase diagram in \Cref{fig:lens}
is symmetric in $(\mathsf{s},\mathsf{t})$,
it suffices to consider only $\pi(\mathsf{s})\coloneqq\pi_{\mathsf{s},\varphi(\mathsf{s})}$.
For path configurations, the symmetry $(\mathsf{s},\mathsf{t})
\leftrightarrow(\mathsf{t},\mathsf{s})$
is realized by the reflection across the line $y = x$.

The measure
$\pi(\mathsf{s})$ is precisely the 
trajectory of the 
\emph{stationary 
stochastic six vertex model} from \cite{Amol2016Stationary}.
Namely, under
$\pi(\mathsf{s})$
the joint distribution of the locations
of the occupied entering vertical and horizontal edges along the boundary of any quadrant 
$\{x, x+1, \dots \} \times \{y,y+1, \dots \} \subset \mathbb{Z}^2$ 
is given by the Bernoulli product
measure with density $\mathsf{s}$ for the vertical and 
$\varphi(\mathsf{s})$ for the horizontal edges, respectively
(that is, each edge is occupied independently with probability
$\mathsf{s}$ for vertical and $\varphi(\mathsf{s})$ for horizontal entering edges).
Then given a boundary condition, we use the stochastic weights depending on $(q, u)$ to sample the configuration
in the quadrant according to the measure $\mathbb{P}_{u}^{\mathrm{s6v}}$.
In other words, the row-to-row transfer matrix in the stochastic six vertex model
defines an interacting particle system in (1+1) dimensions (that is, 
a discrete time Markov process on configurations in $\mathbb{Z}$).
The path configuration in $\mathbb{Z}^2$ under $\pi(\mathsf{s})$ is 
the trajectory of the evolution of the Bernoulli measure of density $\mathsf{s}$
(on a horizontal slice)
under this interacting particle system.

We refer to the 
part $\partial\mathfrak{H}$ of the phase diagram as the \emph{KPZ phase}
of the stochastic six vertex model, since in this phase the 
fluctuations of the height function live on scale $N^{1/3}$ 
(in domains of size $N$)
along 
the characteristic
direction, and are described by the Baik--Rains distribution \cite{Amol2016Stationary}
(the distribution was introduced in
\cite{baik2000limiting_BR_distribution}).
This asymptotic fluctuation behavior is typical
for stationary models in the KPZ
universality class \cite{CorwinKPZ}.

We discussed pure states
corresponding to slopes 
$(\mathsf{s},\mathsf{t})$ from
$\overline{\mathfrak{H}}\coloneqq\mathfrak{H}\cup \partial\mathfrak{H}$. 
For $(\mathsf{s},\mathsf{t})$
not on the boundary of $[0,1]^2$ but
outside of $\overline{\mathfrak{H}}$,
the pure states conjecturally exist 
and exhibit \emph{liquid phase} behavior in the sense of
\cite{KOS2006}. That is, the height
fluctuations in a domain of size $N$
grow logarithmically with $N$.
The liquid phase behavior for the stochastic six vertex model is
a major open problem.
In the present paper we mostly
deal with pure states in the KPZ phase, 
and it would be very interesting to extend at least some of our
results (besides the existence of the irreversible dynamics
on the torus)
to the liquid phase.

\subsection{Yang--Baxter equation}
\label{sub:YBE}

The vertex weights $w_u$
\eqref{eq:w_u_weights}
satisfy the Yang--Baxter equation which we now recall.
Define the cross vertex weights as follows:
\begin{equation}
	\label{eq:X_vertex}
	X_{u,v}(i_1,j_1;i_2,j_2)
	=
	X_{u,v}
	\Biggl(
		\scalebox{.5}{\begin{tikzpicture}
		[scale=1.5, baseline=16.5pt]
		\draw[line width=3] (0,0)--++(1,1);
		\draw[line width=3] (0,1)--++(1,-1);
		\node[left] at (0,0) {\LARGE$i_1$};
		\node[left] at (0,1) {\LARGE$i_2$};
		\node[right] at (1,0) {\LARGE$j_2$};
		\node[right] at (1,1) {\LARGE$j_1$};
		\end{tikzpicture}}
	\Biggr)
	\coloneqq w_{u/v}(i_1,j_1;i_2,j_2),\qquad i_1,j_1,i_2,j_2\in \left\{ 0,1 \right\}.
\end{equation}

\begin{proposition}[Yang--Baxter equation]
	\label{prop:YBE}
	For any fixed $i_1,i_2,i_3,j_1,j_2,j_3 \in \left\{ 0,1 \right\}$ we have
	\begin{equation}
		\label{eq:YBE}
		\begin{split}
			&
			\sum_{k_1,k_2,k_3 \in \left\{ 0,1 \right\}}
			X_{u,v}(i_2,i_1;k_2,k_1)\,
			w_u(i_3,k_1;k_3,j_1)\,
			w_v(k_3,k_2;j_3,j_2)
			\\&\hspace{100pt}=
			\sum
			_{k_1',k_2',k_3'\in \left\{ 0,1 \right\}}
			w_v(i_3,i_2;k_3',k_2')\,
			w_u(k_3',i_1;j_3,k_1')
			\,
			X_{u,v}(k_2',k_1';j_2,j_1)
			.
		\end{split}
	\end{equation}
	See \Cref{fig:YBE} for an illustration of the sums involved in both sides of 
	\eqref{eq:YBE}.
\end{proposition}
\begin{proof}
	There are finitely many choices of the boundary
	conditions $i_1,i_2,i_3,j_1,j_2,j_3$, and
	for each such choice
	the Yang--Baxter equation
	\eqref{eq:YBE} (an identity between rational functions
	in $u,v$, and~$q$)
	is verified in a straightforward way.
\end{proof}

\begin{figure}[htpb]
	\centering
	\includegraphics[width=.6\textwidth]{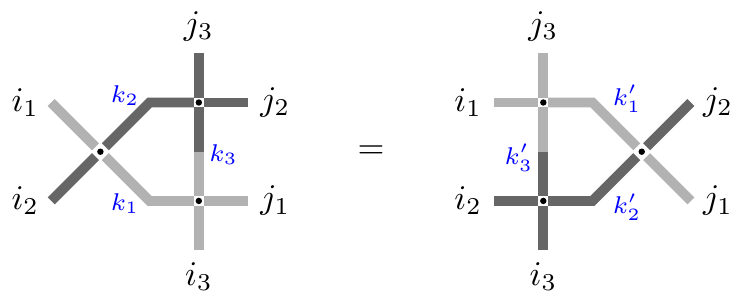}
	\caption{An illustration of the Yang--Baxter equation. 
		The straight vertex in the lighter shade has weight $w_u$, and in the 
		darker shade has weight $w_v$. The cross
		vertex has weight $X_{u,v}$.}
	\label{fig:YBE}
\end{figure}

\begin{remark}
	\label{rmk:number_of_terms_YBE}
	One can readily check that for any choice of 
	$i_1,i_2,i_3,j_1,j_2,j_3$
	the number of summands in each part of the Yang--Baxter
	equation
	\eqref{eq:YBE}
	is at most two.
\end{remark}

\subsection{Row operators}
\label{sub:row_operators}

Let $V=\mathbb{C}^2$ with the canonical orthonormal basis $e_0,e_1$.
We think that the space $V$ is associated
with the vertical direction at a vertex,
and use the vertex weights
$w_u$ \eqref{eq:w_u_weights}
to define four operators $\mathsf{A}_u,\mathsf{B}_u,\mathsf{C}_u$, and $\mathsf{D}_u$
in $V$. The operators $\mathsf{A}_u$ and $\mathsf{D}_u$ act diagonally:
\begin{equation*}
	\begin{split}
		\mathsf{A}_u e_0&=w_u(0,0;0,0)e_0,\qquad 
		\mathsf{A}_u e_1=w_u(1,0;1,0)e_1\\
		\mathsf{D}_u e_0&=w_u(0,1;0,1)e_0,\qquad 
		\mathsf{D}_u e_1=w_u(1,1;1,1)e_1,
	\end{split}
\end{equation*}
and the other two operators act as
\begin{equation*}
	\begin{split}
		\mathsf{B}_u e_0&=w_u(0,1;1,0)e_1,\qquad 
		\mathsf{B}_u e_1=0,
		\\
		\mathsf{C}_u e_1&=w_u(1,0;0,1)e_0,
		\qquad 
		\mathsf{C}_u e_0=0.
	\end{split}
\end{equation*}

Stacking the vertices horizontally, one can define the action of 
the operators 
$\mathsf{A}_u,\mathsf{B}_u,\mathsf{C}_u,\mathsf{D}_u$
in finite tensor powers
$V^{\otimes k}$. This can be done inductively:
\begin{equation*}
	\begin{split}
		\mathsf{A}_u(v_1\otimes v_2)&=
		\mathsf{A}_u v_1 \otimes \mathsf{A}_u v_2
		+
		\mathsf{C}_u v_1 \otimes \mathsf{B}_u v_2
		,
		\qquad 
		\mathsf{B}_u(v_1\otimes v_2)=
		\mathsf{B}_u v_1 \otimes \mathsf{A}_u v_2
		+
		\mathsf{D}_u v_1 \otimes \mathsf{B}_u v_2,
		\\
		\mathsf{C}_u(v_1\otimes v_2)&=
		\mathsf{C}_u v_1 \otimes \mathsf{D}_u v_2
		+
		\mathsf{A}_u v_1 \otimes \mathsf{C}_u v_2,
		\qquad 
		\mathsf{D}_u(v_1\otimes v_2)=
		\mathsf{D}_u v_1 \otimes \mathsf{D}_u v_2
		+
		\mathsf{B}_u v_1 \otimes \mathsf{C}_u v_2
		.
	\end{split}
\end{equation*}
Here $v_1\in V^{\otimes m},v_2\in V^{\otimes l}$, where $m+l=k$, $m,l<k$.

Multiplying two operators like $\mathsf{A}_u\mathsf{A}_v$ (acting in $V$)
corresponds to stacking two vertices vertically, with the spectral 
parameters $u$ 
and $v$ in the lower and the upper vertex, respectively.
Thanks to the Yang--Baxter equation (\Cref{prop:YBE}), the operators
satisfy a number of quadratic relations.
We do not need all of the relations, so let us 
list some of them which are used below in the paper:
\begin{align}
	\label{eq:commutation_relations_AABBCCDD}
		\mathsf{A}_u\mathsf{A}_v&=\mathsf{A}_v\mathsf{A}_u,\qquad 
		\mathsf{B}_u\mathsf{B}_v=\mathsf{B}_v\mathsf{B}_u,\qquad 
		\mathsf{C}_u\mathsf{C}_v=\mathsf{C}_v\mathsf{C}_u,\qquad 
		\mathsf{D}_u\mathsf{D}_v=\mathsf{D}_v\mathsf{D}_u,
		\\
		\label{eq:comm_AC_1}
		\mathsf{A}_v\mathsf{C}_u&=
		\frac{1-q}{1-qu/v}\,
		\mathsf{A}_u\mathsf{C}_v
		+
		\frac{1-u/v}{1-qu/v}\,
		\mathsf{C}_u\mathsf{A}_v,
		\\
		\label{eq:comm_AC_2}
		\mathsf{C}_v\mathsf{A}_u&=
		\frac{q(1-u/v)}{1-qu/v}\,
		\mathsf{A}_u\mathsf{C}_v+
		\frac{u/v(1-q)}{1-qu/v}\,
		\mathsf{C}_u\mathsf{A}_v,
		\\
		\label{eq:comm_BD_1}
		\mathsf{D}_v\mathsf{B}_u&=
		\frac{u/v(1-q)}{1-qu/v}\,
		\mathsf{D}_u\mathsf{B}_v+
		\frac{q(1-u/v)}{1-qu/v}\,
		\mathsf{B}_u\mathsf{D}_v
		,
		\\
		\label{eq:comm_BD_2}
		\mathsf{B}_v\mathsf{D}_u&=
		\frac{1-u/v}{1-qu/v}\,
		\mathsf{D}_u\mathsf{B}_v+
		\frac{1-q}{1-qu/v}\,
		\mathsf{B}_u\mathsf{D}_v
		.
\end{align}
By summing 
\eqref{eq:comm_AC_1} and \eqref{eq:comm_AC_2},
we see that 
$\mathsf{A}_v\mathsf{C}_u+\mathsf{C}_v\mathsf{A}_u$
is symmetric in $(u,v)$, 
which together with \eqref{eq:commutation_relations_AABBCCDD}
implies that
the operators
$\mathsf{A}_u+\mathsf{C}_u$ and 
$\mathsf{A}_v+\mathsf{C}_v$
commute. Similarly,
$\mathsf{B}_u+\mathsf{D}_u$ and 
$\mathsf{B}_v+\mathsf{D}_v$ commute.

Note relations \eqref{eq:commutation_relations_AABBCCDD}--\eqref{eq:comm_BD_2}
hold not only when acting in $V$ (in a single-vertex situation,
which is \Cref{prop:YBE}),
but also in each finite tensor power $V^{\otimes k}$. This is due to the fact that 
we can iterate the Yang--Baxter equation and move the cross
vertex horizontally through a row of adjacent vertices.

\subsection{Row operators in the half-infinite tensor product}
\label{sub:row_operators_infinite}

We need the half-infinite tensor product $V^{[0,\infty)}$ of the spaces $V$ with 
the fixed vector $e_0$.
By definition, the space
$V^{[0,\infty)}$ has the orthonormal
basis indexed by finite subsets
$\lambda\subset \mathbb{Z}_{\ge0}$
(recall notation from \Cref{sub:notation}),
where 
\begin{equation*}
	e_{\lambda}\coloneqq e_{m_0}\otimes e_{m_1}\otimes e_{m_2}\otimes \ldots,\qquad 
	m_i=\mathbf{1}_{i\in \lambda}.
\end{equation*}
We see that 
all but finitely many of the $m_i$'s 
are equal to $0$. We do not need the Hilbert space
completion of $V^{[0,\infty)}$ since all our
expressions involving $V^{[0,\infty)}$ below are in terms of 
matrix elements.

Let us define how the operators $\mathsf{A}_u$ and $\mathsf{B}_u$ act in $V^{[0,\infty)}$.
Representing 
$V^{[0,\infty)}=\bigoplus_{k=0}^{\infty}V^{[0,\infty)}_{k}$,
where 
$V^{[0,\infty)}_{k}$
is the span of $e_\lambda$ with $\ell(\lambda)=k$,
we have
\begin{equation*}
	\mathsf{A}_u\colon
	V^{[0,\infty)}_{k}
	\to
	V^{[0,\infty)}_{k},\qquad 
	\mathsf{B}_u\colon
	V^{[0,\infty)}_{k}
	\to
	V^{[0,\infty)}_{k+1}.
\end{equation*}
Indeed, each matrix element
$\langle \mathsf{A}_u e_\lambda,e_\mu \rangle $
with $\ell(\lambda)=\ell(\mu)$
is the product of the vertex weights
$w_u$ \eqref{eq:w_u_weights} over all vertices indexed by $\mathbb{Z}_{\ge0}$,
and in this product all but finitely many of the 
vertices have weight 
$w_u(0,0;0,0)=1$. This implies that the action 
of $\mathsf{A}_u$ (in terms of matrix elements) is well-defined,
and similarly for 
$\langle \mathsf{B}_u e_\lambda,e_\nu \rangle $, where
$\ell(\nu)=\ell(\lambda)+1$.

Note that the other two operators, $\mathsf{C}_u$ and $\mathsf{D}_u$,
do not act in the infinite tensor product
$V^{[0,\infty)}$ due to the presence of infinitely many
vertex weights $w_u(0,1;0,1)\ne 1$ in the corresponding products for their
matrix elements.

\begin{remark}
	\label{rmk:probabilities_as_matrix_elements}
	The operators in $V^{[0,\infty)}$ allow to express 
	probabilities in the stochastic six vertex model
	in the quadrant (\Cref{def:s6v_quadrant}) as matrix elements.
	For example, for the step boundary conditions, the probability
	to observe a configuration $\lambda$
	of occupied vertical edges 
	at height $l\ge1$ from the bottom is given by
	\begin{equation*}
		\langle e_\lambda, \mathsf{B}_{u_l}\ldots \mathsf{B}_{u_2}\mathsf{B}_{u_1}(e_0\otimes e_0\otimes e_0\otimes\ldots ) \rangle .
	\end{equation*}
	This probability is nonzero if and only if $\ell(\lambda)=l$.
\end{remark}

\section{Bijectivisation and transition probabilities}
\label{sec:bijectivisation}

\subsection{Bijectivisation of summation identities}
\label{sub:bijectivisation}

We recall the notion of bijectivisation 
from Bufetov--Petrov \cite[Section 2]{BufetovPetrovYB2017}.
Suppose we have two disjoint
finite sets $\mathcal{A}, \mathcal{B}$, so that each element
$x \in \mathcal{A} \cup \mathcal{B}$ is equipped with a
positive weight $\mathbf{w}(x)$, such that 
\begin{equation}
	\label{eq:bijectivisation_of_identity}
	\sum_{a\in \mathcal{A}}\mathbf{w}(a)
	=
	\sum_{b\in \mathcal{B}}\mathbf{w}(b).
\end{equation}
Identity \eqref{eq:bijectivisation_of_identity}
defines probability distributions on $\mathcal{A}$ and $\mathcal{B}$
with probability weights proportional to 
$\{\mathbf{w}(a)\}_{a\in \mathcal{A}}$
and
$\{\mathbf{w}(b)\}_{b\in \mathcal{B}}$,
respectively. A \emph{bijectivisation}
is a coupling between these two probability distributions,
expressed via conditional probabilities:

\begin{definition}[Bijectivisation]
	\label{def:bijectivisation}
	A \emph{bijectivisation}
	is a family of forward and backward
	transition probabilities
	$\mathbf{p}^{\mathrm{fwd}}(a \rightarrow b)\ge0$, $\mathbf{p}^{\mathrm{bwd}}(b \rightarrow a)\ge0$, 
	where $a\in \mathcal{A}$, $b\in \mathcal{B}$,
	satisfying
	\begin{itemize}
	\item Sum to one property:
	\begin{equation}
		\label{eq:bij_sum_to_one}
		\sum_{b \in \mathcal{B}} \mathbf{p}^{\mathrm{fwd}}(a \rightarrow b) = 1 \quad   \forall a \in \mathcal{A},
		\qquad 
		\qquad 
		\sum_{a \in \mathcal{A}} \mathbf{p}^{\mathrm{bwd}}(b \rightarrow a) = 1 \quad  \forall b \in \mathcal{B} .
	\end{equation}
	\item Reversibility condition:
	\begin{equation}
		\label{eq:bij_reversibility}
		\mathbf{w}(a)\, \mathbf{p}^{\mathrm{fwd}}(a \rightarrow b) = 
		\mathbf{w}(b)\, \mathbf{p}^{\mathrm{bwd}}(b \rightarrow a),
		\qquad 
		\forall a\in \mathcal{A},\ b\in \mathcal{B}.
	\end{equation}
	\end{itemize}
\end{definition}

For general $\mathcal{A}$ and
$\mathcal{B}$, a bijectivisation is not unique.

\begin{remark}
	\label{rmk:bij_is_unique}
	In the special case when $|\mathcal{A}|$ or $|\mathcal{B}|$ is equal to $1$,
	there is at most one bijectivisation. 
	More precisely, in this case the solution 
	$\{\mathbf{p}^{\mathrm{fwd}}(a \rightarrow b),
	\mathbf{p}^{\mathrm{bwd}}(b \rightarrow a)\}_{a\in \mathcal{A},\ b\in \mathcal{B}}$
	to 
	the family of linear equations
	\eqref{eq:bij_sum_to_one}--\eqref{eq:bij_reversibility}
	\emph{is unique}. When, moreover, this solution is 
	nonnegative, then it determines a 
	\emph{bona fide} bijectivisation
	(i.e., a stochastic Markov map).
\end{remark}

\begin{figure}[htb]
\centering
	\includegraphics[width=.85\textwidth]{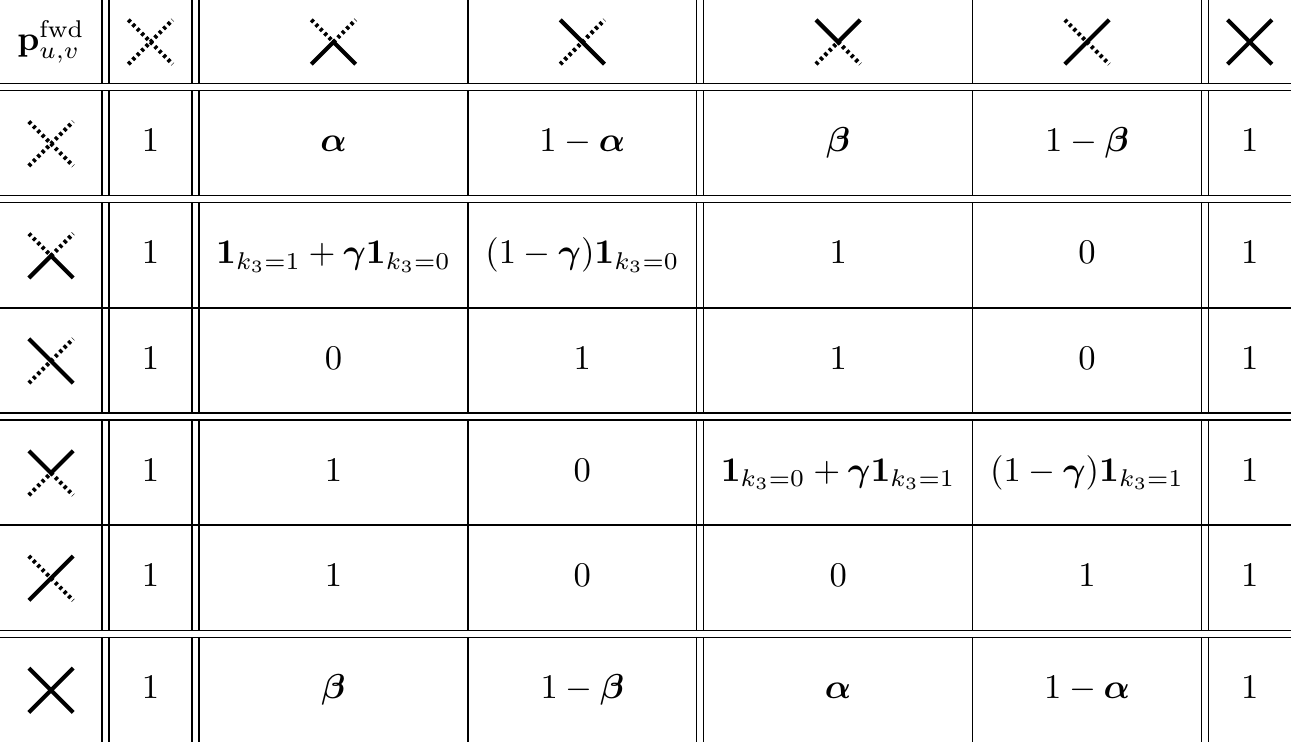}
	\caption{The bijectivisation of the Yang--Baxter equation in terms of the forward
		transition probabilities 
		$\mathbf{p}_{u,v}^{\mathrm{fwd}}\left( (k_1,k_2,k_3)\to (k_1',k_2',k_3')\mid I,J \right)$
		expressed through the 
		quantities \eqref{eq:YBE_bij_alpha_beta_gamma_parameters}.
		The row label is the cross vertex state on the left hand side, and the
		column label is the cross vertex state on the right hand side.
		The cross vertex states determine $i_1,i_2,k_1,k_2$ and $j_1,j_2,k_1,k_2'$
		(recall the labels in \Cref{fig:YBE}).
		Throughout most of the table, the remaining states of the edges 
		are determined uniquely (otherwise $I,J$ are incompatible).
		However, in four cells additional information is required,
		and it is provided in terms of the indicators $\mathbf{1}_{k_3=0}$
		and $\mathbf{1}_{k_3=1}$.
	}
\label{fig:bijProbs}
\end{figure}

\subsection{Bijectivisation of the Yang--Baxter equation}
\label{sub:bijectivisation_of_YBE}

Here we apply the 
general definition of the bijectivisation 
from the previous \Cref{sub:bijectivisation}
to the Yang--Baxter equation of 
\Cref{prop:YBE}.

Denote $I=\{i_1,i_2,i_3 \}$
and $J=\{ j_1,j_2,j_3 \}$. The 
combination of $I$ and $J$
encodes a choice of the boundary conditions
in \eqref{eq:YBE}.
For each such choice, let
$\mathcal{A}_{I, J}$ and $\mathcal{B}_{I,J}$
index the nonzero summands in the left,
respectively, the right-hand side of
\eqref{eq:YBE}.
One can check that for any $I,J$, we have 
one of the following three possibilities:
\begin{itemize}
	\item Either the sums in both sides 
		of \eqref{eq:YBE} are empty
		For 
		example, this happens when
		$i_1+i_2+i_3\ne j_1+j_2+j_3$.
		Call this the \emph{incompatible} case. 
	\item Or one or both of
		$|\mathcal{A}_{I, J}|$
		and
		$|\mathcal{B}_{I, J}|$ is equal to $1$.
		Call this the \emph{one-to-two} case.
	\item 
		We have
		$|\mathcal{A}_{I, J}| = |\mathcal{B}_{I, J}| =2$,
		but the Yang--Baxter equation 
		\eqref{eq:YBE}
		looks as 
		$\mathbf{w}(a_1)+\mathbf{w}(a_2) = \mathbf{w}(b_1)+\mathbf{w}(b_2)$,
		where $\mathbf{w}(a_1)=\mathbf{w}(b_1)$ 
		and $\mathbf{w}(a_2)=\mathbf{w}(b_2)$.
		In other words, we can always match each of the two terms in the 
		left-hand side to the corresponding term in the right-hand
		side which has the same weight.
		Call this the \emph{two-to-two} case.
\end{itemize}

For the one-to-two case, there is 
at most one nonnegative bijectivisation 
of the Yang--Baxter equation
(by 
\Cref{rmk:bij_is_unique}).
In the two-to-two case, we make the natural choice and 
assign the bijectivisation to be deterministic. 
That is, in the notation of the two-to-two case, 
set 
$\mathbf{p}^{\mathrm{fwd}}(a_1\to b_1)
=
\mathbf{p}^{\mathrm{fwd}}(a_2\to b_2)
=1$,
and all the other forward probabilities to zero
(and similarly to the backward probabilities).
In this way, 
for any choice of the boundary conditions
$I,J$,
we have outlined a single
solution to the linear equations
\eqref{eq:bij_sum_to_one}--\eqref{eq:bij_reversibility}.
Denote this solution 
by
\begin{equation}
	\label{eq:transition_probabilities_bijectivisation}
	\mathbf{p}_{u,v}^{\mathrm{fwd}}\left( (k_1,k_2,k_3)\to (k_1',k_2',k_3')\mid I,J \right),
	\qquad 
	\mathbf{p}_{u,v}^{\mathrm{bwd}}\left( (k_1',k_2',k_3')\to (k_1,k_2,k_3)\mid I,J \right)
\end{equation}
(when $I$ and $J$ are incompatible, by agreement, we set all these probabilities
equal to $0$).
Here $u,v$ is the spectral parameter of the lower (resp. the upper) vertex
in the left-hand side of the Yang--Baxter equation
\eqref{eq:YBE}.

Denote
\begin{equation}
	\label{eq:YBE_bij_alpha_beta_gamma_parameters}
	\boldsymbol\upalpha \coloneqq
	\frac{(1-q) \, (1- q u)\, v}{(1 - q v)\, (v - q u)}
	,
	\qquad 
	\boldsymbol\upbeta \coloneqq
	\frac{(1-q)\,  (1- v)\, u}{(1 - u)\, (v - q u)}
	,\qquad 
	\boldsymbol\upgamma \coloneqq
	\frac{(1-v)\, (1- q u)}{(1 - u)\, (1 - q v)}
	.
\end{equation}

\begin{proposition}
	\label{prop:YBE_bij_probabilities}
	For all $I,J$, the solutions \eqref{eq:transition_probabilities_bijectivisation}
	are expressed through the quantities 
	$\boldsymbol\upalpha, \boldsymbol\upbeta, \boldsymbol\upgamma$
	and their complementaries
	$1-\boldsymbol\upalpha, 1-\boldsymbol\upbeta, 1-\boldsymbol\upgamma$.
	The forward solutions $\mathbf{p}_{u,v}^{\mathrm{fwd}}$
	are given in the table in \Cref{fig:bijProbs}. 
	When $0<u<v<1$,
	the forward and the backward solutions 
	are all nonnegative.
\end{proposition}
\begin{proof}
	The form of the forward solutions
	in \Cref{fig:bijProbs}
	is verified in a straightforward way
	for each choice of $I,J$.
	The backward solutions are readily found from 
	the reversibility condition \eqref{eq:bij_reversibility},
	but we do not need their explicit form 
	in the present work. 

	Finally, conditions 
	$0<u<v<1$ together with $0<q<1$ 
	ensure that 
	$\boldsymbol\upalpha,\boldsymbol\upbeta,\boldsymbol\upgamma$
	belong to $(0,1)$, which guarantees the nonnegativity of the 
	forward probabilities in \Cref{fig:bijProbs}. 
	The fact that the backward probabilities
	are then also nonnegative follows from the reversibility.
\end{proof}

\subsection{Two-row Markov operator}
\label{sub:two_row_Markov_operator}

Take the two-row lattice $\mathbb{Z}_{\ge0}\times \left\{ 1,2 \right\}$
and consider a path configuration 
on it encoded by the triple
$(\kappa, \mu, \lambda)$ of finite subsets of $\mathbb{Z}_{\ge0}$.
At the bottom boundary, paths enter according to $\kappa$,
at the top boundary they exit according to $\lambda$,
and $\mu$ encodes the occupation of internal vertical edges. 
We assume that no paths proceed infinitely far to the right,
and that
on the left the boundary edges 
$(-1,1)-(0,1)$ and
$(-1,2)-(0,2)$ are both occupied or are both empty. 
That is, either $\ell(\lambda)=\ell(\mu)=\ell(\kappa)$,
or 
$\ell(\lambda)=\ell(\mu)+1=\ell(\kappa)+2$.
See \Cref{fig:triple_configuration} for an illustration.

\begin{figure}[htpb]
	\centering
	\includegraphics[width=.5\textwidth]{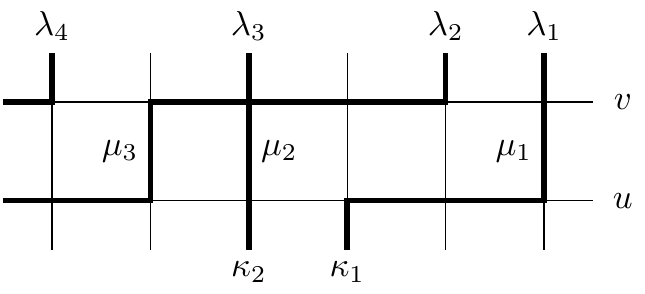}
	\caption{A path configuration 
		in $\mathbb{Z}_{\ge0}\times \left\{ 1,2 \right\}$
		encoded by 
		$\lambda=(5,4,2,0)$, $\mu=(5,2,1)$, and $\kappa=(3,2)$.
		In this example both left boundary edges are occupied.}
	\label{fig:triple_configuration}
\end{figure}

\begin{definition}[Two-row bijectivisation]
	\label{def:two_row_bij}
	Take a triple $(\kappa,\mu,\lambda)$ as described above
	with 
	$\ell(\lambda)=\ell(\mu)=\ell(\kappa)$
	or 
	$\ell(\lambda)=\ell(\mu)+1=\ell(\kappa)+2$.
	Let the spectral parameters $u,v$
	with $0<u<v<1$ be associated to the bottom and,
	respectively, the top row of the lattice 
	$\mathbb{Z}_{\ge0}\times \left\{ 1,2 \right\}$.
	We sample a new configuration 
	$(\kappa, \nu, \lambda)$, where $\nu$ is random and depends on the 
	old triple $(\kappa,\mu,\lambda)$, as follows:
	\begin{itemize}
	\item First, add an extra cross vertex 
		at the left boundary of the lattice between rows $1$ and $2$.
		Depending on the left boundary, the cross is fully empty or 
		fully occupied. In both cases its weight is equal to $1$,
		see \eqref{eq:X_vertex} and \eqref{eq:w_u_weights}.

	\item For each $i = 0, 1, 2, \dots$, 
		drag the cross past the $i^{th}$ vertical edge in the
		lattice.
		This is a random operation 
		involving the transition probabilities
		$\mathbf{p}^{\mathrm{fwd}}_{u,v}$ from \Cref{fig:bijProbs}. 
		Since $\kappa,\mu,\lambda$ are finite subsets,
		this random procedure eventually becomes deterministic far to the right
		because all edges are eventually empty.
		This corresponds to 
		the forward transition probabilities becoming trivial:
		$
		\mathbf{p}_{u,v}^{\mathrm{fwd}}\left( (0,0,0)\to (0,0,0)\mid \left\{ 0,0,0 \right\},\left\{ 0,0,0 \right\} \right)=1
		$.
		The resulting state of the cross vertex is empty, 
		and it has weight $1$.
	\end{itemize}
	Denote the law of the resulting 
	random $\nu$ by
	$U_{u,v}(\mu \rightarrow \nu \mid \kappa, \lambda)$.
	We call
	$U_{u,v}$ the \emph{two-row Markov transition 
	operator} which takes $\mu$ to a random $\nu$
	given the boundary conditions encoded by~$\kappa,\lambda$.
\end{definition}

\subsection{Action on two-row Gibbs measures}
\label{sub:tworow_action_Gibbs}

Within the setting of the previous \Cref{sub:two_row_Markov_operator},
let us show how
the operator $U_{u,v}$ (\Cref{def:two_row_bij})
swaps the spectral parameters $u \leftrightarrow v$ in Gibbs measures
on path configurations in the two-row lattice
$\mathbb{Z}_{\ge0}\times \left\{ 1,2 \right\}$.

Assume that the old triple 
$(\kappa,\mu,\lambda)$ 
with $\ell(\lambda)=\ell(\mu)+1=\ell(\kappa)+2$
has the $(u,v)$-Gibbs 
distribution. This means that the conditional distribution
of $\mu$ is 
\begin{equation*}
	\mathbb{P}_{u,v}(\mu\mid \kappa,\lambda)=\frac{1}{Z}
	\langle e_\lambda, \mathsf{B}_v e_\mu \rangle \,
	\langle e_\mu, \mathsf{B}_u e_\kappa \rangle,
\end{equation*}
where $Z=Z(u,v)$ is the normalizing constant
(the other case $\ell(\lambda)=\ell(\mu)=\ell(\kappa)$ is similar,
but the operators $\mathsf{B}$ should be replaced with $\mathsf{A}$).
The Yang--Baxter equation (\Cref{prop:YBE})
implies that $Z(u,v)$ is symmetric in $u,v$.
Indeed, after summing over $\mu$ to get $Z(u,v)$,
we see that the symmetry follows from 
the commutation of the operators
$\mathsf{B}_u$ and $\mathsf{B}_v$,
see \eqref{eq:comm_BD_1}--\eqref{eq:comm_BD_2}.
The Markov operator $U_{u,v}$ extends this symmetry to the 
level of probability distributions:

\begin{proposition}
	\label{prop:U_swap_uv_tworow}
	Let $0 < u < v < 1$. For any triple 
	$(\kappa,\mu,\lambda)$
	of finite
	subsets of $\mathbb{Z}_{\ge0}$, we have
	\begin{equation}
		\label{eq:U_swap_uv_tworow}
		\sum_{\tilde \mu}
		\langle e_\lambda, \mathsf{B}_v e_{\tilde\mu} \rangle 
		\langle e_{\tilde \mu}, \mathsf{B}_u e_\kappa \rangle
		\,
		U_{u,v}(\tilde \mu\to \mu\mid \kappa,\lambda)
		=
		\langle e_\lambda, \mathsf{B}_u e_\mu \rangle 
		\langle e_\mu, \mathsf{B}_v e_\kappa \rangle
	\end{equation}
	The same identity also
	holds with the operators $\mathsf{B}$ replaced everywhere by $\mathsf{A}$.
\end{proposition}
\begin{proof}
	We prove the statement with $\mathsf{B}$, the
	one with $\mathsf{A}$ is analogous.
	The action of $U_{u,v}$ starts by 
	adding the full cross (having weight $1$)
	on the left. Then it proceeds by 
	dragging the cross 
	through the lattice, and 
	sampling $\mu$ given $\kappa,\tilde \mu,\lambda$
	by a sequence of $\mathbf{p}_{u,v}^{\mathrm{fwd}}$ steps.
	Thus, we can
	write the 
	left-hand side of
	\eqref{eq:U_swap_uv_tworow}
	as the product of local vertex weights 
	and the local transition probabilities
	$\mathbf{p}_{u,v}^{\mathrm{fwd}}$. 
	The dragging of the cross and 
	the summation over $\tilde\mu$ in the left-hand side
	of \eqref{eq:U_swap_uv_tworow}
	allows to repeatedly apply
	the local relation
	\begin{equation*}
		\begin{split}
		&
		\sum_{k_1,k_2,k_3}
		\underbrace{X_{u,v}(i_2,i_1;k_2,k_1)
		w_u(i_3,k_1;k_3,j_1)\,
		w_v(k_3,k_2;j_3,j_2)}_{\textnormal{from the LHS 
		of the Yang--Baxter equation \eqref{eq:YBE}}}
		\mathbf{p}_{u,v}^{\mathrm{fwd}}\left( (k_1,k_2,k_3)\to (k_1',k_2',k_3')
		\mid I,J \right)
		\\
		&
		=
		\sum_{k_1,k_2,k_3}
		\underbrace{w_v(i_3,i_2;k_3',k_2')\,
		w_u(k_3',i_1;j_3,k_1')
		\,
		X_{u,v}(k_2',k_1';j_2,j_1)}
		_{\textnormal{from the RHS 
		of the Yang--Baxter equation \eqref{eq:YBE}}}
		\mathbf{p}_{u,v}^{\mathrm{bwd}}\left( (k_1',k_2',k_3')\to (k_1,k_2,k_3)\mid I,J \right)
		\\&=
		w_v(i_3,i_2;k_3',k_2')\,
		w_u(k_3',i_1;j_3,k_1')
		\,
		X_{u,v}(k_2',k_1';j_2,j_1),
		\end{split}
	\end{equation*}
	which follows from the 
	properties \eqref{eq:bij_sum_to_one}--\eqref{eq:bij_reversibility} of the bijectivisation
	(recall that $I=\left\{ i_1,i_2,i_3 \right\}$,
	$J=\left\{ j_1,j_2,j_3 \right\}$).
	The term 
	$w_v(i_3,i_2;k_3',k_2')\,
	w_u(k_3',i_1;j_3,k_1')$
	in the 
	right-hand side of the local relation
	signifies the appearance of the local $(v,u)$-Gibbs distribution.
	The additional
	cross vertex weight 
	$X_{u,v}(k_2',k_1';j_2,j_1)$ participates in the next local relation
	in the process of dragging the cross.

	Ultimately, one arrives at the
	configuration $(\kappa,\mu,\lambda)$
	whose weight has swapped spectral parameters.
	The eventual state of the 
	cross vertex is empty because it is to the right of the rightmost path.
	Since the weight of an empty cross is $1$, we can
	remove the cross vertex weight, 
	and we are left with the right hand side of the desired
	identity \eqref{eq:U_swap_uv_tworow}.
\end{proof}

\subsection{Action on infinite configurations}
\label{sub:truncation}

In the previous \Cref{sub:two_row_Markov_operator,sub:tworow_action_Gibbs}
we defined the two-row Markov operator
$U_{u,v}(\tilde \mu\to \mu\mid \kappa,\lambda)$ acting on finite 
configurations of vertical arrows. 
Here we extend this operator to possibly infinite 
subsets of $\mathbb{Z}_{\ge0}$.
Recall that for $\upmu\subseteq\mathbb{Z}_{\ge0}$
and $h\ge1$, the truncation is 
$\upmu^{[<h]}=\upmu\cap\left\{ 0,1,\ldots,h-1  \right\}$.

Pick
$\upkappa,\uplambda\subseteq\mathbb{Z}_{\ge0}$,
and assume that $\upmu\subseteq\mathbb{Z}_{\ge0}$
has a $(u,v)$-Gibbs distribution on the two-row
lattice $\mathbb{Z}_{\ge0}\times \left\{ 1,2 \right\}$.
That is, for any $h\ge1$ and any choice of the occupations
of the horizontal edges
$(h-1,1)-(h,1)$ and $(h-1,2)-(h,2)$ on the right boundary,
the conditional distribution of 
$\upmu^{[<h]}$ 
is given by $\mathbb{P}_{(u,v)}^{\mathrm{rect}}$
with the corresponding boundary conditions.

For $h\ge1$, denote by
$U_{u,v}^{[<h]}(\upmu^{[<h]}\to \upnu^{[<h]}\mid \upkappa^{[<h]},\uplambda^{[<h]})$ 
the probability that 
by randomly dragging the cross
to the right past the horizontal coordinate $h-1$
(using the probabilities $\mathbf{p}^{\mathrm{fwd}}_{u,v}$ from
\Cref{fig:bijProbs}),
the state of the internal edges at $0,1,\ldots,h-1 $
is given by $\upnu^{[<h]}$.
By the same computation as in the proof of \Cref{prop:U_swap_uv_tworow},
the distribution of 
$\upnu^{[<h]}$ is $(v,u)$-Gibbs. 
Moreover, as $h\to+\infty$, 
the distributions of $\upnu^{[<h]}$ are compatible.
Thus, by the Kolmogorov extension, we arrive at a 
random subset $\upnu\subseteq\mathbb{Z}_{\ge0}$
which has a $(v,u)$-Gibbs distribution.

\medskip

Let us now take the stochastic six vertex model
in the quadrant as in \Cref{def:s6v_quadrant},
with the
spectral parameters
$\mathbf{u}=(u_1,u_2,\ldots )$
and step-$\uplambda$
or empty-$\uplambda$ boundary conditions.
The full configuration of the up-right paths in the quadrant
can be encoded by 
a sequence 
$\uplambda^{(i)}$, $i\ge0$,
of subsets 
of $\mathbb{Z}_{\ge0}$,
where $\uplambda^{(0)}=\uplambda$ (the bottom boundary condition),
and each $\uplambda^{(i)}$ represents which 
vertical edges among
$\{(h,i)-(h,i+1)\colon h\in \mathbb{Z}_{\ge0}\}$
are occupied by paths.

\begin{definition}
	\label{def:L_k_operator}
	Fix $k\ge1$ and assume that $u_k<u_{k+1}$.
	Let 
	$L_{k,\mathbf{u}}$
	be the Markov transition operator acting 
	on the sequence $(\uplambda^{(i)})_{i\ge0}$
	by randomly changing $\uplambda^{(k)}$ to $\upnu^{(k)}$ according
	to the two-row transition probability
	$U_{u_{k},u_{k+1}}\left( \uplambda^{(k)}\to\upnu^{(k)}\mid 
	\uplambda^{(k-1)},\uplambda^{(k+1)} \right)$.
	For all $i\ne k$, the operator $L_{k,\mathbf{u}}$
	leaves $\uplambda^{(i)}$ intact.
\end{definition}

\begin{proposition}
	\label{prop:action_of_L_k_on_s6v}
	We have
	\begin{equation}
		\label{eq:action_of_L_k_on_s6v}
		\mathbb{P}_{\mathbf{u}}^{\mathrm{s6v}}
		L_{k,\mathbf{u}}=
		\mathbb{P}_{s_k\mathbf{u}}^{\mathrm{s6v}},
		\qquad 
		s_k\mathbf{u}\coloneqq(u_1,\ldots,u_{k-1},u_{k+1},u_k,u_{k+2},\ldots  ),
	\end{equation}
	where 
	$\mathbb{P}_{\mathbf{u}}^{\mathrm{s6v}}
	L_{k,\mathbf{u}}$
	is the action of a Markov operator on a probability measure,
	$s_k$ is the $k$-th elementary permutation, and 
	the boundary conditions of both
	stochastic six vertex models in the quadrant
	in \eqref{eq:action_of_L_k_on_s6v} are the same.
\end{proposition}
\begin{proof}
	Let 
	$\upkappa=\uplambda^{(k-1)}$,
	$\upmu=\uplambda^{(k)}$,
	$\uprho=\uplambda^{(k+1)}$,
	and $\upnu = \upnu^{(k)}$.
	The subsets 
	$\upkappa,\upmu,\uprho$ encode the configuration of 
	the stochastic six vertex model
	before the application of $L_{k,\mathbf{u}}$, 
	and $\upnu$ is the result of this application.

	Under the step-$\uplambda$ boundary conditions,
	for any truncation $h\ge1$
	the conditional
	distribution
	of $\upmu^{[<h]}$ given $\upkappa^{[<h]},\uprho^{[<h]}$
	is
	\begin{equation*}
		\mathbb{P}_{(u,v)}^{\mathrm{s6v}}
		(
		\upmu^{[<h]}\mid \upkappa^{[<h]},\uprho^{[<h]}
		)
		=
		\frac{1}{Z}
		\langle 
		e_{\uprho^{[<h]}},
		\left( \mathsf{B}_v+\mathsf{D}_v \right)
		e_{\upmu^{[<h]}}
		\rangle \,
		\langle 
		e_{\upmu^{[<h]}},
		\left( \mathsf{B}_u+\mathsf{D}_u \right)
		e_{\upkappa^{[<h]}}
		\rangle.
	\end{equation*}
	For the empty-$\uplambda$ boundary conditions,
	both operators $\mathsf{B}+\mathsf{D}$ should be replaced by
	$\mathsf{A}+\mathsf{C}$ with the same spectral parameters,
	and the argument is analogous.
	Applying the operator $U_{u,v}^{[<h]}$
	which maps $\upmu^{[<h]}$ to a random subset
	$\upnu^{[<h]}$,
	and arguing as in the proof of \Cref{prop:U_swap_uv_tworow},
	we see that the distribution of $\upnu^{[<h]}$
	is 
	$\mathbb{P}_{(v,u)}^{\mathrm{s6v}}
	(
	\upnu^{[<h]}\mid \upkappa^{[<h]},\uprho^{[<h]}
	)$.
	Indeed, 
	this follows from the commutation
	of $\mathsf{B}_u+\mathsf{D}_u$ with
	$\mathsf{B}_v+\mathsf{D}_v$,
	which in turn is a consequence of the Yang--Baxter
	equation, see \eqref{eq:comm_BD_1}--\eqref{eq:comm_BD_2}. 
	This completes the proof.
\end{proof}

\subsection{Two-row Markov operator via coin flips}
\label{sub:sample}

Here we present another description 
of the Markov operator $U_{u,v}$
which is adapted to taking the continuous
time limit in \Cref{sec:continuous_time_limit} below.
Fix spectral parameters $0<u<v<1$, and recall that we are dragging a cross
through the two-row lattice $\mathbb{Z}_{\geq 0} \times \{1,2\}$, with
spectral parameters $u$ and $v$ on rows $1$ and $2$,
respectively.

Let us argue in the setting of finite subsets. 
Assume that $\kappa,\mu,\lambda$ with either $\ell(\lambda)=\ell(\mu)+1=\ell(\kappa)+2$
or $\ell(\lambda)=\ell(\mu)=\ell(\kappa)$
encode the configurations of occupied vertical edges as in \Cref{fig:triple_configuration}.
We aim to present an algorithm sampling $\nu$
from the distribution
$U_{u, v}(\mu \rightarrow \nu \mid \kappa, \lambda)$. 
This algorithm involves independent random coin flips to
\emph{initiate jumps} at each horizontal coordinate $i$. If a jump is
initiated, then it propagates to the right according to
certain rules. A jump can be either up or down, let us describe them.

Denote the horizontal edges 
$(i,1)-(i+1,1)$ and $(i,2)-(i+1,2)$
by $e_{i,1}$ and $e_{i,2}$, respectively.
The occupation of these edges changes after initiating a jump
(and the vertical edge's occupation changes accordingly).
Note that the occupation of the previous horizontal edges
$e_{i-1,1},e_{i-1,2}$
does not change after a jump at $i$ is initiated.

\paragraph{Initiating up jump.}
An up jump can be initiated at $i$
if the local path configuration 
is one of 
\begin{equation}
	\label{eq:initiate_up_jump}
	\pcu,\pau,\pbu
	\;\;.
\end{equation}
After initiating an up jump, the path occupying
$e_{i,1}$ jumps up from $e_{i,1}$ to $e_{i, 2}$, and
the vertical edge $(i,1)-(i,2)$ becomes occupied. For
example, in the first case in \eqref{eq:initiate_up_jump} we have
\begin{equation*}
	\pcu \rightarrow \pcid \;\;.
\end{equation*}

\paragraph{Propagation of up jump.}
The up jump
propagates as follows.
Find the largest $c = c(i) =
c(i, \kappa,\lambda,\mu) \geq 0$ such that $e_{i+j,1}$ is
occupied and $e_{i+j,2}$ is unoccupied for all $j
=0,1,\dots, c$, and for all these edges we swap the
occupation status of $e_{i+j,1}$ and  $e_{i+j,2}$. We also
change the occupation status of $(i+c+1,1)-(i+c+1,2)$
from occupied to unoccupied. This results in a well
defined configuration. See \Cref{fig:jump_prop} for an illustration.

\paragraph{Initiating down jump.}
A down jump can be initiated at $i$ if the local configuration is one of
$$\pcd,\pad,\pbd \;\;.$$ 
Then the path occupying $e_{i,2}$ jumps down from $e_{i,2}$ to $e_{i,1}$,
and the vertical edge $(i,1)-(i,2)$ becomes unoccupied.
For example, 
$$\pbd \rightarrow \pau \;\;.$$

\paragraph{Propagation of down jump.}

The down jump propagates as follows. Find the largest $ c =
 c(i) =  c(i, \kappa,\lambda,\mu) \geq 0$ such that
$e_{i+j,1}$ is unoccupied and $e_{i+j,2}$ is occupied for
all $j =0,1,\dots,  c$. For all these edges, swap the
occupation status of $e_{i+j,1}$ and $e_{i+j,2}$. 
Change the state of the last vertical edge 
$(i+ c+1,1)-(i+ c+1,2)$ from unoccupied to occupied.

 \paragraph{Sampling procedure.}

Having defined initiation and propagation 
of the up and down jumps, we are now in a position to describe how the 
random finite subset $\nu$ with distribution $U_{u,v}$
is sampled.
Start by setting $i = 0$. 
While $i \leq \lambda_1$ (where $\lambda_1$ is the position of the rightmost 
occupied vertical edge), sequentially repeat the following steps:
\begin{enumerate}
\item If we can initiate an up jump at $i$, do 
	so according to an independent coin flip with probability
	$1-\boldsymbol\upgamma,1-\boldsymbol\upalpha$, or 
	$1-\boldsymbol\upbeta$ (see \eqref{eq:YBE_bij_alpha_beta_gamma_parameters})
	depending on the local configuration as in the table in
	\Cref{fig:Coin_flips}, left.

\item 
If we can initiate a down jump at $i$, do so
according to an independent coin flip with probability
$1-\boldsymbol\upgamma,1-\boldsymbol\upalpha$, or 
$1-\boldsymbol\upbeta$ (see \eqref{eq:YBE_bij_alpha_beta_gamma_parameters})
depending on the local configuration as in the table in
\Cref{fig:Coin_flips}, right.

\item If we initiated an up or down jump, let $c(i)$ be as described above, and propagate the jump to the configuration at positions $i+1, \dots, i+c(i)$. Then set $i = i+c(i)+1$.

\item If we cannot initiate
	any jump at $i$, set
	$i = i + 1$.
\end{enumerate}

\begin{figure}[htpb]
	\centering
	\begin{tabular}{ | m{6.5em} | m{2cm}| } 
\hline
Up jump & Probability \\ 
  \hline
	$\pcu \rightarrow \pcid$		& $1-\boldsymbol\upgamma$ \\[23pt]
  \hline
$\pau \rightarrow \pbd$		& $1-\boldsymbol\upalpha$  \\[23pt]
  \hline
$\pbu \rightarrow \pad$		& $1-\boldsymbol\upbeta$ \\[23pt]
  \hline
\end{tabular}
\qquad 
\qquad 
\begin{tabular}{ | m{6.5em} | m{2cm}|  } 
\hline
Down jump & Probability \\ 
  \hline
$\pcd \rightarrow \pciu$		& $1-\boldsymbol\upgamma$ \\[23pt] 
  \hline
$\pad \rightarrow \pbu$		& $1-\boldsymbol\upalpha$  \\[23pt] 
  \hline
$\pbd \rightarrow \pau$		& $1-\boldsymbol\upbeta$ \\[23pt] 
  \hline
\end{tabular}
\caption{Coin flip probabilities for initiating a jump expressed
	through the quantities defined in \eqref{eq:YBE_bij_alpha_beta_gamma_parameters}.
Note that the probabilities for initiating up and down jumps 
are the same under the inversion of occupation of all edges.}
	\label{fig:Coin_flips}
\end{figure}

\begin{proposition}
	\label{prop:coin_flips_work_s3}
	The algorithm for sampling random $\nu$ given
	$\kappa,\mu,\lambda$
	indeed produces $\nu$ with the distribution
	$U_{u,v}(\mu\to\nu\mid \kappa,\lambda)$.
\end{proposition}
\begin{proof}
	From the table of the forward transition 
	probabilities 
	$\mathbf{p}_{u,v}^{\mathrm{fwd}}$ in 
	\Cref{fig:bijProbs} we see that the probability to change the 
	occupation state of a vertical edge by dragging the cross
	is equal to $1-\boldsymbol\upalpha,1-\boldsymbol\upbeta$,
	or $1-\boldsymbol\upgamma$ depending on the state of the horizontal and cross
	edges around. Indeed, the state of the cross before and after the 
	dragging is determined uniquely by the initiated jump.
	The jump is not initiated with the complementary probability
	$\boldsymbol\upalpha,\boldsymbol\upbeta$, or $\boldsymbol\upgamma$, respectively, 
	which corresponds to another state of the cross vertex after the dragging.
	Next, the rules for propagation of up or down jumps come from the 
	parts of the table in 
	\Cref{fig:bijProbs} where 
	$\mathbf{p}_{u,v}^{\mathrm{fwd}}=1$. This completes the proof.
\end{proof}

	\begin{figure}[h]
\centering
	\includegraphics[width=.7\textwidth]{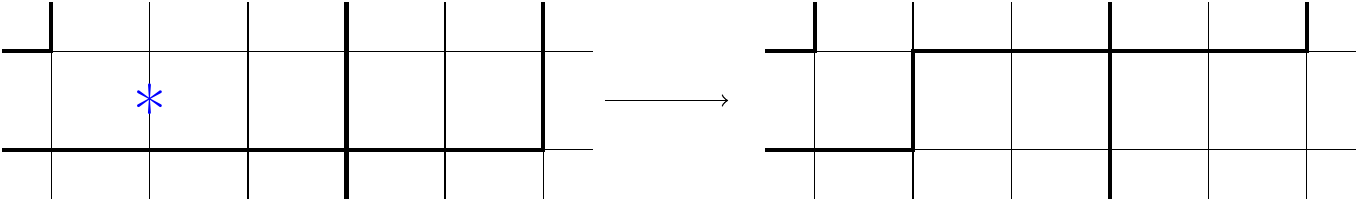}
\caption{A generic up jump with propagation. In this case an
up jump was initiated at the vertical edge with an asterisk.
Note that the propagation can continue past vertical paths
which go straight up.}
\label{fig:jump_prop}
\end{figure}

Let us extend \Cref{prop:action_of_L_k_on_s6v}
to infinite subsets
$\upkappa,\upmu,\uplambda$ 
(the setup of \Cref{sub:truncation}).
For each $h\ge1$, we may truncate the sampling
algorithm by allowing the initiation of jumps 
only at $i=0,1,\ldots,h-1 $.
Reading the resulting vertical path configuration after all the jumps, 
we clearly get $\upnu^{[<h]}$
from the distribution
$U_{u,v}^{[<h]}(\upmu^{[<h]}\to \upnu^{[<h]}\mid \upkappa^{[<h]},\uplambda^{[<h]})$.
For $h\to+\infty$, these truncations of the sampling algorithm
are consistent.
Therefore, \Cref{prop:action_of_L_k_on_s6v}
holds when the Markov operator $L_{k,\mathbf{u}}$
is defined as above in this subsection using jumps initiation and propagation.

\begin{remark}
	\label{rmk:U_uv_identity}
	When $u=v$, from \eqref{eq:YBE_bij_alpha_beta_gamma_parameters} we see
	that $\boldsymbol\upalpha=\boldsymbol\upbeta=\boldsymbol\upgamma=1$.
	This means that no up or down jumps can be initiated,
	so the Markov map $U_{u,u}$ is simply the identity operator.
\end{remark}

\section{Continuous time limit in the quadrant}
\label{sec:continuous_time_limit}

\subsection{Moving $u_1$ up to infinity}
\label{sub:u1_to_inf}

In this section we construct
a continuous time Markov chain
on the homogeneous 
stochastic six vertex model configurations in
the quadrant $\mathbb{Z}_{\ge0}\times \mathbb{Z}_{\ge1}$.
The continuous time Markov chain is a Poisson type limit of the 
discrete time Markov chain 
obtain by a repeated application of the operators $L_{k,\mathbf{u}}$
(\Cref{def:L_k_operator}),
as the inhomogeneous parameters $u_i$ become equal.
The Taylor expansion of the transition probabilities
around the identity (cf. \Cref{rmk:U_uv_identity})
leads to the desired continuous time dynamics.
This Poisson type limit is analogous to the one in
\cite{PetrovSaenz2019backTASEP}.
In particular, we also get space-inhomogeneous jump rates
linearly depending
on the $y$ coordinate.

Start with a sequence
$\mathbf{u}=(u_1,u_2,\ldots )$
of spectral parameters satisfying $0<u_1<u_2<\ldots <1$.
Let $(\uplambda^{(y)})_{y\ge0}$
be a sequence of subsets of $\mathbb{Z}_{\ge0}$
encoding a state of the stochastic six vertex
$\mathbb{P}^{\mathrm{s6v}}_{\mathbf{u}}$
model with step-$\uplambda$ or empty-$\uplambda$
boundary conditions (see \Cref{def:s6v_quadrant}). 
Here $\uplambda^{(0)}=\uplambda$ is the fixed
bottom boundary condition.

\begin{definition}
	\label{def:L_full_operator}
	Denote by $\mathbf{L}_{\mathbf{u}}$ the one-step Markov
	operator which is the
	result of the application of the infinite sequence of 
	Markov operators
	$L_{1,\mathbf{u}},L_{2,s_1\mathbf{u}}, L_{3,s_2s_1\mathbf{u}},\ldots $
	(in this order).
	Here each $s_k$ is the elementary permutation
	$(k,k+1)$, the first operator $L_{1,\mathbf{u}}$ involves
	the spectral parameters
	$u_1,u_2$, the next operator 
	$L_{2,s_1\mathbf{u}}$ involves 
	$u_1,u_3$, and so on.

	Let $(\tilde \uplambda^{(y)})_{y\ge0}$
	be the random sequence encoding the 
	result of the application of $\mathbf{L}_{\mathbf{u}}$
	to the sequence
	$(\uplambda^{(y)})_{y\ge0}$ we started with.
	The new random sequence 
	is well-defined
	since for any finite $y_0$, 
	the first several layers
	$(\tilde \uplambda^{(y)})_{0\le y\le y_0}$
	are obtained 
	from $(\uplambda^{(y)})_{y\ge0}$
	by finitely many Markov operators.
\end{definition}

\begin{figure}[htpb]
	\centering
	\includegraphics[width=.3\textwidth]{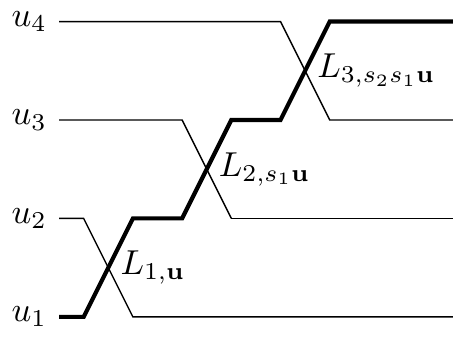}
	\caption{Moving the spectral parameter $u_1$ in the 
		stochastic six vertex model $\mathbb{P}^{\mathrm{s6v}}_{\mathbf{u}}$
		up to infinity
		by applying a sequence of two-row Markov operators
		$L_{k,s_{k-1}\ldots s_1 \mathbf{u} }$.}
	\label{fig:u1_inf}
\end{figure}

Denote by $S\mathbf{u}\coloneqq (u_2,u_3,u_4,\ldots )$
the one-sided shift of the sequence 
of the spectral parameters.

\begin{proposition}
	\label{prop:action_of_big_L}
	The operator $\mathbf{L}_{\mathbf{u}}$ acts on the measure
	$\mathbb{P}_{\mathbf{u}^{\mathrm{s6v}}}$ 
	as
	$\mathbb{P}_{\mathbf{u}}^{\mathrm{s6v}}
	\mathbf{L}_{\mathbf{u}}=
	\mathbb{P}_{S\mathbf{u}}^{\mathrm{s6v}}$.
\end{proposition}
See \Cref{fig:u1_inf} for an illustration.
\begin{proof}[Proof of \Cref{prop:action_of_big_L}]
	Immediately follows by iterating \Cref{prop:action_of_L_k_on_s6v}.
\end{proof}

\subsection{Poisson type limit and jump rates}
\label{sub:u_q_cont}

Here we employ the description of the Markov operator
$L_{k,\mathbf{u}}$ in terms of independent coin
flips (\Cref{prop:coin_flips_work_s3})
to obtain a Poisson type limit of the transition probabilities.

\begin{lemma}
	\label{lemma:limit_of_probabilities}
	The quantities $\boldsymbol\upalpha,\boldsymbol\upbeta,\boldsymbol\upgamma$
	\eqref{eq:YBE_bij_alpha_beta_gamma_parameters}	
	depending on the spectral parameters $0<u<v<1$
	admit the following expansions as $v\to u$:
	\begin{equation*}
		\begin{split}
			\boldsymbol\upalpha(u,v)&=
			1-\mathfrak{a}(u)\cdot (v-u)+O(v-u)^2,\\
			\boldsymbol\upbeta(u,v)&=
			1-\mathfrak{b}(u)\cdot (v-u)+O(v-u)^2,\\
			\boldsymbol\upgamma(u,v)&=
			1-\mathfrak{c}(u)\cdot (v-u)+O(v-u)^2,
		\end{split}
	\end{equation*}
	where
	\begin{equation}
		\label{eq:abc_rates}
			\mathfrak{a}(u) \coloneqq  
			\frac{(1 - u)\, q}{(1 - q)\,(1 - q u)\, u },\qquad
			\mathfrak{b}(u) \coloneqq  
			\frac{1 - q u}{(1 - u) \, (1 - q)\, u},\qquad
			\mathfrak{c}(u) \coloneqq 
			\frac{1-q}{(1-u)\, (1 - q u)}.
	\end{equation}
\end{lemma}
\begin{proof}
	Straightforward Taylor expansion.
\end{proof}

We utilize the expansion from
\Cref{lemma:limit_of_probabilities}
together with the iterated 
moving of the bottom spectral parameter up to infinity, as 
defined in the previous \Cref{sub:u1_to_inf}.
In this subsection we outline the main expansions, and in the next
\Cref{sub:u_q_cont}
we define the generator of the continuous time dynamics, 
and show the existence of the dynamics.

Define
\begin{equation*}
	\mathbf{Q}_m=\mathbf{L}_{\mathbf{u}}\mathbf{L}_{S\mathbf{u}}
	\mathbf{L}_{S^2\mathbf{u}}\ldots\mathbf{L}_{S^{m-1}\mathbf{u}},
	\qquad m=1,2,\ldots .
\end{equation*}
(The order of the composition of Markov operators means that 
$\mathbf{L}_{\mathbf{u}}$ is applied first.)
By \Cref{prop:action_of_big_L}, we have
$\mathbb{P}_{\mathbf{u}}^{\mathrm{s6v}}
\mathbf{Q}_m=
\mathbb{P}_{S^m\mathbf{u}}^{\mathrm{s6v}}$.

Let us take spectral parameters close to each other.
Fix real parameters $u$ and $\eta>0$ such that $u,u+\eta\in (0,1)$.
Let $\varepsilon>0$ be sufficiently small, and define
\begin{equation}
	\label{eq:u_i_for_limit}
	u_i \coloneqq u+(1-e^{-i \varepsilon})\,\eta,\qquad i=1,2,\ldots. 
\end{equation}
Denote by $\tau\in \mathbb{R}_{\ge0}$ the rescaled time.
Then we have
\begin{equation*}
	\mathbb{P}_{\mathbf{u}}^{\mathrm{s6v}}
	\mathbf{Q}_{\lfloor \tau/\varepsilon \rfloor }=
	\mathbb{P}^{\mathrm{s6v}}_{\mathbf{u}[\tau]},\qquad 
	\mathbf{u}[\tau]_i \coloneqq u+\bigl( 1-e^{-\varepsilon\left( i+\lfloor \tau/\varepsilon \rfloor  \right)} \bigr)\,\eta,\quad i=1,2,\ldots .
\end{equation*}

As $\varepsilon\to0$, the parameters $u_i$ \eqref{eq:u_i_for_limit}
become all equal to $u$, and
$\mathbf{u}[\tau]_i$ become all equal to 
$u+(1-e^{-\tau})\ssp\eta$.
The difference of the spectral parameters
$\mathbf{u}[\tau]_1$ and 
$\mathbf{u}[\tau]_{k+1}$
(which are exchanged at horizontal layer $k$ 
at the step $\mathbf{L}_{S^{\lfloor \tau/\varepsilon \rfloor -1}\mathbf{u}}$
in
the chain
$\mathbf{Q}_{\lfloor \tau/\varepsilon \rfloor }$)
has the form
\begin{equation}
	\label{eq:difference_spectral_parameters_Poisson_limit}
	\mathbf{u}[\tau]_{k+1}
	-
	\mathbf{u}[\tau]_1
	=
	\eta \ssp e^{-\tau}
	\bigl( 1- e^{-k\varepsilon } \bigr)
	+O(\varepsilon^2)
	=
	k\ssp \varepsilon \ssp \eta \ssp e^{-\tau}
	+O(\varepsilon^2)
	,\qquad \varepsilon\to0.
\end{equation}
Therefore, the $\varepsilon\to0$ limit of the 
Markov transition operators
$\mathbf{Q}_{\lfloor \tau/\varepsilon \rfloor }$
should lead to a continuous
time Markov chain with the transition semigroup
$(\mathcal{Q}(\tau))_{\tau\in \mathbb{R}_{\ge0}}$, which acts on the
homogeneous stochastic six vertex model
in the quadrant as
$\label{eq:P_act}
\mathbb{P}^{\mathrm{s6v,\,hom}}_{u}
\mathcal{Q}(\tau)=
\mathbb{P}^{\mathrm{s6v,\,hom}}_{u+(1-e^{-\tau})\eta}$.
We see that from $\tau=0$ to $\tau=+\infty$,
the chain $\mathcal{Q}(\tau)$ continuously
increases the spectral parameter
$u$ to $u+\eta$.
The definition of $\mathcal{Q}(\tau)$
employs the probabilities 
$\mathfrak{a}(u),\mathfrak{b}(u),\mathfrak{c}(u)$
\eqref{eq:abc_rates},
and is given in the next \Cref{sub:q_cont_def}.

\subsection{Continuous time chain in the quadrant}
\label{sub:q_cont_def}

Let us define the continuous time Markov semigroup
$\mathcal{Q}(\tau)$ in terms of its generator
$G_{u,\eta,\tau}^{\mathrm{quad}}$.
The generator depends on $u,\eta$,
and also on the time variable $\tau$. The latter means that
the continuous time Markov chain is time-inhomogeneous.
First, recall a basic definition:
\begin{definition}
	[Time-inhomogeneous Poisson process]
	\label{def:inhom_Poisson}
	A random locally finite 
	point configuration $(\tau_1<\tau_2<\ldots )\subset\mathbb{R}_{>0}$
	is said to be distributed as an inhomogeneous Poisson process
	with bounded rate function $r(\tau)>0$ iff 
	\begin{itemize}
		\item The number of points $\tau_i$ in each interval $[s,t]\subset\mathbb{R}_{>0}$
			is a Poisson distributed random variable with mean
			$\int_s^t r(\tau)d\tau$;
		\item For finitely or countably many 
			disjoint intervals $[s_i,t_i]\subset \mathbb{R}_{>0}$,
			the numbers of random points in them are independent
			random variables.
	\end{itemize}
	For short, we say
	that the arrivals $\tau_i$ in the Poisson process
	occur according to an
	\emph{exponential clock with time-dependent
	rate} $r(\tau)$.
\end{definition}

\begin{definition}
	\label{def:seed_pair}
	Let $v_1,v_2$ be a pair of vertices at vertically adjacent positions $(j,k), (j,k+1)$
	in the quadrant. If the local configuration
	of the paths around $v_1,v_2$
	is one of the six configurations in \Cref{fig:Coin_flips},
	we call the pair $(v_1,v_2)$ a \emph{seed pair}.
\end{definition}

We attach to each seed pair an independent exponential clock with 
the time-dependent rate
\begin{equation}
	\label{eq:rate_tau_dependent}
	k\ssp \eta \ssp e^{-\tau}\mathfrak{R}_{v_1,v_2}(u+(1-e^{-\tau})\ssp \eta),
\end{equation}
where $k$ is the $y$-coordinate of $v_1$, and
$\mathfrak{R}_{v_1,v_2}(u)$ is given in 
\Cref{fig:rates_for_G_quad}.
The rate
\eqref{eq:rate_tau_dependent}
is the coefficient by $\varepsilon$
in the expansion of $1-\boldsymbol\upalpha,1-\boldsymbol\upbeta$, or $1-\boldsymbol\upgamma$ as
in \Cref{lemma:limit_of_probabilities}, 
where we took into account
the inhomogeneity 
coming from exchanging the spectral parameters, see
\eqref{eq:difference_spectral_parameters_Poisson_limit}.

When the clock at a seed pair $(v_1,v_2)$ rings, 
this generates an up or down jump of the horizontal path,
as illustrated in \Cref{fig:Coin_flips}.
This jump then instantaneously (at the same time 
moment, without any waiting)
propagates to the right according to the rules 
given in \Cref{sub:sample}.

\begin{figure}[htpb]
	\centering
	\begin{tabular}{ | m{5em} | m{1cm}| } 
\hline
Up jump & Rate \\ 
	\hline
	$\pcu $		& $\mathfrak{c}(u)$ \\[23pt]
	\hline
	$\pau $		& $\mathfrak{a}(u)$  \\[23pt]
		\hline
		$\pbu $		& $\mathfrak{b}(u)$ \\[23pt]
		\hline
	\end{tabular}
	\qquad 
	\qquad 
	\begin{tabular}{ | m{5em} | m{1cm}|  } 
	\hline
	Down jump & Rate \\ 
		\hline
		$\pcd $		& $\mathfrak{c}(u)$ \\[23pt] 
		\hline
		$\pad $		& $\mathfrak{a}(u)$  \\[23pt] 
		\hline
	$\pbd $		& $\mathfrak{b}(u)$ \\[23pt] 
		\hline
	\end{tabular}
		\caption{Jump rates $\mathfrak{R}_{v_1,v_2}(u)$ 
			in the generator $G^{\mathrm{quad}}_{u,\eta,\tau}$,
			where
			$\mathfrak{a}(u),\mathfrak{b}(u),\mathfrak{c}(u)$ are
			given in \eqref{eq:abc_rates}.}
	\label{fig:rates_for_G_quad}
\end{figure}

Thus defined jumps lead to the following infinitesimal
generator of a time-inhomogeneous continuous
time Markov chain:

\begin{definition}[Generator]
	\label{def:generator}
	Let $\sigma$ denote a configuration of up-right
	paths in the quadrant.
	If $(v_1,v_2)$ is a seed pair for $\sigma$,
	denote by $\sigma_{(v_1,v_2)}$ the result of 
	initiating an (up or down) 
	jump of the horizontal path 
	at $(v_1,v_2)$, and the propagation of this jump
	according to the rules in \Cref{sub:sample}.
	Let
	$f(\sigma)$ be a cylindric function.
	That is, $f$ depends on $\sigma$ only through the restriction of $\sigma$ to a 
	finite window inside the quadrant (this window depends on $f$).
	The generator of the dynamics on the stochastic six vertex model in the quadrant
	is, by definition, the operator acting as
	\begin{equation}\label{eq:quad_generator}
		(G^{\mathrm{quad}}_{u,\eta,\tau}f)(\sigma)=
		\eta \ssp e^{-\tau}
		\sum_{\textnormal{$(v_1,v_2)$ is a seed pair}}
		y(v_1)\ssp
		\mathfrak{R}_{v_1,v_2}
		\bigl(u+(1-e^{-\tau})\eta\bigr)
		\left( f(\sigma_{(v_1,v_2)})-f(\sigma) \right).
	\end{equation}
	Here the sum is over all seed pairs $(v_1,v_2)$ of $\sigma$.
	While the number of seed pairs
	may be infinite, the
	action of $G^{\mathrm{quad}}_{u,\eta,\tau}$ \eqref{eq:quad_generator} 
	is well-defined on cylindric functions.
\end{definition}

We aim to define a Markov semigroup $\mathcal{Q}(\tau)$
with generator \eqref{eq:quad_generator} which can
start from configurations belonging to a certain space of
\emph{regular initial configurations} $\Omega^{\mathrm{quad}}$.
This makes sure that $\mathcal{Q}(\tau)$ does not
make infinitely many jumps through a finite space in finite time.
Let us define the space of initial configurations, and 
then prove that the semigroup $\mathcal{Q}(\tau)$ and the corresponding Markov process
$\sigma(\tau)$ exists.

\begin{definition}
	\label{def:quad_initial_configurations}
	Let $\Omega^{\mathrm{quad}}$ be the
	set of up-right path configurations $\sigma$ in the quadrant
	with the following condition:
	For each $R'>0$ there exists
	an $R = R(R')$ such that the configuration in the region
	$[0, R'] \times [R,\infty)$ either is fully empty
	(each vertex has state $(0,0;0,0)$)
	or is fully packed 
	(each vertex has state $(1,1;1,1)$).
		
\end{definition}

\begin{lemma}
	\label{lemma:regular_bc_in_quadrant_has_probability_1}
	For the homogeneous stochastic six vertex
	model with step-$\uplambda$ or 
	empty-$\uplambda$ boundary conditions (\Cref{def:s6v_quadrant})
	we have
	$\mathbb{P}_{u}^{\mathrm{s6v,\,hom}}(\Omega^{\mathrm{quad}})=1$
	for any $0<u<1$.
\end{lemma}
\begin{proof}
	In the step-$\uplambda$ case,
	the configuration in the region
	$[0, R'] \times [R,\infty)$ is stochastically 
	monotone in $\uplambda$. That is, when
	one adds an extra occupied initial edge to $\uplambda$,
	the probability that the configuration
	in $[0, R'] \times [R,\infty)$
	is fully packed does not decrease. This follows by considering 
	the stochastic six vertex weights 
	(\Cref{fig:6types}), 
	and observing that 
	for fixed $j_1$,
	the probability that $i_2=1$ increases when $i_1=1$.
	When $\uplambda=\varnothing$, the probability 
	of $\Omega^{\mathrm{quad}}$ is equal to $1$
	thanks to the Law of Large Numbers established in \cite{BCG6V}.
	Indeed, the latter states that the bottom boundary of the fully occupied region
	in the quadrant is linear, see \Cref{fig:s6v_sim}, left, for a simulation.
	
	\begin{figure}[htpb]
		\centering
		\includegraphics[width=.4\textwidth]{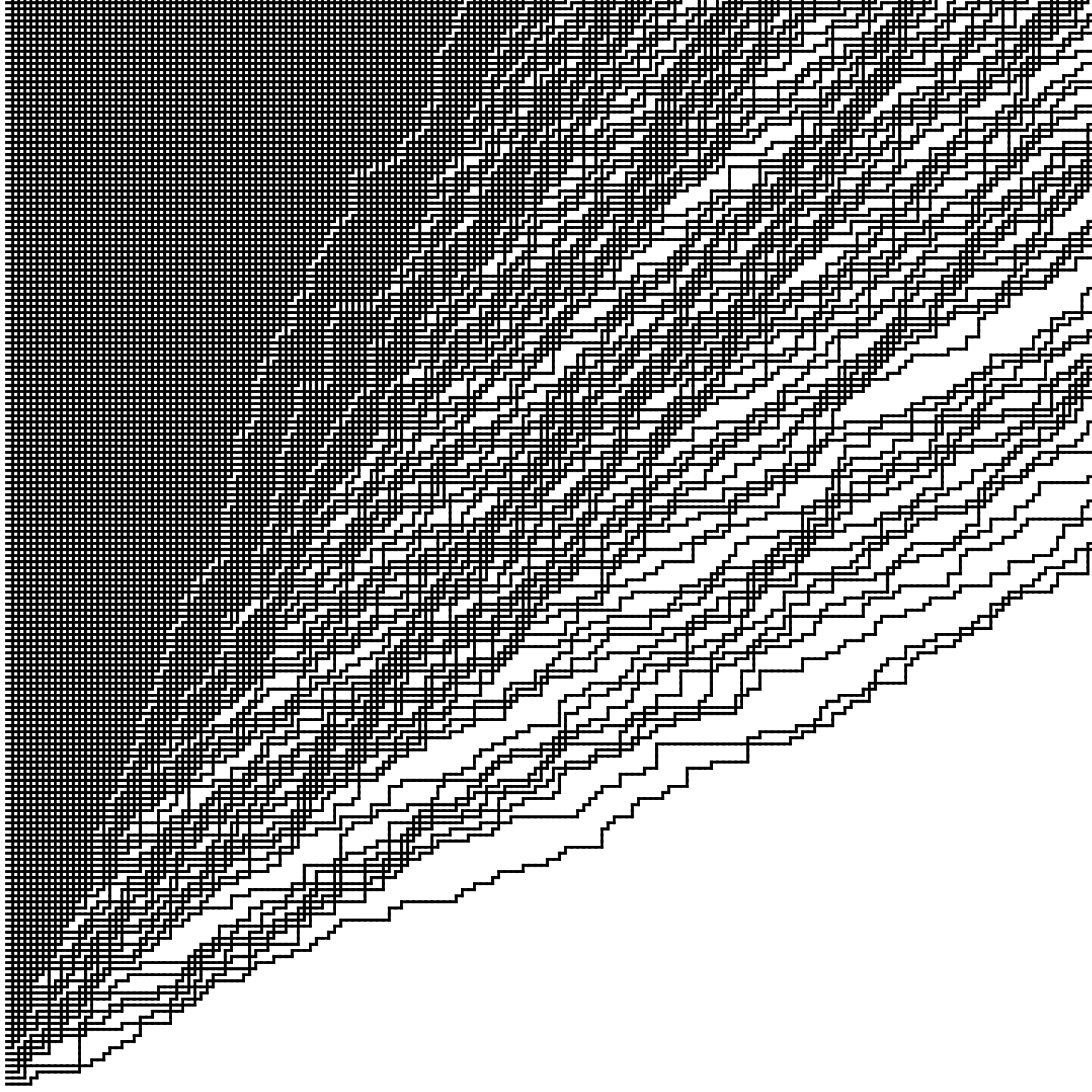}
		\qquad 
		\includegraphics[width=.4\textwidth]{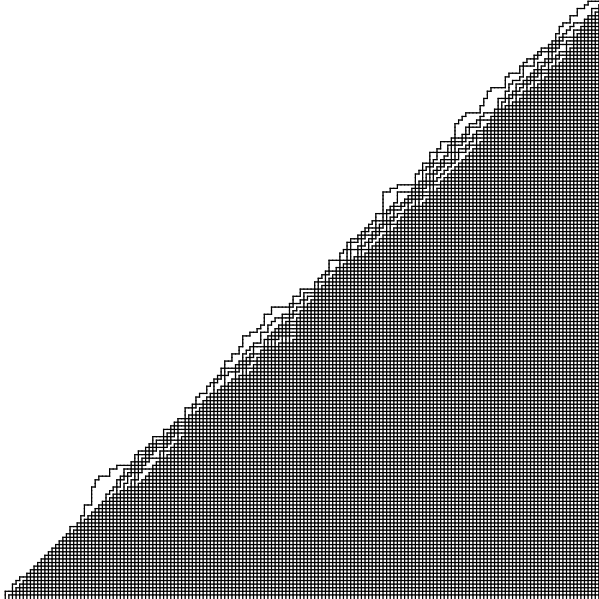}
		\caption{A simulation of the stochastic six vertex
			model in the quadrant with step-$\varnothing$ (left)
			and empty-$\mathbb{Z}_{\ge0}$ (right)
		boundary conditions.}
		\label{fig:s6v_sim}
	\end{figure}

	Similarly, for empty-$\uplambda$ boundary conditions, 
	the probability that
	$[0, R'] \times [R,\infty)$
	is empty of paths does not increase 
	when adding extra occupied edges to $\uplambda$.
	When $\uplambda=\mathbb{Z}_{\ge0}$, 
	the 
	probability 
	of $\Omega^{\mathrm{quad}}$ is also equal to $1$.
	Indeed, since $\delta_1$
	(the probability of going up)
	is smaller than $\delta_2$,
	the whole region which is slightly
	above the diagonal of the quadrant 
	is empty, see \Cref{fig:s6v_sim}, right, for a simulation.
\end{proof}

\begin{theorem}
	\label{thm:quad_process_exists}
	Assume that
	$\sigma(0)\in \Omega^{\mathrm{quad}}$.
	There exists a continuous time 
	Markov chain $\sigma(\tau)$, $\tau\in \mathbb{R}_{\ge0}$,
	on configurations of 
	up-right paths 
	whose generator at time $\tau$ is 
	$G^{\mathrm{quad}}_{u,\eta,\tau}$
	given by \eqref{eq:quad_generator}
	(here we use the convention around 
	time-dependent rates, cf. \Cref{def:inhom_Poisson}).
	Moreover, the transition operator $\mathcal{Q}(\tau)$ of this Markov chain
	acts on the homogeneous six vertex model 
	(with step-$\uplambda$ or empty-$\uplambda$ boundary
	conditions for arbitrary fixed $\uplambda\subset \mathbb{Z}_{\ge0}$)
	as follows:
	\begin{equation}
		\label{eq:quad_Q_action_on_s6v_homogeneous}
		\mathbb{P}_u^{\mathrm{s6v,\,hom}}\ssp
		\mathcal{Q}(\tau)=
		\mathbb{P}_{u+(1-e^{-\tau})\ssp\eta}^{\mathrm{s6v,\,hom}}.
	\end{equation}
\end{theorem}

\begin{figure}[h]
	\centering
	\includegraphics[width=.35\textwidth]{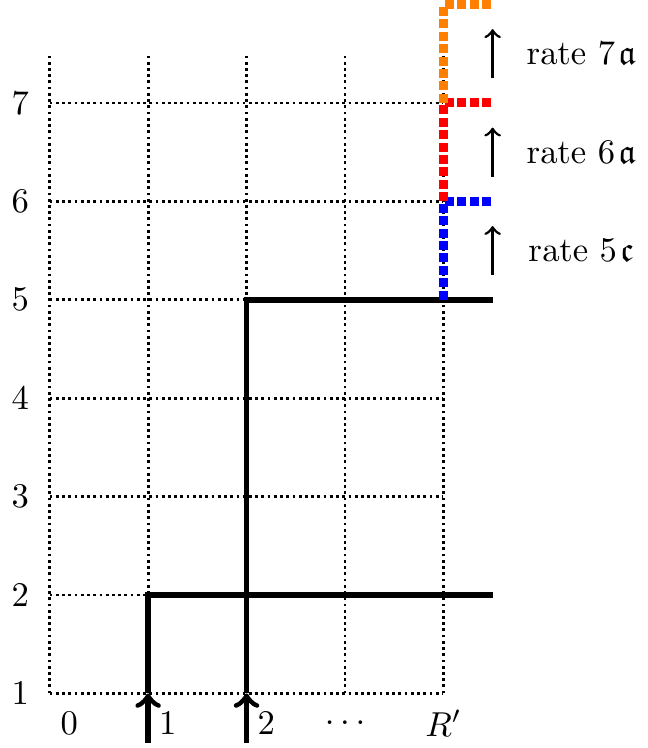}
	\caption{The top path crossing the right boundary of the 
	vertical strip jumps up three times.}
	\label{fig:seed_arrivals_quadrant}
\end{figure}

\begin{proof}
	Recall the truncation of the bijectivisation Markov
	operator $U$ to $U^{[<h]}$, see \Cref{sub:truncation}.
	Since the jump rates in our generator
	$G^{\mathrm{quad}}_{u,\eta,\tau}$
	are Poisson limits of the ones in the operators $U$,
	the generator is compatible with these truncations, too. 
	That is, 
	$G^{\mathrm{quad}}_{u,\eta,\tau}$
	preserves the space of cylindric functions
	$f(\sigma)$
	which depend on $\sigma$ only through the restriction of $\sigma$ to 
	a finite vertical strip
	$\left\{ 0,1,\ldots,R'  \right\}\times \mathbb{Z}_{\ge1}$ (where $R'>0$ is fixed).
	In other words, due to the very construction of the jump rates in 
	$G^{\mathrm{quad}}_{u,\eta,\tau}$,
	the dependence does not propagate
	from $\left\{ R'+1,R'+2,\ldots  \right\}\times \mathbb{Z}_{\ge1}$
	into
	$\left\{ 0,1,\ldots,R'  \right\}\times \mathbb{Z}_{\ge1}$.
	Therefore, it suffices to show existence of the 
	process $\sigma(\tau)$ restricted to an arbitrary
	vertical strip $\left\{ 0,1,\ldots,R'  \right\}\times \mathbb{Z}_{\ge1}$, and the full quadrant process
	then arises by Kolmogorov extension.

	Consider the following collection of independent 
	time-inhomogeneous Poisson processes:
	\begin{equation*}
		\{W_{j,k}^\mathfrak{l}\colon \mathfrak{l} =\mathfrak{a},\mathfrak{b},\mathfrak{c};\, j \geq 0, k \geq 1\}.
	\end{equation*}
	Here $(j,k)$ are the lattice coordinates, and 
	the Poisson process 
	$W_{j,k}^\mathfrak{l}=W_{j,k}^\mathfrak{l}(\tau)$ has time-dependent
	rate $k\ssp \eta e^{-\tau}\ssp \mathfrak{l}\left( u+(1-e^{-\tau})\ssp \eta \right)$,
	where $\mathfrak{l}$ is one of the letters $\mathfrak{a},\mathfrak{b}$, or $\mathfrak{c}$,
	see \eqref{eq:abc_rates}. 
	We think of the Poisson process $W_{j,k}^\mathfrak{l}$
	as attached to the vertical edge
	$(k,j)-(k,j+1)$.

	Fix an initial condition $\sigma(0)\in \Omega^{\mathrm{quad}}$, 
	and $R',R=R(R')>0$ such that
	$\sigma(0)$ is empty in the region
	$[0,R']\times[R,\infty)$
	(see \Cref{def:quad_initial_configurations}).
	The case when $\sigma(0)$ is full in 
	$[0,R']\times[R,\infty)$
	is treated analogously.

	We can define
	the evolution $\sigma(\tau)$ 
	restricted to 
	$\left\{ 0,1,\ldots,R'  \right\}\times \mathbb{Z}_{\ge1}$
	as a function of all the Poisson processes
	$W_{j,k}^\mathfrak{l}$,
	in the spirit of the Harris
	graphical construction \cite{Harris1978}.
	That is, for each seed pair $(v_1,v_2)$, if there is an
	arrival in one of the Poisson processes attached to the edge $v_1-v_2$, then
	it triggers the corresponding jump of the horizontal
	path up or down, which propagates to the right.
	We call such an arrival in a Poisson process the \emph{seed arrival}.
	
	To complete the construction of the process in
	$\left\{ 0,1,\ldots,R'  \right\}\times \mathbb{Z}_{\ge1}$,
	we need to show that 
	in any finite time interval, there are
	almost surely finitely many seed arrivals
	in the Poisson processes
	$W_{j,k}^\mathfrak{l}$,
	where $0\le j\le R'$, $k\ge1$.
	Note that this is not obvious since the jump 
	rates depend linearly on the vertical coordinate $k$,
	and hence are unbounded.

	Since the configuration 
	in $[0,R']\times[R,\infty)$ is initially empty,
	infinitely many 
	seed arrivals might arise only 
	when up-right paths 
	perform infinitely many up jumps.
	More precisely, the top of the paths 
	crossing the vertical line with $x$-coordinate $R'$
	must jump up infinitely many times.
	See \Cref{fig:seed_arrivals_quadrant}
	for an illustration.

	The process of up jumps of the top path
	is bounded from above by the 
	pure birth process with $\mathrm{rate}(k\to k+1)= C k$, for some 
	constant $C>0$. (This process is also called the 
	Yule process.) It is well-known (e.g., \cite{KMG57BDClassif}) 
	that this pure birth process does not 
	run off to infinity in finite time
	because the sum of its inverse rates diverges:
	$\sum_{k=1}^{\infty}(Ck)^{-1}=+\infty$.
	This implies that the desired process
	$\sigma(\tau)$ restricted
	to 
	$\left\{ 0,1,\ldots,R'  \right\}\times \mathbb{Z}_{\ge1}$
	does not perform infinitely many jumps in finite time,
	and hence completes the construction of the 
	dynamics $\sigma(\tau)$ (with the Markov generator $\mathcal{Q}(\tau)$) in the quadrant.

	Formula 
	\eqref{eq:quad_Q_action_on_s6v_homogeneous}
	for the
	action of $\mathcal{Q}(\tau)$
	on the homogeneous stochastic six vertex model
	follows as a Poisson limit
	of \Cref{prop:action_of_big_L},
	as explained  
	in \Cref{sub:u_q_cont}.
\end{proof}

\section{Markov process preserving full plane Gibbs measures}
\label{sec:Markov_full_plane_process}

\subsection{Bulk limit of the quadrant dynamics. Heuristics}
\label{sub:bulk_limit_heuristics}

In this section we discuss the full plane continuous time 
dynamics arising in the bulk of the 
process $\mathcal{Q}(\tau)$ constructed in \Cref{sec:continuous_time_limit}
above. 
Consider running $\mathcal{Q}(\tau)$ with an initial
configuration sampled from the stochastic six vertex model
$\mathbb{P}_{u}^{\mathrm{s6v,\,hom}}$
with, say, step-$\varnothing$ boundary conditions. 
Let $\varepsilon > 0$
be small, and consider a rectangular
part of the lattice 
around a point with coordinates
$(\lfloor x/\varepsilon \rfloor,\lfloor y/\varepsilon\rfloor)$.
Assume that the limit
$\varepsilon\to0$ preserves the lattice scale, that is, 
the rectangular part of the lattice 
turns into the full plane $\mathbb{Z}^2$ as $\varepsilon\to0$.

The local statistics of
the path configuration 
around
$(\lfloor x/\varepsilon \rfloor,\lfloor y/\varepsilon\rfloor)$
are described 
\cite{aggarwal2020limit}
by the pure state of a slope
$(\mathsf{s},\mathsf{t})$ belonging to either the KPZ phase 
or the frozen phase, in the terminology from
\Cref{sub:pure_states}. See also \Cref{fig:s6v_sim},
left,
for a simulation. In the rest of the discussion we ignore the frozen part
and focus on the KPZ phase.

Slowing down $\mathcal{Q}(\tau)$ so that it runs at speed $\varepsilon$,
the transition rates (for initiating jumps)
around 
$(\lfloor x/\varepsilon \rfloor,\lfloor y/\varepsilon\rfloor)$
in a finite time interval
$[0,\tau_0]$
have the form 
\begin{equation*}
	\varepsilon\ssp 
	\eta
	\left\lfloor \frac{y}{\varepsilon} \right\rfloor 
	\mathfrak{R}_{v_1,v_2}\left( u+O(\varepsilon) \right)
	\left( 1+O(\varepsilon) \right)
	=\eta \ssp y \ssp
	\mathfrak{R}_{v_1,v_2}(u)+O(\varepsilon),
\end{equation*}
Here $O(\varepsilon)$ may depend on $\tau_0$, but $\tau_0$ is fixed.
Since $y$ is also fixed, we see that as $\varepsilon\to0$,
the 
jump rates converge to finite values proportional to 
$\mathfrak{R}_{v_1,v_2}(u)$. In other words, 
we arrive at a (so far, hypothetical)
full plane continuous time Markov chain
with \emph{homogeneous} rates
$\mathrm{const}\cdot\mathfrak{R}_{v_1,v_2}(u)$.
Throughout this section we denote the full plane chain by $\mathcal{C}(t)$,
$t\in \mathbb{R}_{\ge0}$.

Therefore, if there is reasonable \emph{locality} in the original process
$\mathcal{Q}(\tau)$ in the quadrant
(more precisely, if we can turn off jumps outside of
a large enough box around 
$(\lfloor x/\varepsilon \rfloor,\lfloor y/\varepsilon\rfloor)$
without affecting the process in a smaller box),
then 
in the bulk limit regime
the mapping of the measures \eqref{eq:quad_Q_action_on_s6v_homogeneous}
turns into the statement that $\mathcal{C}(t)$
should \emph{preserve} the local distribution,
the pure Gibbs state with slope $(\mathsf{s},\mathsf{t})$ in the KPZ phase.

In \Cref{sec:Markov_full_plane_process,sec:proof_lemma_S6}
we prove the existence of the process
$\mathcal{C}(t)$ together with the preservation of the KPZ pure states.
We obtain $\mathcal{C}(t)$ directly from the 
jump rates, and not as a bulk limit of the 
dynamics in the quadrant. However, we employ the dynamics in the quadrant
to show that the full plane process preserves the KPZ pure Gibbs state.

\subsection{Admissible configurations of up-right paths}
\label{sub:full_plane_process}

We aim to construct 
a continuous time Markov process
$\mathcal{C}(t)$
on up-right path configurations
in the full plane $\mathbb{Z}^2$ with the following
dynamics. Recall that to each 
path configuration
we associate its seed pairs $(v_1,v_2)$ of vertically adjacent
vertices, see
\Cref{def:seed_pair}.

\begin{definition}[Jumps in the process $\mathcal{C}(t)$]
	\label{def:C_t_jump_rates}
	Each seed pair $(v_1,v_2)$ has 
	an exponential clock with rate $\mathfrak{R}_{v_1,v_2}(u)$,
	see \Cref{fig:rates_for_G_quad}.
	When the clock at $(v_1,v_2)$ rings,
	the horizontal path passing through this seed pair
	jumps up or down. This jump then instantaneously 
	propagates to the right until it finds a vertical edge where it can stop
	(see \Cref{fig:jump_prop} for an 
	illustration, and \Cref{sub:sample}
	for the definition of jump propagation).
	Note that in contrast with the dynamics in the quadrant, 
	in the full plane case all rates are time-independent and are homogeneous
	throughout the plane.
\end{definition}

Since the jumps may potentially propagate very far to the right, it is 
not immediately clear how to define the process
$\mathcal{C}(t)$ 
with the jump rates
from \Cref{def:C_t_jump_rates}
even locally. Indeed, a new jump can be 
initiated anywhere along a very long 
horizontal path, and close to its right end there could be 
infinitely many propagated jumps in finite time.
In other words, defining the 
generator of this process formally as
\begin{equation}\label{eq:full_plane_generator}
	(G_{u}f)(\sigma)=
	\sum_{\textnormal{$(v_1,v_2)$ is a seed pair}}
	\mathfrak{R}_{v_1,v_2}(u)
	\left( f(\sigma_{(v_1,v_2)})-f(\sigma) \right),
\end{equation}
similarly to \eqref{eq:quad_generator},
where the sum is over all seed pairs $(v_1,v_2)$ of $\sigma$,
may lead to divergence for some configurations $\sigma$.
Therefore, let us first define the space of admissible 
configurations of up-right paths:

\begin{definition}
	\label{def:admissible_configurations}
	Fix real $R,A>0$.
	Let 
	$\Lambda_R \coloneqq [-R,R]^2 \cap \mathbb{Z}^2$
	be the finite square with side $2R$ around the origin.
	Let $E_{R,A}$ be the 
	event (i.e., a subspace of up-right path configurations)
	such that there exists a horizontal sequence
	of vertices $(x_0, y), (x_0+1,y), \dots, (x_0 + n, y) =:
	(x_1,y)$ where the path configuration around $(x_0 + j, y)$ is not
	equal to $\ru$
	for any $0 \leq j \leq n$, where $n > A$ and
	$(x_1,y) \in \Lambda_R$ (here $\ru$ is a shorthand for the vertex
	$(0,1;1,0)$).
	In words, configurations in $E_{R,A}$ have
	horizontal strings of occupied edges of length $>A$ in an $R$-neighborhood of the origin.
	Clearly, for $B>A$ we have $E_{R,B}\subset E_{R,A}$.
	
	Define the set of \emph{admissible configurations} to be
	\begin{equation}
		\label{eq:Omega_admissible_set}
		\Omega \coloneqq 
		\bigcup_{p \in \mathbb{Z}_{\geq 1}}
		\bigcup_{R_0 \in \mathbb{Z}_{\ge1} } \bigcap_{R \geq R_0,\, R\in \mathbb{Z}}
		(E_{R, (\log R)^p})^c.
	\end{equation}
	In words, configurations in $\Omega$ are such that
	for some $p\ge1$ and all large enough $R$,
	every horizontal string of adjacent
	vertices of length larger than $(\log R)^p$ ending in
	$\Lambda_R$
	contains $\ru$.
\end{definition}

Next, let us show that under pure states in the 
KPZ phase, almost surely the configuration of up-right paths in admissible.
Fix the parameters of the model $q,u\in(0,1)$,
and the density of vertical occupied
edges $\mathsf{s}=\rho\in(0,1)$. 
Let $\pi(\rho)=\pi_{\rho,\varphi(\rho)}$
be the corresponding KPZ pure state (\Cref{sub:pure_states}).
Denote this measure by $\pi$, for short.

\begin{lemma}\label{lem:strbound}
Let $\zeta > 0$. There exists a constant 
$C_\zeta > 0$ such that for all $R > 0$,
\begin{equation*}
	\pi(E_{R,C_\zeta \log R}) \leq C_\zeta R^{-\zeta}.
\end{equation*}
\end{lemma}
\begin{proof}
	It suffices to show that for some constant $0 < \theta <
	1$, the probability under $\pi$ of a particular sequence
	of vertices $(x,y), (x+1,y), \dots, (x + n, y)$ not having
	a vertex equal to $\ru$ is upper bounded by $\theta^n$.
	Indeed, then we have by taking union bound:
	\begin{equation*}
		\pi(E_{R,C_\zeta \log R})\le 
		\mathrm{const}\cdot R^2 \theta^{C_\zeta \log R}= 
		\mathrm{const}\cdot R^{2-C_\zeta \log(1/\theta)},
	\end{equation*}
	which may be made less than $C_\zeta R^{-\zeta}$ by a choice of $C_\zeta$.

	Now, conditioning on the configuration up to the vertex
	$(x+i-1,y)$, we have two possibilities
	(recall the vertex weights in \Cref{fig:6types}):
	\begin{itemize}
	\item The horizontal edge exiting from $(x + i-1, y)$ is
		occupied. Conditioning on this, the probability of $\ru$
		at $(x +i, y)$ is at least $(1-\rho)(1-\delta_2)$. 
	\item The horizontal edge exiting $(x + i - 1, x+ i)$ is
		not occupied. Conditioning on this, the probability of
		$(x + i+1, y)$ being $\ru$ is at least 
		$\rho(1-\rho)
		(1-\delta_1)(1-\delta_2) $. 
	\end{itemize}
	Taking
	$\theta = 
	\max\left(1-\rho(1-\rho) (1-\delta_1)(1-\delta_2), 1-(1-\rho)(1-\delta_2) \right)$,
	leads to the upper bound 
	of the probability that no vertex in the sequence is $\ru$
	by $\theta^{\lfloor n/2 \rfloor }$.
	This completes the proof.
\end{proof}

\begin{proposition}
	\label{prop:admissible_are_almost_sure}
	We have $\pi(\Omega)=1$.
\end{proposition}
\begin{proof}
	Follows from \Cref{lem:strbound}
	by a Borel--Cantelli type
	argument. Indeed, taking $\zeta>1$ in \Cref{lem:strbound},
	we have
	\begin{equation*}
		\sum_{R=1}^{\infty} C_\zeta R^{-\zeta}<\infty
		\qquad 
		\Longrightarrow 
		\qquad 
		\pi
		\biggl( 
			\bigcap_{R_0\in \mathbb{Z}_{\ge1}} 
			\bigcup_{R\ge R_0,\,R\in \mathbb{Z}}
			E_{R,C_\zeta\log R}
		\biggr)=0.
	\end{equation*}
	Since $E_{R,C_\zeta \log R}\supset E_{R,(\log R)^2}$ for all $R$
	large enough, we see that 
	$\pi(\cup_{R_0}\cap_{R\ge R_0} E_{R,(\log R)^2})=0$.
	This implies that
	$\pi(\Omega^c)=0$ 
	as $\Omega^c$ is the intersection for all
	$p\ge1$, and the part with $p=2$ already has zero probability.
\end{proof}

\subsection{Formulations}

We construct the full plane dynamics $\mathcal{C}(t)$
as a limit of the truncated processes,
which exists for admissible configurations. 

\begin{definition}[Truncated process]
	\label{def:truncated_dynamics}
	Let $R>0$, and define $\mathcal{C}^R(t)$, $t\in \mathbb{R}_{\ge0}$,
	to be a continuous time Markov process on configurations
	of up-right paths with the jumps described in \Cref{def:C_t_jump_rates},
	with the modification that 
	new jumps can be initiated only by seed pairs 
	$(v_1,v_2)$ inside the finite square
	$\Lambda_R$.
\end{definition}

The truncated processes $\mathcal{C}^R$ are clearly well-defined.

\begin{remark}
	\label{rmk:coupling_of_truncated_processes}
	For a fixed initial configuration
	$\mathcal{C}^{R}(0)=\mathcal{C}_0$
	the distributions of 
	$\{\mathcal{C}^{R}\}_{t\in \mathbb{R}_{\ge0}}$
	can be coupled for various values of $R$. Indeed, a natural coupling 
	of $\mathcal{C}^{R}$
	and $\mathcal{C}^{R'}$, $R'>R$, 
	corresponds to using the same Poisson clocks for both $\mathcal{C}^{R}$
	and $\mathcal{C}^{R'}$ inside $\Lambda_R$, and turning on additional Poisson clocks
	for all seed pairs in $\Lambda_{R'}\setminus \Lambda_R$ for $\mathcal{C}^{R'}$. 
\end{remark}
We aim to show that when the initial configuration $\mathcal{C}_0$ 
belongs to $\Omega$ 
\eqref{eq:Omega_admissible_set},
then
the restrictions of the processes 
$\{\mathcal{C}^{R}(t)\}_{0\le t\le t_0}$ to $\Lambda_K$ for fixed
$K,t_0$ stabilize as $R\to+\infty$.
This would lead to the two main results 
of this section which we now formulate.

\begin{theorem}[Existence]
	\label{thm:C_t_full_plane_well_defined}
	For all $\mathcal{C}_0 \in \Omega$, there exists a Markov
	process $\{\mathcal{C}(t)\}_{t\in \mathbb{R}_{\ge0}}$
	started from $\mathcal{C}(0) =
	\mathcal{C}_0$, which evolves according to the jumps
	described in \Cref{def:C_t_jump_rates}.
	That is, 
	the action of the generator $(G_u f)(\sigma)$ \eqref{eq:full_plane_generator}
	is well-defined for $\sigma\in \Omega$,
	and corresponds to a continuous time Markov process.
	Moreover, for all $t$ we have $\mathcal{C}(t)\in \Omega$.
\end{theorem}

Recall that $\pi$
is the KPZ pure Gibbs state with density
$\rho$ of the occupied vertical edges,
where $0<\rho<1$.

\begin{theorem}[Preservation of KPZ pure states]
	\label{thm:C_t_preserves_Gibbs}
	The Markov process $\mathcal{C}(t)$ from
	\Cref{thm:C_t_full_plane_well_defined}
	preserves the measure $\pi$. 
	More precisely, for any cylindric function $f$ we have 
	\begin{equation}\label{eq:preservation_of_Gibbs}
		\int_{\Omega} 
		\mathbb{E}_{\mathcal{C}_0} [f(\mathcal{C}(t))] \ssp \pi(d \mathcal{C}_0) 
		=
		\pi [ f ],
	\end{equation}
	where
	$\mathbb{E}_{\mathcal{C}_0}$ is the expectation
	with respect to the Markov chain
	$\mathcal{C}(t)$ started from $\mathcal{C}_0$,
	and $\pi[f]$ is the integral of $f$ over the measure $\pi$.
\end{theorem}

Note that in \eqref{eq:preservation_of_Gibbs} we could integrate 
over the set of all up-right path configurations instead of admissible
configurations $\Omega$, but this is the 
same since $\pi(\Omega)=1$ by \Cref{prop:admissible_are_almost_sure}.

The proofs of \Cref{thm:C_t_full_plane_well_defined,thm:C_t_preserves_Gibbs}
occupy the rest of
\Cref{sec:Markov_full_plane_process}
and also \Cref{sec:proof_lemma_S6}.

\subsection{Proof of Theorem \ref{thm:C_t_full_plane_well_defined}}
\label{sub:proof_thm_57}

The following is the first and main lemma in our argument.
Recall the sets $E_{R,A}$ from \Cref{def:admissible_configurations}
in which there are long horizontal paths 
leading to far propagation of jumps.
Throughout the rest of the section we assume the following lemma.

\begin{lemma}[Main estimate]
	\label{lemma:main_full_plane_lemma}
Fix real numbers $\zeta > 0$, $p_0>0$, $C > 0$, $T > 0 $,
and $R > 0$. There is a constant $\tilde{C} =
\tilde{C}_{\zeta}$ such that the following holds for all
$\tilde{R}  \geq R > 0$. If we start the truncated dynamics
$\mathcal{C}^{\tilde{R}}(t)$ from a configuration
$\mathcal{C}_0 \notin E_{R, C (\log R)^{p_0}}$, then for
some $p \geq p_0$, 
\begin{equation*}
	\mathbb{P}
	\left(
		\exists\ssp t \in [0, T] \textnormal{\ such that\ } \mathcal{C}^{\tilde{R}}(t) \in E_{R, \tilde{C} (\log R)^p}
	\right) 
	\leq \tilde{C} R^{-\zeta}.
\end{equation*}
\end{lemma}

In words, if there is initially a bound on the length of jump propagation,
then at each finite time a slightly worse bound holds with 
high probability.
The proof of \Cref{lemma:main_full_plane_lemma} 
utilizes a nontrivial coupling and
is postponed till
the next \Cref{sec:proof_lemma_S6}.

\medskip

Given a path configuration $\mathcal{C}$, for $e$ an edge in
$\mathbb{Z}^2$, denote by $\updelta_e = \updelta_e(\mathcal{C})$
the path occupation indicator variable for that edge. 
For a trajectory of the truncated process $\mathcal{C}^R(t)$, denote the
corresponding edge indicator at time $t$ by
$\updelta^R_e(t)$.
The random variables
$\updelta^R_e(t)$ are naturally coupled for various values of $R$,
see \Cref{rmk:coupling_of_truncated_processes}.

\begin{lemma}\label{lemma:delta_comparison_Rlim}
Let $\zeta, T > 0$, and $e$ be a fixed edge
of $\mathbb{Z}^{2}$. 
Let the initial condition $\mathcal{C}_0$ belong to $\Omega$. 
Then there exists a
constant
${C} = {C}_\zeta > 0$ such that for any sufficiently large
$R', R$ with $R' > R$ we have
\begin{equation*}
	\mathbb{P}\left(\updelta^R_e(t)\neq \updelta^{R'}_e(t) \ \text{for some }
	t \in [0, T]\right) \leq {C} R^{-\zeta}.
\end{equation*}
\end{lemma}
In words, the probability that two
truncations diverge on a fixed edge is small.

\begin{proof}[Proof of \Cref{lemma:delta_comparison_Rlim}]
	We closely follow the proof of Proposition 7.6
	of \cite{Toninelli2015-Gibbs}, adapting it to our setting. 

	Suppose that $e$ is the vertical
	edge $((0,0),(0,1))$. This does not restrict the generality since
	a change in a horizontal edge is accompanied by a change of a vertical edge
	which is sufficiently close (since $\mathcal{C}_0\in \Omega$).

	Plug in the constant $p_0$ from 
	the fact that $\mathcal{C}_0\in \Omega$ and $C=1$
	into \Cref{lemma:main_full_plane_lemma}.
	Let $p,\tilde C$ be the constants from
	\Cref{lemma:main_full_plane_lemma}.
	Let $E_1$ denote the event that either
	$\mathcal{C}^{R}(t') \in E_{R, \tilde{C}_{\zeta} (\log
	R)^p}$ or $\mathcal{C}^{R'}(t') \in E_{R, \tilde{C}_{\zeta}
	(\log R)^p}$ for some time $t' \in [0, T]$.
	In words, under $E_1$ there exists a time $t'$ when
	in
	$\mathcal{C}^R(t')$ or $\mathcal{C}^{R'}(t')$
	we see a horizontal string of vertices of 
	length $\tilde{C}_{\zeta} (\log R)^p$, 
	none of which are $\ru$.
	By \Cref{lemma:main_full_plane_lemma},
	$\mathbb{P}(E_1)\le \tilde{C} R^{-\zeta}$,
	which means that in the rest of the 
	proof we may assume that $E_1$ did not happen.

	Suppose
	$\updelta^{R'}_e(t) \neq \updelta^R_e(t)$ for some $t \in [0, T]$ and $E_1$ did not happen.
	Denote this event by $E_2$.
	Let $t_1 \leq  T$
	be the first time at which $\updelta^{R'}_e(t_1^{+}) \neq
	\updelta^R_e(t_1^{+})$. (Throughout
	the proof, $t^{\pm}$ mean one-sided limits.)
	Then, there must have been a clock that rang
	at time $t_1$ which caused $e$ to change in 
	one but not the other of the configurations
	$\mathcal{C}^R$ and $\mathcal{C}^{R'}$.
	Note that in both processes
	$\mathcal{C}^R$ and $\mathcal{C}^{R'}$
	the jumps cannot propagate by more than 
	$\tilde{C}(\log R)^p$.
	Thus, 
	there must be an
	edge $e_1$ touching the rectangle $[-\tilde{C}(\log R)^p, 0]
	\times [0,1]$ such that $\updelta_{e_1}^{R'}(t_1^{-}) \neq
	\updelta_{e_1}^R(t_1^{-})$. Let $t_2 < t_1$ be the first time
	at which $\updelta_{e_1}^{R'}(t_2^+) \neq \updelta_{e_1}^R(t_2^+)$,
	and there must be another edge $e_2$ to the left of 
	$e_1$ such that 
	$\updelta_{e_2}^{R'}(t_2^{-}) \neq
	\updelta_{e_2}^R(t_2^{-})$,
	and the first time $t_3$ when the occupations
	of $e_2$ diverged in two processes.
	We continue this argument, and obtain a sequence
	of edges and times, which terminates 
	with an edge $e_n$ outside of $\Lambda_R$.
	We see that $\mathcal{C}^R$ could not ever change the state of 
	$e_n$, and it might have changed under $\mathcal{C}^{R'}$.

	To summarize, in $E_2$ 
	there exists a sequence of
	clocks that ring in $\Lambda_R$, at times $0 < t_n < t_{n-1}
	< \cdots < t_2 < t_1$ and positions $(x_1,y_1), (x_2,y_2),
	\dots, (x_n, y_n)$, such that $|y_i - y_{i-1}| \leq 1$ and
	$x_i - x_{i-1} \leq \tilde{C} (\log R)^p$. 
	This implies
	that $n \geq R / (\tilde{C} (\log R)^p)$.
	Let us bound the probability of $E_2$ from above.
	We use two more observations:
	\begin{itemize}
		\item 
			We choose $n$ locations in $[0,R]$
			where the clocks must ring such that the distance between
			two consecutive locations is $\le \tilde C (\log R)^p$.
			Therefore, each next location chooses from at most 
			$3C (\log R)^p$ available locations, and 
			thus the total number of ways to choose the locations is 
			upper bounded by
			$(3 \tilde{C} (\log R)^p)^n$.
		\item 
			There are at least $n$ clock rings in $[0,T]$,
			and we know that the rate of each clock ringing
			is bounded by a constant. Therefore, 
			the total number of clock rings during $[0,T]$ is 
			stochastically dominated by a Poisson random variable with 
			mean $\theta T$ for some $0<\theta < + \infty$. Therefore, 
			$\mathbb{P}( \mathrm{Poisson}(\theta T) \geq  n) \leq \mathrm{const}\cdot
			(e\ssp \theta T)^n/(n!)$.
	\end{itemize}
	Therefore, we have
	\begin{equation*}
		\mathbb{P}(E_2) \leq \sum_{n \geq R / (\tilde{C} (\log R)^p)} 
		( 3 \tilde{C} (\log R)^p )^n \cdot
		\mathbb{P}(\text{Poisson}(\theta T) \geq  n)
		\le \sum_{n \geq R / ( \tilde{C} (\log R)^p)} 
		\frac{(C')^n  (\log R)^{p n}}{n!}.
	\end{equation*}
	For fixed $R$,
	the sum over $n$ is, up to a constant, bounded from above by its first term
	$(C'')^{n_0} (\log R)^{p n_0}/(n_0!)$, where
	$n_0=R/(\tilde C (\log R)^p)$.
	As $R\to+\infty$, one readily checks that this first term
	goes to zero faster than any power of $R$, 
	and so is of order $o(R^{-\zeta})$.
	Combining the bounds on the probabilities of $E_1$ and $E_2$ yields the result.
\end{proof}

\begin{lemma}\label{lemma:limit_of_edge_variables_in_R_exists}
	Fix an admissible
	initial configuration $\mathcal{C}_0\in \Omega$,
	and an edge $e$. Recall that 
	$\updelta^R_e(t)$
	is the occupation of $e$ under the $R$-truncated process.
	With probability $1$, the limit
	\begin{equation}
		\label{eq:limit_of_truncations_of_edge_variables_in_R}
		\updelta_e(t) \coloneqq \lim_{R\rightarrow \infty} \updelta^R_e(t) 
	\end{equation}
	exists uniformly on compact intervals of $\mathbb{R}_{\geq 0}$.
\end{lemma}
Since there are countably many edges, the almost sure limit in
\eqref{eq:limit_of_truncations_of_edge_variables_in_R}
also exists simultaneously for all edges $e$.
\begin{proof}[Proof of \Cref{lemma:limit_of_edge_variables_in_R_exists}]
	Let $\Delta t>0$ be fixed.
	If 
	$\lim\limits_{R \rightarrow \infty}
	(\updelta^R_e(t), t \in [0, \Delta t])$ does not exist in
	the sense of uniform convergence,
	then for infinitely many positive integers
	$R$,
	there exists a time
	$t \in [0, \Delta t]$ for which $\updelta^R_e(t)\neq
	\updelta^{R+1}_e(t)$. However, by 
	\Cref{lemma:delta_comparison_Rlim}, for arbitrary $\zeta > 1$
	and a constant $\tilde{C} > 0$ depending on $\zeta$ but not on $R$, we have
	\begin{equation*}
		\sum_{R \geq 1} \mathbb{P}(\updelta^R_e(t)\neq
		\updelta^{R+1}_e(t), \ \text{for some } t 
		\in [0, \Delta
		t]) \leq \tilde{C} \sum_{R \geq 1} R^{-\zeta} < \infty .
	\end{equation*}
	By
	the 
	Borel--Cantelli lemma, 
	for any $\Delta t$, we get, on a probability $1$ event,
	the uniform convergence of the paths $\{ (\updelta^R_e(t),
	t \in [0, \Delta t]) \}$ as $R \rightarrow \infty$. Taking
	the intersection of these events as $\Delta t$ ranges over
	positive integers yields the result. 
\end{proof}

\Cref{lemma:main_full_plane_lemma} and the proof technique of \Cref{lemma:delta_comparison_Rlim} imply 
the following bound on the propagation speed:
\begin{lemma}
	\label{lemma:bound_on_propagation_speed}
	Let $R_1$ and $R_2>0$ be truncation radii,
	let $R > 0$ be an integer with
	$R < \min(R_1,R_2)$, and let $k < R$ be finite (i.e., 
	not thought of as large). Fix time $\Delta t>0$.

	Let $\mathcal{C}_1(0) = \mathcal{C}_1$,
	$\mathcal{C}_2(0) = \mathcal{C}_2 \in \Omega$ be
	configurations that agree inside $\Lambda_R$. Then there
	exists a coupling of the trajectories of
	$\mathcal{C}_1^{R_1}(t)$ and $\mathcal{C}_2^{R_2}(t)$, under which
	they agree inside $\Lambda_k$ for all $t \in [0, \Delta t]$ with
	probability at least $1- C R^{-\zeta}$ for $C$
	independent of $R,R_1,R_2$. 
\end{lemma}

Let us now show that 
the edge random variables $\updelta_e(t)$
indeed lead to a Markov process,
which, moreover, stays
admissible throughout the whole time.

\begin{lemma}\label{lem:markov_property}
	For any admissible $\mathcal{C}(0) =\mathcal{C}_0 \in \Omega$,
	the joint distribution of the 
	edge occupation
	trajectories 
	$\{\updelta_e(t)\}_{t\ge0}$
	over all edges $e$ defines a Markov process
	$\mathcal{C}(t)$ with values in $\Omega$ started from
	$\mathcal{C}(0) = \mathcal{C}_0$.
\end{lemma}

\begin{proof}
	The fact that the stochastic process
	$\mathcal{C}(t)$
	(coming from the limit 
	in \Cref{lemma:limit_of_edge_variables_in_R_exists})
	stays
	in the space $\Omega$ of admissible configurations
	once started in $\Omega$ 
	follows as $\tilde{R}\to\infty$ limit of 
	the estimate in \Cref{lemma:main_full_plane_lemma}.

	Let us show the Markov property of the stochastic
	process $\mathcal{C}(t)$.
	Let $F\bigl(
	\{\mathcal{C}(t)\}_{t \in [0,s]}\bigr) = F\bigl( \{\updelta_{e_1}(t),
	\dots, \updelta_{e_k}(t)\}_{t \in [0,s]} \bigr)$ be a bounded
	continuous functional on the space of right continuous paths
	in $\Omega$, which depends only on the evolution of finitely
	many edges $e_1,\dots, e_k$, and only on their values up to
	some time~$s$. 
	Let $\mathcal{F}_{t_0}$ be the $\sigma$-algebra generated
	by the random variables $\updelta_e(t)$ with $t\le t_0$.
	Then we have 
	\begin{equation}
		\label{eq:proof_of_Markov_property}
		\begin{split}
		\mathbb{E}_{\mathcal{C}_0}[F\bigl(  \{\mathcal{C}(t+t_0)\}_{t
		\in [0,s]}\bigr) \mid \mathcal{F}_{t_0}   ] &= \lim_{R
		\rightarrow \infty} \mathbb{E}_{\mathcal{C}_0}[F\bigl(
		\{\mathcal{C}^R(t+t_0)\}_{t \in [0, s]} \bigr)\mid
		\mathcal{F}_{t_0}   ]   \\ &= \lim_{R \rightarrow \infty}
		\mathbb{E}_{\mathcal{C}^R(t_0)}[ F\bigl(
		\{\mathcal{C}^R(t)\}_{t\in [0,s]} \bigr)  ] .
		\end{split}
	\end{equation}
	In the final expression above, $\mathcal{C}^R(t_0)$ is a
	random configuration obtained from running the truncated
	dynamics started from $\mathcal{C}_0$ for time $t_0$. The
	first equality is by the bounded convergence theorem, and
	the second is by the Markov property for the truncated
	dynamics. We claim that the final expression in
	\eqref{eq:proof_of_Markov_property} 
	is
	equal to $\mathbb{E}_{\mathcal{C}(t_0)}[ F\bigl(
	\{\mathcal{C}(t)\}_{t \in [0,s]} \bigr)  ]$. 

	First, with high probability we have
	$\mathcal{C}^R(t_0)
	= \mathcal{C}(t_0)$ on a large square $\Lambda_{R_0}$ for
	large but fixed $R_0 < R$, as $R \rightarrow \infty$.
	Indeed, by \Cref{lemma:delta_comparison_Rlim},
	for a constant $C_{R_0,\zeta}$ independent of $R$ we have 
	\begin{equation*}\label{eqn:bd1}
		\mathbb{P}(\mathcal{C}^R(t_0)|_{\Lambda_{R_0}} =
		\mathcal{C}(t_0)|_{\Lambda_{R_0}}) \leq C_{R_0,\zeta }
		R^{-\zeta}.
	\end{equation*}
	As a result, almost surely
	\begin{equation*}
		\lim_{R \rightarrow \infty} \mathbf{1}_{
			\mathcal{C}(t_0)|_{\Lambda_{R_0}} =\ssp
		\mathcal{C}^R(t_0)|_{\Lambda_{R_0}} }  = 1.
	\end{equation*}

	Now let us apply \Cref{lemma:bound_on_propagation_speed}, or, more precisely,
	its limiting version as $R_1\to+\infty$.
	Conditioned on the fact that the 
	configurations
	$\mathcal{C}(t_0)$ and $\mathcal{C}^R(t_0)$
	agree inside $\Lambda_{R_0}$, we can upper bound the probability that  
	the edge indicators for $e_1,\dots, e_k$ are different in $\mathcal{C}(t)$
	(which is the limit of $\mathcal{C}^{R_1}(t)$ as $R_1\to+\infty$)
	and
	$\mathcal{C}^R(t)$ for some time $t \in
	[0, s]$. Namely,
	\begin{equation*}
		\mathbf{1}_{ \mathcal{C}(t_0)|_{\Lambda_{R_0}} =
		\mathcal{C}^R(t_0)|_{\Lambda_{R_0}}} \left|
		\mathbb{E}_{\mathcal{C}^R(t_0)}[ F\bigl(
		\{\mathcal{C}^R(t)\}_{t\in [0,s]} \bigr)  ] -
		\mathbb{E}_{\mathcal{C}(t_0)}[ F\bigl(  \{\mathcal{C}(t)\}_{t
		\in [0,s]} \bigr)  ] \right| \leq \|F \| C_0 R_0^{-\zeta}
	\end{equation*}
	for some $C_0$ independent of $R_0$.
	Thus, with probability $1$ we have
	\begin{equation*}
		\limsup_{R \rightarrow \infty} \left|\mathbb{E}_{\mathcal{C}^R(t_0)}[ F\bigl(  \{\mathcal{C}^R(t)\}_{t\in [0,s]} \bigr)  ] \\
		- \mathbb{E}_{\mathcal{C}(t_0)}[ F\bigl(  \{\mathcal{C}(t)\}_{t \in [0,s]} \bigr)  ]\right|
		\leq \|F\| C_0 R_0^{-\zeta} .
	\end{equation*}
	Taking $R_0 \rightarrow \infty$ through the positive integers gives the Markov
	property.
\end{proof}

\begin{proof}[Proof of \Cref{thm:C_t_full_plane_well_defined}]
	To finalize the proof of \Cref{thm:C_t_full_plane_well_defined},
	it remains to verify the correct jump rates in the generator
	$G_u$ \eqref{eq:full_plane_generator},
	as the existence of the process in the space
	$\Omega$ of admissible configurations follows from \Cref{lem:markov_property}.

	Take a finite box $\Lambda_k$ and a large integer $R>k$.
	We approximate the dynamics of
	$\{\updelta_e(t)\}_{e \in \Lambda_k}$ by the truncated
	dynamics $\{\updelta^R_e(t)\}_{e \in \Lambda_k}$.
	Keeping track of the error terms in
	\Cref{lemma:main_full_plane_lemma,lemma:delta_comparison_Rlim},
	we see that the probability that
	$\{\updelta_e(t)\}_{e \in \Lambda_k} \neq
	\{\updelta_e^R(t)\}_{e \in \Lambda_k}$ for some $t \in [0,
	\Delta t]$ is upper bounded by $C \Delta t R^{-\zeta}$ for
	any $\zeta>0$ and for some $C$ depending on $\zeta$.

	Let $f$ be a cylindric function depending only on the
	occupation of the edges in $\Lambda_k$.
	Denote by $G^R_u$ the generator as in
	\eqref{eq:full_plane_generator}, but with
	with Poisson clocks
	outside of $\Lambda_R$ turned off. Let $f_0 \coloneqq
	f(\mathcal{C}_0)$ be the value of $f$ at the initial
	configuration $\mathcal{C}_0\in \Omega$. 
	We have
	\begin{align*}
		\frac{1}{\Delta t} \ssp\mathbb{E}_{\mathcal{C}_0} [
		f(\{\updelta_e(\Delta t)\}_{e \in \Lambda_k})- f_0 ] &=
		\frac{1}{\Delta t}\left( \mathbb{E}_{\mathcal{C}_0} [
		f(\{\updelta^R_e(\Delta t)\}_{e \in \Lambda_k}) -f_0]
		+\Delta t \, O( R^{-\zeta}) \right) \\
		&=  
		G^R_u f ( \{\updelta_e(0)\}_{e \in \Lambda_k} ) 
		+ O(\Delta t+R^{-\zeta}).
	\end{align*}
	Then, taking $\Delta t \rightarrow 0$ we have 
	\begin{equation*}
		\limsup_{\Delta t \rightarrow 0}
		\left|  
			\frac{1}{\Delta t}\ssp
			\mathbb{E}_{\mathcal{C}_0} [ f(\{\updelta_e(\Delta t)\}_{e\in \Lambda_k})- f_0 ] -
			G^R_u f(\{\updelta_e(0)\}_{e\in \Lambda_k} )
		\right| \leq C R^{-\zeta}
	\end{equation*}
	for some constant $C$. 
	Sending $R \rightarrow \infty$ gives the desired generator.
\end{proof}

\subsection{Proof of Theorem \ref{thm:C_t_preserves_Gibbs}}

We now show that the Markov process $\mathcal{C}(t)$ 
preserves the 
KPZ pure Gibbs state $\pi=\pi(\rho)$ with density
$\rho$ of the occupied vertical edges,
where $0<\rho<1$. For this, we utilize the 
dynamics in the quadrant constructed in
\Cref{sec:continuous_time_limit}.

\begin{definition}
	\label{def:nu_N_u_quadrant_measure}
	Fix a large positive integer $N$.
	Let $\nu_{N,u}$ be the 
	stochastic six vertex measure
	with parameters $u,q$ (cf. \Cref{sub:vertex_weights}; 
	and we suppress the dependence on $q$)
	in the quadrant 
	with the empty-$\uplambda$ 
	boundary conditions
	(\Cref{def:s6v_quadrant}),
	where $\uplambda\subset\mathbb{Z}_{\ge0}$
	is a random subset defined as follows.
	For each $\frac{N}{2}\le i\le \frac{3N}{2}$, 
	independently toss a coin with probability of 
	success $\rho$. When the coin is a success, include $i$ into $\uplambda$.
	There are no incoming vertical paths at the bottom boundary outside 
	the interval $\bigl[\frac{N}{2},\frac{3N}{2}\bigr]$.
\end{definition}

Let $M= M(N)$ be such that the
ratio $M/N$ converges to a positive constant
$\leq \frac{1-\delta_2}{20}$,
where $\delta_2$ is in a small neighborhood of $\delta_2(u)$
\eqref{eq:delta_1_2_through_t_u}.
For any $k< M$, define the 
$k \times k$ box $\Lambda_k^{N} = [N-k, N+k]
\times [M-k, M+k] \subset \mathbb{Z}_{\geq 0}\times
\mathbb{Z}_{\geq 1}$. 
For a cylindric
function $f$ on full plane configurations such that $f \in
\sigma(\Lambda_k)$ (i.e., depending only on the path
configuration in $\Lambda_k = [-k, k]^2$), we
denote by $f_N$ its pullback under the shift by $(-N, -M)$.
That is, $f_N$ is a function
depending only on the configuration in the shifted $k \times k$ box
$\Lambda_k^N$. 

Denote by $\mathcal{C}_N(t)$ the process
in the quadrant (\Cref{def:generator} and \Cref{thm:quad_process_exists})
started from $\nu_{N,u}$ and run at the (slowed down) speed $1/M(N)$.
Take $R<M$, and 
let $\mathcal{C}_N^R(t)$ be the truncation of the process
$\mathcal{C}_N(t)$ corresponding to turning off the Poisson
clocks outside of the shifted lattice square
$[N-R,N+R]\times [M-R,M+R]$. We assume that $\mathcal{C}_N^R(t)$
also starts from $\nu_{N,u}$.

Recall that we denote by $\mathcal{C}(t)$ and $\mathcal{C}^R(t)$
the usual and the truncated full plane processes,
where for $\mathcal{C}^R(t)$ we turn off the Poisson 
clocks
outside the lattice square $\Lambda_R=[-R,R]^2$.
Let $\mathcal{C}(t)$ and $\mathcal{C}^R(t)$ start from the pure state $\pi$.

To establish \Cref{thm:C_t_preserves_Gibbs},
it suffices to show that 
$\mathbb{E}_{\pi}\left[ f(\mathcal{C}(t)) \right]=\pi[f]$
for any bounded cylindric function $f\in \sigma(\Lambda_k)$.
Recall that $\mathbb{E}_{\pi}$ is the expectation with respect
to the process started from the pure state $\pi$,
and $\pi[f]$ is the integral of the function $f$ against $\pi$.

We will approximate $f(\mathcal{C}(t))$ by 
$f_N(\mathcal{C}_N(t))$ employing a sequence of couplings 
coming from the bounds on
information propagation 
(established in \Cref{sub:proof_thm_57}),
and also 
from a
coupling lemma and the local statistics theorem of
\cite{aggarwal2020limit} applied to $\nu_{N,u}$. 
The result will then follow from
the mapping of the stochastic six vertex measures
under the dynamics on the quadrant
(\Cref{thm:quad_process_exists}),
and continuity of $\nu_{N,u}$ in $u$.

We proceed by establishing several lemmas.

\begin{lemma}
	\label{lemma:proof_58_l1}
	We have
	$ \left| \mathbb{E}_{\pi}[ f ( \mathcal{C}(t) ) ] -
	\mathbb{E}_{\pi} [ f ( \mathcal{C}^R(t) ) ] \right| \leq
	\varepsilon_R$, with $\varepsilon_R \rightarrow 0$ as $R
	\rightarrow \infty$.
\end{lemma}
\begin{proof}
	For large enough $C$ we know that $\pi(\mathcal{C}(0)
	\in E_{R, C \log R}) \leq C R^{-\zeta}$, see \Cref{lem:strbound}. 
	Furthermore,
	given an initial configuration $\mathcal{C}(0)$ such that
	$\mathcal{C}(0) \notin E_{R, C \log R}$, under the
	standard coupling of the full plane process
	$\mathcal{C}(t)$ with its truncated counterpart
	$\mathcal{C}^R(t)$, the configurations
	$\mathcal{C}(t)$ and
	$\mathcal{C}^R(t)$ agree on $\Lambda_k$ (the subset 
	determining the values of $f$) except on an
	event with probability at most $O(R^{-\zeta})$. 
\end{proof}

\begin{lemma}
	\label{lemma:proof_58_l2}
	We have 
	$\left| \mathbb{E}_{\pi}[ f ( \mathcal{C}^R(t) ) ] -
	\mathbb{E}_{\nu_{N,u}} [ f_N ( \mathcal{C}_N^R(t) ) ] \right|
	\rightarrow 0$ as $N \rightarrow \infty$ with $R$
	arbitrary but fixed.
\end{lemma}
\begin{proof}
	This is a statement about finite state space Markov
	chains. It suffices to see that the jump rates of
	$\mathcal{C}_N^R$ converge to those of $\mathcal{C}^R$, and
	that $\nu_{N,u}|_{\Lambda_R^N}$ converges to
	$\pi|_{\Lambda_R}$ (if we identify the two boxes). 
	The former claim is established by the computation
	in \Cref{sub:bulk_limit_heuristics}.	
	The
	latter claim follows, for example, from an application of
	Proposition 2.17 of \cite{aggarwal2020limit}.
\end{proof}

Note that the events $E_{R,A}$
(\Cref{def:admissible_configurations})
make
sense for quarter plane configurations, so long as we
replace the box $\Lambda_R$ centered at $(0,0)$ with 
the shifted box $\Lambda^{N}_R$, where $R$ is fixed, and $N$ is sufficiently large
so that $R<\min(N,M(N))$.
Denote the corresponding shifted event centered at $(N,M)$ by $E^{N}_{R,A}$.

\begin{lemma}\label{lem:quarterplanepropdef}
	There exists $C$ such that for all positive integers $R<\min(N,M(N))$ we have
	\begin{equation*}
		\mathbb{P}\left( 
			\textnormal{there exists}\ t\in [0,\Delta t]\colon
			\mathcal{C}_N(t)\in E_{R,C\log R}^N
		\right)\le CR^{-\zeta},
	\end{equation*}
	where $\mathcal{C}_N(t)$ is the quarter plane process
	started from $\nu_{N,u}$.
\end{lemma}
\begin{proof}
	This is established similarly to \Cref{lemma:main_full_plane_lemma},
	see also \Cref{lemma:bound_on_propagation_speed}.
\end{proof}

\begin{lemma}
	\label{lemma:proof_58_l3}
	We have
	\begin{equation*}
		\limsup_{N \rightarrow \infty} \left|   \mathbb{E}_{\nu_{N,u}}
		[ f_N ( \mathcal{C}_N^R(t) ) ] - \mathbb{E}_{\nu_{N,u}} [
		f_N ( \mathcal{C}_N(t) ) ] \right| \leq \varepsilon_R,
	\end{equation*}
	with
	$\varepsilon_R \rightarrow 0$ as $R \rightarrow \infty$.
\end{lemma}
\begin{proof}
	This follows similarly to
	\Cref{lemma:delta_comparison_Rlim}. 
	Indeed, one can bound the probability that, at any time in
	some compact time interval, either $\mathcal{C}_N^R(t)$ or
	$\mathcal{C}_N(t)$ develop long sequences of consecutive
	horizontally adjacent vertices which are not $\ru$. 
	These bounds are provided by
	\Cref{lem:quarterplanepropdef} and an analogue of
	\Cref{lemma:main_full_plane_lemma} (proved in the same manner
	by coupling, see \Cref{sec:proof_lemma_S6} below).
\end{proof}

\begin{lemma}
	\label{lemma:proof_58_l4}
	We have
	$ \left| \nu_{N,u+t/M(N)} [ f_N  ] - \pi[ f  ] \right| \rightarrow
	0$ as $N \rightarrow \infty$.
\end{lemma}
\begin{proof}
	Here we use the fact that the map $\varphi(\rho)$
	\eqref{eq:phi_of_rho_define}
	is continuous in $u$.
	Namely,
	for each $N > 0$ let $M = M(N) =
	\frac{1-\delta_2 - \varepsilon}{20} N$ for some $\varepsilon >0$
	small enough. Note that we can couple the boundary
	conditions of a six vertex configuration $\mathcal{C}_N$
	sampled from $\nu_{N,u}$ with those of $\mathcal{C}$
	sampled from $\pi$ such they agree with probability $1$ on
	$[N/2, 3N/2] \times \{0\}$. Therefore by Proposition 2.17
	of \cite{aggarwal2020limit}, for some constant $c$
	independent of $N$, there is a coupling of 
	the configurations 
	$\mathcal{C}_N$ and 
	$\mathcal{C}$ such that they agree on the set of edges in
	$[N-M,N+M] \times [0,2 M]$ with probability at least $1 -
	c^{-1}e^{-c M}$. We apply this with $u$ replaced by $u + t
	/ M$ and $\pi$ replaced by the Gibbs measure
	$\tilde{\pi}^M$ with $x$ slope $\rho$ and spectral
	parameter $u + t / M$. Then we note that
	$\tilde{\pi}^M|_{\Lambda_k^N}$ can in turn be coupled with
	$\pi|_{\Lambda_k}$ (the marginal of the Gibbs measure with
	spectral parameter $u$ on $\Lambda_k$, where the edges in
	$\Lambda_k$ are identified with those of $\Lambda_k^N$),
	such that they agree with probability $1-\varepsilon_M$, with
	$\varepsilon_M \rightarrow 0$ as $M \rightarrow \infty$.
	Therefore, we get a coupling of $\nu_{N,u +
	t/M}|_{\Lambda_k^N}$ with $\pi|_{\Lambda_k}$ such that the
	configurations agree with probability lower bounded by 
	$1- c^{-1}e^{-c M} - \varepsilon_M$.
	Since $M \rightarrow \infty$ as $N \rightarrow \infty$, this implies the claim.
\end{proof}

Now we can finish the proof of
\Cref{thm:C_t_preserves_Gibbs}. 
Using \Cref{lemma:proof_58_l1,lemma:proof_58_l2,lemma:proof_58_l3}, the
measure mapping property of the quarter plane dynamics 
(\Cref{thm:quad_process_exists}), and
then \Cref{lemma:proof_58_l4}, we get
\begin{align*}
\mathbb{E}_{\pi}[ f(\mathcal{C}(t)) ] &= \mathbb{E}_{\pi}[ f(\mathcal{C}^R(t)) ] + \varepsilon_R  \\
&= \mathbb{E}_{\nu_{N,u}}[ f_N(\mathcal{C}_N^R(t)) ] + \varepsilon_N + \varepsilon_R \\
&= \mathbb{E}_{\nu_{N,u}}[ f_N(\mathcal{C}_N(t)) ] + \varepsilon_N' + \varepsilon_R'\\
&= \nu_{N,u+t /M}[ f_N ] + \varepsilon_N' + \varepsilon_R'\\
&= \pi[ f ] + \varepsilon_N'' + \varepsilon_R'
\end{align*}
where $\varepsilon_N'' \rightarrow 0$ as $N \rightarrow \infty$
with $R$ fixed, and $\varepsilon_R' \rightarrow 0$ as $R
\rightarrow \infty$. So taking $N \rightarrow \infty$ first,
then $R \rightarrow \infty$, yields the result.

\section{Proof of Lemma \ref{lemma:main_full_plane_lemma} via coupling}
\label{sec:proof_lemma_S6}

Recall that
$E_{R,A}$ is the
set of up-right path configurations
which have
horizontal strings of occupied edges (i.e., no vertices
of occupation type $\ru$)
of length $>A$, in an $R$-neighborhood of the origin.
Under the truncated full plane dynamics $\mathcal{C}^R(t)$, 
configurations in $E_{R,A}$
might lead to jump propagation of length $>A$.
In this section we prove \Cref{lemma:main_full_plane_lemma}
which states that
if the initial configuration was not in 
$E_{R,C(\log R)^{p_0}}$, then with high probability 
it will never be in 
$E_{R,\tilde C(\log R)^{p}}$ with some $p\ge p_0$
under the dynamics $\mathcal{C}^{\tilde R}(t)$, up to time $t\le T$.
Here $\tilde R\ge R$ is an arbitrary integer (and we keep the notation $\tilde R$
throughout the section for consistency with \Cref{lemma:main_full_plane_lemma}).
We achieve this bound via a monotone coupling of the 
dynamics $\mathcal{C}^{\tilde R}(t)$ on a given horizontal
slice
with a particle system which is easier to analyze.

\subsection{Annihilation-Jump particle system}
\label{sub:particle_system}

We start by defining the jump rates of the
\emph{Annihilation-Jump particle system}
(AJ). Its state space 
consists of a
sequence of \emph{$\oplus$ particles} $a_i \in
\mathbb{Z}_{\leq 0}$, and a sequence of \emph{$\ominus$
particles} $b_i  \in \mathbb{Z}_{\leq 0}$, which satisfy
either $\cdots < b_2 <  a_2 < b_1< a_1 \leq 0$, or $\cdots <
a_2 <  b_2 < a_1< b_1 \leq 0$. 
Pick $\mathfrak{m}>0$, this parameter is the overall time scaling factor in the AJ dynamics.
The possible jumps and their
rates are as follows:
\begin{itemize}
	\item 
		Any particle at position $z$, with the closest
		particle to its left at position $y < z$, jumps to
		position $x$ at rate $\mathfrak{m}$ for any $y < x < z$. 
		This rule does not distinguish $\oplus$ and $\ominus$ particles,
		and is like in the Hammersley process \cite{hammersley1972few}, \cite{aldous1995hammersley}.

	\item 
		Take any $\ominus$ particle $b_i$ and let the closest $\oplus$ particle to
		its right be $a_j$ (so $j = i$ or $i-1$). 
		The pair $(b_i,
		a_j)$ disappears at rate $2\mathfrak{m}$, and the sequences of
		remaining particles $\{a_k \}_{k \neq j}, \{b_k \}_{k \neq
		i}$ are relabeled.  If $i = 1$ and $b_1$ is the rightmost
		particle, then $b_1$ simply disappears.

	\item 
		Take any $\oplus$ particle $a_i$ and let the closest $\ominus$ particle to
		its right be $b_j$ (so $j=i$ or $i-1$). The pair $(a_i, b_j)$ disappears at
		rate $2\mathfrak{m}$, and the remaining particles
		are relabeled. 
		If $i =
		1$ and $a_1$ is the rightmost particle, then $a_1$ simply
		disappears.
\end{itemize}
Similarly to the full plane dynamics $\mathcal{C}(t)$,
a priori 
it is not clear that
the AJ particle system is well-defined.
Indeed, there may be initial configurations
leading to infinitely many jumps through a finite space in finite time.
However, we only need to analyze a truncated version of AJ. 

\begin{definition}[Truncated AJ particle system]
	Let $\tilde{R},N_0 >0$ be positive integers which determine the truncation. 
	Given an
	initial configuration, define $a_i(t) = a_i^{\tilde{R}}(t),
	b_i(t) = b_i^{\tilde{R}}(t)$ to evolve according the
	following rules. Put an independent rate $\mathfrak{m}$ Poisson clock $P_z(t)$
	at each lattice site $z \in \mathbb{Z}_{\leq 0}$. 
	If $P_z$,
	$|z| \leq \tilde{R}$, rings at time $t$, then
	we take the leftmost particle to the right of $z$
	(regardless of whether it is 
	$\oplus$ or $\ominus$) and place it into $z$.
	If the clock $P_z$ rings at a particle $z=a_i$ or $z=b_j$,
	ignore this ring.

	Also take a collection of rate $2\mathfrak{m}$ Poisson clocks $P_i^a(t)$
	and $P_i^b(t)$ for
	each integer $i \geq 1$. For these clocks,
	\begin{itemize}
		\item If $P_i^b$ rings at time $t$, and $i\leq N_0$, then
			$b_{i}(t^-)$ and the $\oplus$ particle to its
			right, $a_j(t^-)$, annihilate each other and disappear,
			and we relabel the particles. If $i = 1$ and $b_1$ is the
			rightmost particle, then $b_1$ simply disappears.
		\item If $P_i^a$ rings at time $t$, and $i \leq N_0$, then
			$a_{i}(t^-)$ and the $\ominus$ particle to its
			right, $b_j(t^-)$, annihilate each other and disappear,
			and we relabel the particles. If $i = 1$ and $a_1$ is the
			rightmost particle, then $a_1$ simply disappears.
	\end{itemize}
	In the truncated dynamics we ignore all other clock rings.
\end{definition}

\subsection{Stochastic domination}
\label{sub:stoch_dom}

Here we describe in which sense the AJ system dominates the full plane dynamics.

Let $\uplambda^{(0)}, \uplambda^{(1)}$ be two infinite
interlacing signatures, that is, $\uplambda_i^{(1)} \geq \uplambda_i^{(0)} \geq
\uplambda_{i+1}^{(1)}$
for all $i \in
\mathbb{Z}$.
This pair represents a six vertex path
configuration on a one-row lattice $\mathbb{Z} \times
\{1\}$, by encoding the positions of 
occupied vertical edges
entering ($\uplambda^{(0)}$) and leaving
($\uplambda^{(1)}$)
this row. By analogy with the AJ system, 
below we also refer to occupied vertical edges as \emph{particles}.
\begin{definition}
Given $\uplambda^{(0)}, \uplambda^{(1)}$, define
sequences $\{\mathbf{A}_i(\uplambda^{(0)},
\uplambda^{(1)})\}_{i=1}^\infty$,
$\{\mathbf{B}_i(\uplambda^{(0)},
\uplambda^{(1)})\}_{i=1}^\infty$ as follows:
\begin{align*}
\mathbf{A}_1 &\coloneqq \max \{ \uplambda^{(1)}_i \colon  \uplambda^{(1)}_i \leq 0, \uplambda^{(1)}_i \neq \uplambda^{(0)}_i , \uplambda^{(1)}_i \neq \uplambda^{(0)}_{i-1} \}; \\
\mathbf{A}_{i+1} &\coloneqq \max \{ \uplambda^{(1)}_j < \mathbf{A}_{i} \colon \uplambda^{(1)}_j \neq \uplambda^{(0)}_j , \uplambda^{(1)}_j \neq \uplambda^{(0)}_{j-1} \}  .
\end{align*}
In other words, the sequence $\mathbf{A} = (\mathbf{A}_1 >
\mathbf{A}_2 > \cdots)$ indexes the positions of the
particles in $\uplambda^{(1)}$ with no particle at the same
position in $\uplambda^{(0)}$.
Similarly,
\begin{align*}
\mathbf{B}_1 &\coloneqq \max \{ \uplambda^{(0)}_i \colon  \uplambda^{(1)}_i \leq 0, \uplambda^{(0)}_i \neq \uplambda^{(1)}_i , \uplambda^{(0)}_i \neq \uplambda^{(1)}_{i+1} \}; \\
\mathbf{B}_{i+1} &\coloneqq \max \{ \uplambda^{(1)}_j < \mathbf{B}_{i} \colon \uplambda^{(0)}_j \neq \uplambda^{(1)}_j , \uplambda^{(0)}_j \neq \uplambda^{(1)}_{j+1}\}  .
\end{align*}
In other words, the sequence $\mathbf{B} = (\mathbf{B}_1 >
\mathbf{B}_2 > \cdots)$ indexes the positions of the
particles in $\uplambda^{(0)}$ with no particle at the same
position in $\uplambda^{(1)}$.
\end{definition}

For two consecutive horizontal rows (at the $0$-th and the first slice)
of a full plane six vertex model
configuration evolving under some Markov dynamics,
denote by
$\mathbf{A}(t), \mathbf{B}(t)$
the time-dependent random variables corresponding to this dynamics.

If
$s = (s_1 > s_2 > s_3 > \cdots), r = (r_1 > r_2 > r_3 >
\cdots)$ are two decreasing sequences of numbers, then we
say $s \leq r$ if $s_i \leq r_i$ for all $i = 1,2,3, \dots$.
We are now in a position to formulate the 
lemma on stochastic domination.

\begin{lemma}
	\label{lemma:monotone_coupling}
	Let $\mathcal{C}^{\tilde{R}}(t)$ denote the trajectory  of
	the full plane six vertex model under the truncated
	dynamics, which gives rise to the quantities
	$\mathbf{A}(t), \mathbf{B}(t)$.
	We can couple $\mathcal{C}^{\tilde{R}}(t)$
	with a truncated AJ particle system $(a(t),b(t))$ with initial
	configuration $a(0)= \mathbf{A}(0), b(0)= \mathbf{B}(0)$ 
	and suitable $\tilde R, N_0$, and $\mathfrak{m}$,
	in
	such a way that with probability $1$, for all $t \geq 0$ we have
	\begin{equation}
		\label{eq:monotone_coupling}
		a(t) \leq \mathbf{A}(t), \qquad  b(t) \leq \mathbf{B}(t).
	\end{equation}
	In words, the AJ particles are always further to the left 
	compared to their six vertex model counterparts.
	We say that \eqref{eq:monotone_coupling}
	means that the AJ system \textbf{stochastically
	dominates} this slice of the full plane dynamics.
	See \Cref{fig:initial_config_for_coupling} 
	for an illustration.
\end{lemma}
\begin{figure}[htpb]
	\centering
	\includegraphics[width=.8\textwidth]{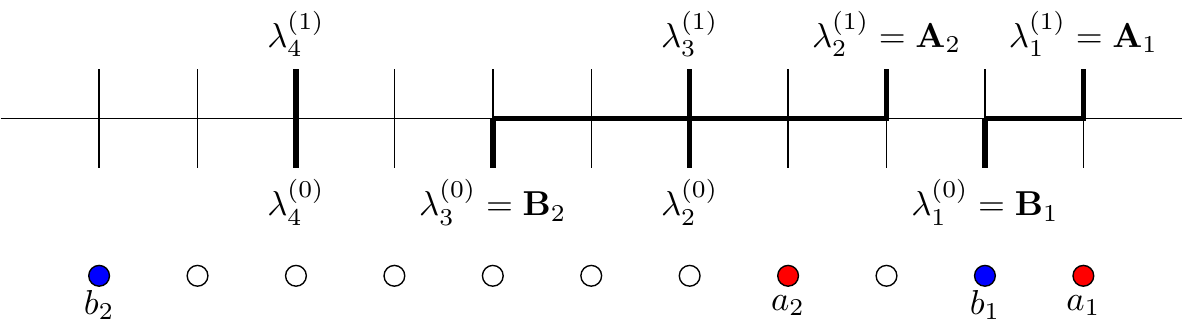}
	\caption{Putting the six vertex and the AJ configurations together.
		In the figure,
		the configurations are 
		ordered as in \eqref{eq:monotone_coupling}.}
	\label{fig:initial_config_for_coupling}
\end{figure}
\begin{proof}
	We couple the Poisson clocks 
	in the truncated full plane six vertex model dynamics $\mathcal{C}^{\tilde R}(t)$
	with the clocks in a suitably truncated AJ particle system.
	For the truncation in the AJ system
	we take the same $\tilde R$, and let 
	let $N_0$ be the maximum integer such that
	$\mathbf{A}_{N_0}(0) ,
	\mathbf{B}_{N_0}(0) \geq -\tilde{R}$. 
	
	We may realize $\mathcal{C}^{\tilde R}(t)$
	by putting independent rate $\mathfrak{m}$ Poisson
	clocks 
	at each
	vertical edge in the finite square $\Lambda_{\tilde R}$, where
	$\mathfrak{m}\equiv \mathfrak{m}(u)\coloneqq \max\left\{ \mathfrak{a}(u),\mathfrak{b}(u),\mathfrak{c}(u) \right\}$,
	and $\mathfrak{a}(u),\mathfrak{b}(u),\mathfrak{c}(u)$ are the rates in the full plane dynamics 
	defined in
	\eqref{eq:abc_rates}.
	When
	a clock at an edge rings (almost surely, there is at most one ring in finite time 
	in the truncated process),
	depending
	on the current configuration 
	and possibly on the
	outcome of an independent coin flip,
	this ring initiates a
	jump in $\mathcal{C}^{\tilde R}$, or we ignore it. 
	Here coin flips are needed to model smaller jump rates. For example,
	if 
	$\mathfrak{m}=\mathfrak{a}>\mathfrak{b}$, then we model rings at rate $\mathfrak{b}$ 
	from the Poisson clock of rate $\mathfrak{m}$ and 
	the coin with probability of success $\mathfrak{b}/\mathfrak{m}$.

	At time $t^-$, denote the 
	six vertex configuration on rows $0$ and $1$ by
	$\uplambda^{(0)} = \uplambda^{(0)}(t^-)$, $\uplambda^{(1)} =
	\uplambda^{(1)}(t^-)$, and the AJ system configuration by $a = a(t^-), b = b(t^-)$. 
	We first define an auxiliary AJ like particle system which is 
	a function of the
	Poisson clock rings in $\mathcal{C}^{\tilde R}$.
	It then will be evident that this auxiliary dynamics 
	is dominated by the AJ system in the same sense as in 
	\eqref{eq:monotone_coupling},
	which will lead to the desired statement. 
	We refer to this auxiliary particle
	system as the ``AJ system'' to simplify the wording.

	Pick $z\in \mathbb{Z}$
	with 
	$|z| \leq \tilde{R}$,
	and consider the configuration of 
	occupied edges around $z$. 
	There are six possible configurations
	corresponding to the allowed configurations under the 
	six vertex model (cf. \Cref{fig:6types}).
	We denote vertical edge locations corresponding to the configurations at
	horizontal levels $0$ and $1$
	by $(z,0)$ and $(z,1)$, respectively.
	In the six cases below we assume that a 
	clock rings at one of these edges, and define how the 
	particle system changes as a result of this ring.
	\begin{enumerate}
		\item 
			Let the path configuration around 
			$z$ be $(0,0;0,0)$, which means that 
			$\uplambda^{(1)}_{i} < z <
			\uplambda^{(0)}_{i-1} $.
			From \Cref{fig:rates_for_G_quad}
			we see that only the clock at the bottom vertical edge $(z,0)$
			may ring.
			In this case, we let
			the leftmost AJ particle to the right of
			$z$ (if it exists) jump into $z$, provided that $z$
			is not occupied by another AJ particle. 
			If no such AJ particle exists, 
			ignore this ring. 
			From now on, we will simply say ``the particle to the right of $z$
			attempts to jump into~$z$''.

			On the six vertex side, after this clock ring 
			the edge $(z,0)$ might stay empty if
			$\uplambda^{(0)},\uplambda^{(-1)}$ locally do not look
			like the configurations in \Cref{fig:rates_for_G_quad}, left,
			or if the coin flips do not produce an actual jump
			initiation in $\mathcal{C}^{\tilde R}$.
			In the remaining situation
			when the edge $(z,0)$ becomes occupied, one of the 
			vertical occupied edges at $(z',0)$, $z'>z$,
			must instantaneously become empty due to the jump propagation. 
			If $z'$ is not one of the $\mathbf{B}_j$'s, 
			then this corresponds to a right jump in the six vertex model
			(namely, a ``creation''
			of a new pair $\mathbf{A}_k,\mathbf{B}_k$ to the right
			of $z$, and relabeling of $\mathbf{A},\mathbf{B}$).
			Alternatively, $z'$ could be equal to the
			leftmost of the $\mathbf{B}_j$'s which are greater than $z$,
			and this is the furthest left jump that may occur under 
			$\mathcal{C}^{\tilde R}$.
			Clearly, in all these cases the domination
			\eqref{eq:monotone_coupling}
			is preserved by the jumps in the two systems.
			
			Here is an illustration of the moves in the six vertex
			model and the AJ system, with the furthest possible left jump
			of $(\mathbf{A},\mathbf{B})$
			under $\mathcal{C}^{\tilde R}$:
			\begin{center}
				\includegraphics[width=.8\textwidth]{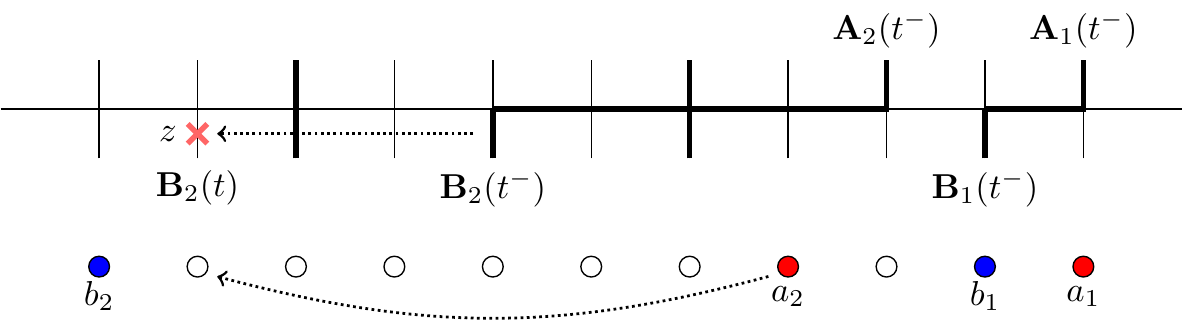}
			\end{center}
			The remaining five cases are considered similarly, 
			and we discuss them in less detail. One readily sees that 
			in each of the remaining five cases, the domination
			\eqref{eq:monotone_coupling} is preserved.
			
		\item 
			Let the path configuration 
			around $z$ be $(1,1;1,1)$,
			which means that 
			$z = \uplambda^{(1)}_{i} =\uplambda^{(0)}_{i-1}$.
			In this case, only the clock at the bottom edge $(z,0)$
			may ring, 
			and in the AJ system we define that
			the particle to the right of $z$ attempts to jump into $z$.
			In the illustration below we 
			display the furthest possible left jump 
			of $(\mathbf{A},\mathbf{B})$
			under 
			$\mathcal{C}^{\tilde R}$,
			which arises when the edge $(z,0)$
			becomes empty, and the furthest possible
			edge to the right becomes occupied:
			\begin{center}
				\includegraphics[width=.8\textwidth]{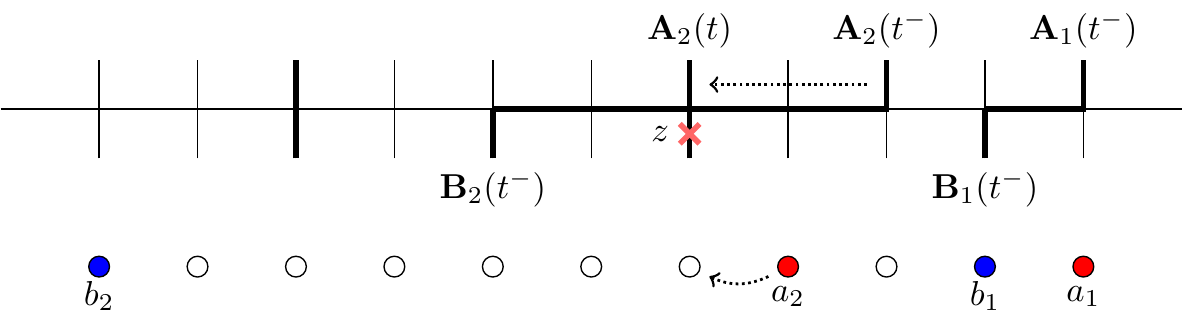}
			\end{center}

		\item 
			Let the path configuration 
			around $z$ be $(0,1;0,1)$,
			which means that 
			$\uplambda^{(0)}_{i} < z < \uplambda^{(1)}_i$.
			In this case, only the clock at the top edge $(z,1)$
			may ring, 
			and in the AJ system we define that
			the particle to the right of $z$ attempts to jump into $z$.
			In the illustration below we 
			display the furthest possible left jump 
			of $(\mathbf{A},\mathbf{B})$
			under 
			$\mathcal{C}^{\tilde R}$:
			\begin{center}
				\includegraphics[width=.8\textwidth]{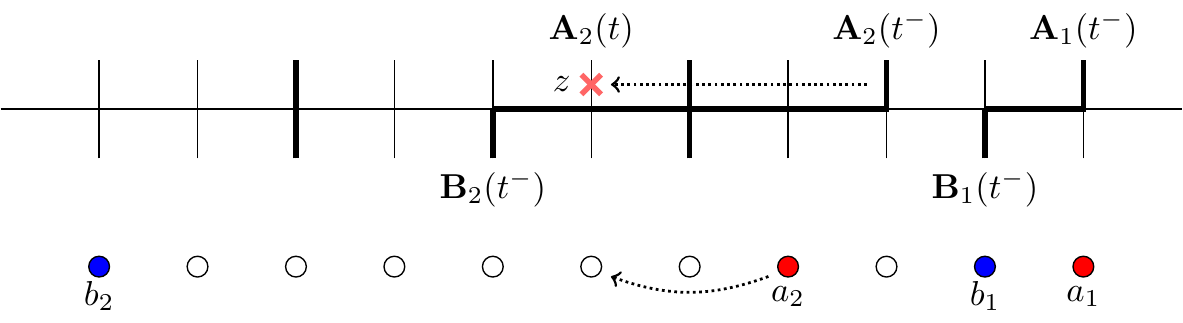}
			\end{center}
			
		\item 
			Let the path configuration 
			around $z$ be $(1,0;1,0)$,
			which means that 
			$z = \uplambda^{(1)}_{i} =\uplambda^{(0)}_{i}$.
			In this case, only the clock at the top edge $(z,1)$
			may ring, 
			and in the AJ system we define that
			the particle to the right of $z$ attempts to jump into $z$.
			In the illustration below we 
			display the furthest possible left jump 
			of $(\mathbf{A},\mathbf{B})$
			under 
			$\mathcal{C}^{\tilde R}$,
			which arises when the edge $(z,1)$ becomes empty, 
			and the furthest possible edge to the right becomes occupied:
			\begin{center}
				\includegraphics[width=.8\textwidth]{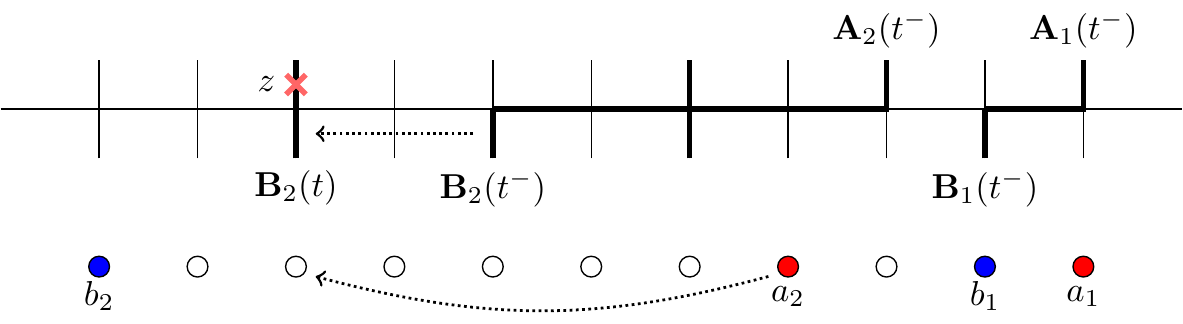}
			\end{center}
		\item 
			Let the path configuration 
			around $z$ be $(0,1;1,0)$,
			which means that 
			$\uplambda^{(0)}_i<z=\uplambda^{(1)}_i<\uplambda^{(0)}_{i-1}$.
			Then, by definition, $z=\mathbf{A}_k$ for some $k\le i$.
			Here the clock may ring at either $(z,0)$ or $(z,1)$
			(which corresponds to the double rate $2\mathfrak{m}$ of annihilation
			in the AJ system).
			In both situations,
			we define that $a_k$
			and the particle to its right annihilate
			(or $a_1$ disappears if it was the rightmost of the AJ particles).
			Here are the illustrations of both cases:
			\begin{center}
				\includegraphics[width=.8\textwidth]{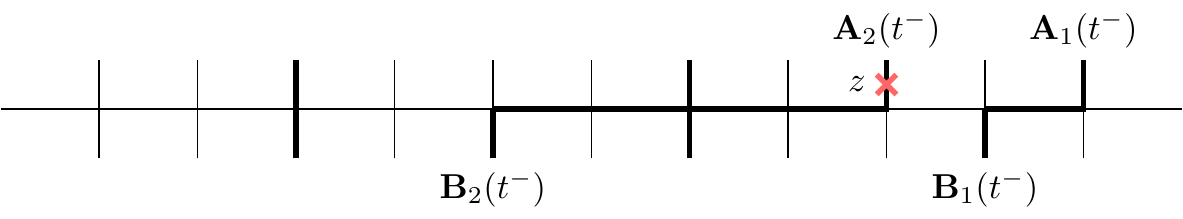}
				
				\includegraphics[width=.8\textwidth]{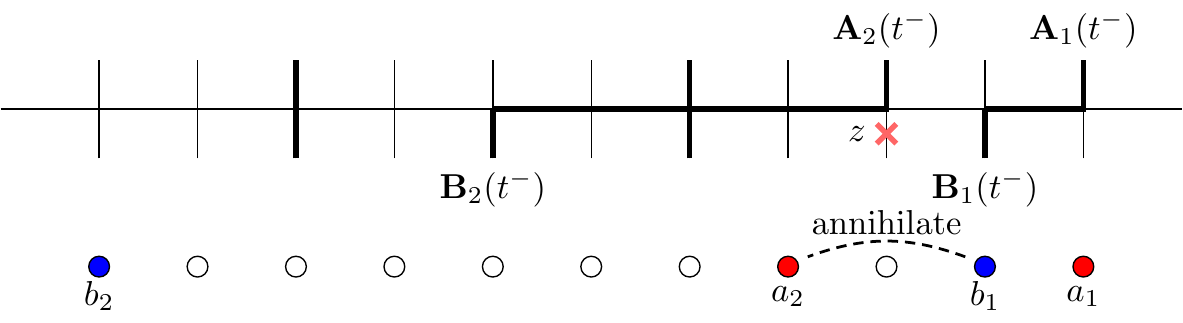}
			\end{center}
			In $\mathcal{C}^{\tilde R}$, 
			both clock rings might lead to 
			no change, or to a right jump of $\mathbf{A}_k$, 
			or to an annihilation of $\mathbf{A}_k$ and $\mathbf{B}_{k-1}$
			(and relabeling of the remaining particles in $\mathbf{A},\mathbf{B}$).
			In all these cases, the domination
			\eqref{eq:monotone_coupling} is preserved.
		\item 
			Let the path configuration 
			around $z$ be $(1,0;0,1)$,
			which means that 
			$\uplambda^{(1)}_{i+1}<z=\uplambda^{(0)}_i<\uplambda^{(1)}_{i}$.
			Then, by definition, $z=\mathbf{B}_k$ for some $k\le i$.
			Here the clock may ring at either $(z,0)$ or $(z,1)$
			(again, this corresponds to the fact 
			that the annihilation rate is $2\mathfrak{m}$).
			In both situations,
			we define that $b_k$
			and the particle to its right annihilate
			(or $b_1$ disappears if it was the rightmost of the AJ particles).
			Here are the illustrations of both cases,
			and we similarly see that the domination
			\eqref{eq:monotone_coupling} is preserved:
			\begin{center}
				\includegraphics[width=.8\textwidth]{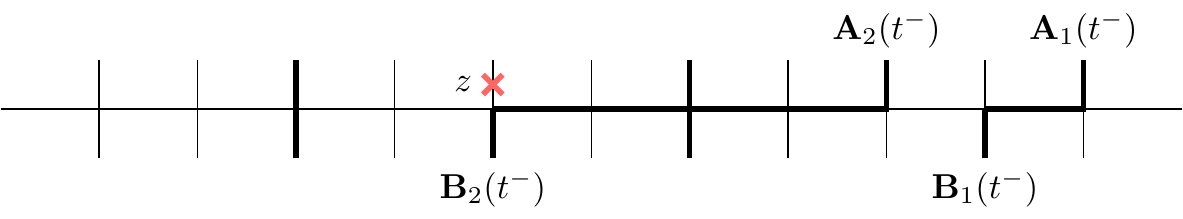}
				
				\includegraphics[width=.8\textwidth]{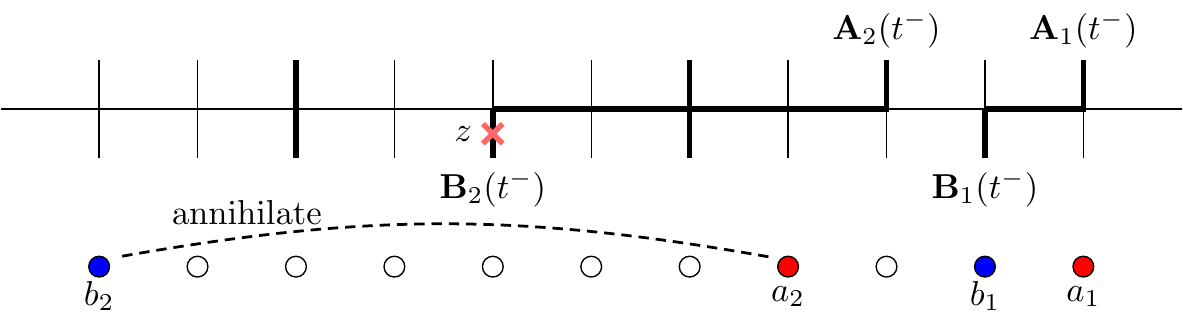}
			\end{center}
	\end{enumerate}
	
	We have thus constructed 
	the auxiliary AJ like particle system on $\mathbb{Z}$
	which is 
	a function of the rate $\mathfrak{m}$ Poisson clock rings in 
	$\mathcal{C}^{\tilde R}$,
	and which dominates the dynamics of 
	$(\mathbf{A},\mathbf{B})$.
	In the auxiliary system,
	particles indeed annihilate at rate 
	$2\mathfrak{m}$ per pair
	$(a,b)$ or $(b,a)$,
	but the auxiliary system lacks some of the 
	AJ system's
	left jumps. Namely, if a clock from $\mathcal{C}^{\tilde R}$
	at $z\notin \{a_i\}_{i\ge1}\cup \{b_i\}_{i\ge 1}$
	contributes to an annihilation event, 
	then it did not produce a jumping event in the auxiliary system.
	Therefore, adding extra independent jump events
	at rate $\mathfrak{m}$ produces the full AJ system,
	which clearly jumps to the left more
	often than the auxiliary system.
	We see that the coupling we constructed indeed
	satisfies the domination \eqref{eq:monotone_coupling},
	which
	completes the proof of \Cref{lemma:monotone_coupling}.
\end{proof}

\subsection{Completing the 
proof of Lemma \ref{lemma:main_full_plane_lemma}
via an estimate in the AJ system}
\label{sub:propbound_proof}

To finalize the proof of \Cref{lemma:main_full_plane_lemma}
we need to show that, given 
that the initial configuration was not in 
$E_{R,C(\log R)^{p_0}}$ for some $R,C$
(that is, it did not have long strings without vertices of type
$\ru$), then with high probability
the configuration will not be in 
$E_{R,\tilde C(\log R)^{p}}$ up to time $t\le T$,
where $p\ge p_0$ and $\tilde R\ge R$.
Due to the $\mathbb{Z}^2$ translation invariance of the dynamics, we 
may 
consider the event that a long string of vertices without
$\ru$ starts at $(0,1)$, and then take a union bound
over all vertices in $\Lambda_R$ (there are $\mathrm{const}\cdot R^2$ of them).

By the stochastic domination 
shown in \Cref{lemma:monotone_coupling},
it suffices to upper bound the following probability under the AJ system:
\begin{equation}
	\label{eq:AJ_probability_need_to_estimate}
	\mathbb{P}^{AJ}(
	a_1(T) \leq - \tilde{C} (\log R)^p).
\end{equation}
Here the AJ particle system is truncated
at $\tilde R,N_0$, and
starts from the configuration 
$a(0) = \mathbf{A}(0), b(0)=\mathbf{B}(0)$,
as described before \Cref{lemma:monotone_coupling}.
Note that since the AJ system jumps only to the left, 
\eqref{eq:AJ_probability_need_to_estimate} also is an upper bound 
for 
$\mathbb{P}^{AJ}(\exists\, t\in [0,T]\colon 
a_1(t) \leq - \tilde{C} (\log R)^p)$. The latter quantity 
is an upper bound for the corresponding quantity in the
dynamics $\mathcal{C}^{\tilde R}$.
The constants $\tilde{C}>0$ and $p\ge p_0 > 0$ 
in \eqref{eq:AJ_probability_need_to_estimate}
will be determined later to make the probability
\eqref{eq:AJ_probability_need_to_estimate} sufficiently small, namely, 
of order
$R^{-\zeta}$ for any $\zeta>0$ (as $R$ grows).

Let $n$ be such that $a_n(0)\le -\tilde C(\log R)^p$.
Observe that
$n$
is large for large $R$ because the initial configuration
is not in $E_{R,C(\log R)^{p_0}}$, namely,
\begin{equation}
	\label{eq:n_in_terms_of_R}
	n\ge \frac{\tilde C}{C}\ssp(\log R)^{p-p_0}.
\end{equation}
Moreover, $a_n(0)$ cannot be too large in the absolute value 
because the initial configuration
is not in $E_{R,C(\log R)^{p_0}}$, namely,
\begin{equation}
	\label{eq:bound_on_a_n}
	|a_n(0)|\le C n(\log R)^{p_0}.
\end{equation}

To determine if the event in
\eqref{eq:AJ_probability_need_to_estimate} occurred,
we may only look at the behavior of the particles started in
$(a_n(0),0]$ up to time $T$. Moreover, by taking $a_n(0)$
sufficiently small, we may assume that
the event in
\eqref{eq:AJ_probability_need_to_estimate} 
is due entirely to 
a combination of annihilations,
and jumps caused by Poisson clocks in the interval
$(a_n(0),0]$ up to time~$T$.
Moreover, 
$a_1(T) \leq - \tilde{C} (\log R)^p$ 
implies that there are 
no particles left 
at time $T$ between $-\tilde C(\log R)^p$ and $0$.
We treat separately the cases
when this absence of particles is 
mainly due to annihilations or mainly due to particles jumping out. 

Denote by 
$K(n)$ the (random) number of particles out of the ones
started at time $t=0$ in $(a_n(0), 0]$, and which did not get
annihilated up to time $t=T$.
Fix $r\in \left( \tfrac 12,1 \right)$.
If $a_1(T)\le -\tilde C(\log R)^p$ and
$K(n)< n^r$, then there were many annihilations, while if 
$K(n)\ge n^r$, then there should have been many jumps.
First, we estimate the probability of many annihilations:

\begin{lemma}
	\label{lemma:K_n_estimate}
	For some $c>0$ we have
	\begin{equation*}
		\mathbb{P}^{AJ}
		\left( 
			K(n)<n^r
		\right)\le \exp\{-c n^{r}\}
	\end{equation*}
	for all large enough $n$.
\end{lemma}
By \eqref{eq:n_in_terms_of_R},
$\exp\left\{ -cn^r \right\}$ decays to zero faster than any power of $R$
as $R\to+\infty$
provided that $(p-p_0)r>1$ (which is ensured by choosing $p$ large enough).
\begin{proof}[Proof of \Cref{lemma:K_n_estimate}]
	We need to bound the probability that at least $n-\lfloor n^r/2 \rfloor $
	particles disappear during time $T$. 
	This is upper bounded by
	\begin{equation*}
		\mathbb{P}\Biggl( 
			\sum_{i=\lfloor n^r/2 \rfloor}^n
			E_i<T
		\Biggr),
	\end{equation*}
	where $\{E_i\}$ are independent exponential random variables, 
	where the rate of $E_i$ is equal to $4i$ (the rate is $4$
	since particles can annihilate with their left or right neighbors).
	This expression is bounded using 
	\cite[Theorem 5.1.(iii)]{janson2018tail}
	with $a_*=4\lfloor n^r/2 \rfloor $ (minimum of the rates of the~$E_i$'s),
	$\mu=\frac{1-r}{4}\log n+O(1)$ (mean of the sum), $\lambda=T/\mu\le 1$,
	and the bound is of the form
	$e^{-a_* \mu(\lambda-1-\log \lambda)}$. 
	The dominating term 
	in the exponent 
	(going fastest to $-\infty$
	as $n\to+\infty$) 
	is
	$-a^*\mu\log(1/\lambda)\sim -\mathrm{const}\cdot
	n^r (\log n )(\log \log n)
	$, which leads to an estimate $\le e^{-cn^r}$, as desired.
\end{proof}

\begin{figure}[htpb]
	\centering
	\includegraphics[width=.7\textwidth]{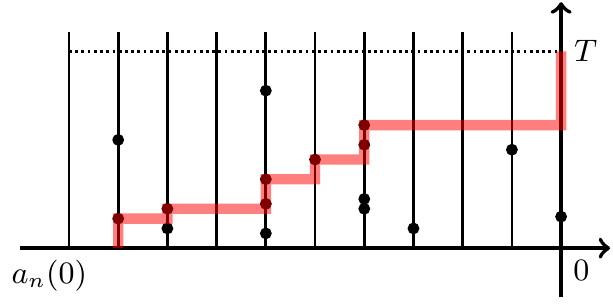}
	\caption{Semi-discrete Poisson percolation and an
	up-right path of percolation length 5. The environment consists of
	independent rate $1$ Poisson processes on $[0,T]$
	placed at integer locations. The last passage
	percolation length is the maximal number of Poisson points collected 
	by an up-right path from $(a_n(0)+1,0)$ to $(0,T)$,
	where we count at most one point per lattice site.
	In terms of
	jumps in the AJ system, the Poisson points collected by the maximal up-right
	path are the times
	when particles jump into the corresponding locations.
	We count at most ony point per lattice site since several clock
	rings at one site might not lead to actual jumps in the AJ system.}
	\label{fig:Poisson_percolation}
\end{figure}

\Cref{lemma:K_n_estimate}
bounds the number of annihilations. 
Let us now consider the case $K(n)\ge n^r$. 
Then the main contribution to
the event $a_1(T)\le -\tilde C(\log R)^p$
comes from having many jumps.
Without annihilations, the AJ system is simply the Hammersley
process on $\mathbb{Z}$, a natural analogue
of the Hammersley process on 
$\mathbb{R}$
introduced in \cite{hammersley1972few}, see also \cite{aldous1995hammersley},
\cite{SeppStickprocess}.
The Hammersley process on~$\mathbb{Z}$
(which is essentially the same as the PushTASEP via the
particle-hole involution) is coupled to the semi-discrete Poisson
last passage percolation
\cite[Section 5]{Toninelli2015-Gibbs}.
We thus observe that
when $K(n)\ge n^r$, the event
$a_1(T)\le -\tilde C(\log R)^p$
may occur only if
there exists an
up-right path 
of percolation length
at least $n^r$
in the (rate 1) semi-discrete Poisson last passage percolation
in the space-time region 
$(a_n(0),0]\times [0,T]\subset \mathbb{Z}\times \mathbb{R}$
(see \Cref{fig:Poisson_percolation} for an illustration).
Denote by $E_n$ the event that 
such a path exists.
We upper bound its probability as follows:

\begin{lemma}
	\label{lemma:E_n_K_n_second_estimate}
	Let $n$ be such that $a_n(0)\le -\tilde C(\log R)^p$. 
	Then we have for some $\tilde c>0$:
	\begin{equation*}
		\mathbb{P}^{AJ}\left( E_n\cap \left\{ K(n)\ge n^r \right\} \right)
		\le
		\frac{\tilde c\ssp |a_n(0)\ssp e\ssp T|^{n^r}}{(\lfloor n^r \rfloor !)^2}.
	\end{equation*}
\end{lemma}
\begin{proof}
	We first estimate the probability that the maximal
	(in the sense of last passage percolation)
	up-right path has length $M$, where $M\ge n^r$.
	The existence of such a path means that there are 
	$M$ Poisson 
	points picked up by the path
	(see \Cref{fig:Poisson_percolation} for an illustration). 
	The number of configurations of such points 
	is equal to the number of 
	sequences $x_1<x_2<\ldots<x_M$
	with $a_n(0)<x_1$ and $x_M\le 0$, which is upper bounded by
	$\frac{|a_n(0)|^M}{M!}$. 
	For a fixed sequence $\{x_i \}$,
	using the strong Markov property, we see that there must be 
	at least $M$ points in the rate $1$ Poisson process
	on the segment $[0,T]$. This probability
	is bounded from above by
	$\frac{(e T)^M}{M!}$. Thus, we have
	\begin{equation*}
		\mathbb{P}\left( \textnormal{the maximal up-right path
		has length $M$} \right)\le
		\frac{|a_n(0)\ssp e\ssp T|^M}{(M!)^2}.
	\end{equation*}
	We see that these quantities decay in $M$ faster than the terms of a geometric
	series, so their sum over $M\ge n^r$ (coming from the union bound)
	is bounded by a constant times the first term.
	This produces the desired estimate.
\end{proof}

To finalize the proof of \Cref{lemma:main_full_plane_lemma},
we
pick the constants $\tilde C$ and $p$
to 
bound the desired probability
$\mathbb{P}^{AJ}(
a_1(T) \leq - \tilde{C} (\log R)^p)$ from above.
We use \Cref{lemma:K_n_estimate,lemma:E_n_K_n_second_estimate}
and a union bound over 
$n\ge n_1 \coloneqq \lfloor \frac{\tilde C}{C}\ssp(\log R)^{p-p_0}\rfloor$,
so that $a_n(0)\le -\tilde C(\log R)^p$:
\begin{equation}
	\label{eq:final_thing_to_estimate}
	\mathbb{P}^{AJ}(
	a_1(T) \leq - \tilde{C} (\log R)^p)
	\le
	\sum_{n \ge n_1}
	\left( e^{-c n^{r}}+ 
	\frac{\tilde c\ssp|a_n(0)\ssp e\ssp T|^{n^r}
	}{(\lfloor n^r \rfloor !)^2}
	\right).
\end{equation}
The first series in \eqref{eq:final_thing_to_estimate} 
is bounded from above by
\begin{equation}
	\label{eq:final_thing_to_estimate_proof1}
	(\textnormal{power of $n_1$})\cdot
	\exp\left( -\mathrm{const}\cdot n_1^r \right)
	=
	(\textnormal{power of $\log R$})\cdot
	\exp\left( -\mathrm{const}\cdot (\log R)^{r(p-p_0)} \right).
\end{equation}
Indeed, one can bound the tail 
$\sum_{n \ge n_1}e^{-c n^{r}}$ by $e^{-cn_1^r}+\int_{n_1}^\infty e^{-cx^r}dx$.
The integral is equal to a constant times an incomplete Gamma integral
of the form $\int_z^\infty t^{a-1}e^{-t}dt$. For fixed $a$ and large $z$ (which is our
case), the behavior is of the form $e^{-z}$ times a power of $z$
(e.g., see \cite[8.11(i)]{DLMF}).
We thus see that by \eqref{eq:final_thing_to_estimate_proof1},
the first series in \eqref{eq:final_thing_to_estimate}
decays faster than any power of $R$ as $R\to+\infty$ as long
as $(p-p_0)r>1$.

For the second summand in \eqref{eq:final_thing_to_estimate} we have,
using \eqref{eq:n_in_terms_of_R} and \eqref{eq:bound_on_a_n}:
\begin{equation*}
	\begin{split}
		&
		\log
		\left( 
				\frac{\tilde c\ssp |a_n(0)\ssp e\ssp T|^{n^r}
				}{(\lfloor n^r \rfloor !)^2}
			\right)
		\le
		\mathrm{const}
		\cdot
		\Bigl(  
			n^r \log |a_n(0)| +n^r \log (eT)
			-2r n^r\log n+2n^r+O(\log n)
		\Bigr)
		\\&\hspace{80pt}
		\le
		\mathrm{const}
		\cdot
		\Bigl(  
			n^r 
			\bigl( (1-2r)(p-p_0)\log(\log R)+p_0\log(\log R) \bigr)
			+\textnormal{lower order terms}
		\Bigr)
	\end{split}
\end{equation*}
Since $r>\frac{1}{2}$,
by taking $p$ large enough we may 
make the first term dominate as $R\to+\infty$. This leads to an 
overall decay of the second sum in \eqref{eq:final_thing_to_estimate}
as $R\to+\infty$, faster than any power of $R$.

To conclude, the desired probability
\eqref{eq:final_thing_to_estimate}
is $O(R^{-\zeta})$ 
for any $\zeta>0$. Taking the final union bound
over all vertices in $\Lambda_R$, as discussed in the beginning
of this \Cref{sub:propbound_proof},
multiplies our estimate by $\mathrm{const}\cdot R^2$.
With this factor the probability still decays 
faster than any power of $R$, 
and so we arrive at the estimate in 
\Cref{lemma:main_full_plane_lemma}.

\section{Current and hydrodynamics}
\label{sec:hydrodynamics}

\subsection{Computing the current}

Recall that with each path configuration
of the six vertex model in $\mathbb{Z}^{2}$ we associate
the height function $h(x,y)$ (defined up to a constant), see 
\Cref{sub:dynamics_definition} in the Introduction.
The full plane dynamics $\mathcal{C}(t)$
gives rise to the time-dependent random function $h_t(x,y)$.
For the KPZ pure state $\pi(\mathsf{s})$
(we recall its definition in \Cref{sub:pure_states}),
we define the corresponding current (average change of the height function)
by
\begin{equation}
	\label{eq:current_def_text}
	J(\mathsf{s},\varphi(\mathsf{s})) \coloneqq 
	\frac{1}{t}
	\ssp
	\mathbb{E}_{\pi(\mathsf{s})}
	\left( h_t(0,0)-h_0(0,0) \right),
\end{equation}
where the initial height function $h_0$
corresponds to the configuration distributed as $\pi(\mathsf{s})$.
Instead of $(0,0)$, we could take an arbitrary face of the lattice (by 
translation invariance of the measure and the dynamics).
The right-hand side of \eqref{eq:current_def_text}
is independent of $t$, so we may send $t\to 0$, and write
\begin{equation*}
	J(\mathsf{s},\varphi(\mathsf{s})) =
	\frac{\partial}{\partial t}
	\ssp
	\mathbb{E}_{\pi(\mathsf{s})}\ssp
	h_t(0,0)\ssp\big\vert_{t=0}.
\end{equation*}
Therefore, we may compute the current by looking at the Markov 
generator $G_u$ of $\mathcal{C}(t)$ given by \eqref{eq:full_plane_generator}.
Namely, we have 
\begin{equation}
	\label{eq:J_average_rate_of_change}
	J(\mathsf{s},\varphi(\mathsf{s}))=
	\sum_{\textnormal{$(v_1,v_2)$ is a vertical edge}}
	\mathfrak{R}_{v_1,v_2}(u)\ssp
	\mathbb{E}_{\pi(\mathsf{s})}
	\left[ 
		\mathbf{1}_{\textnormal{$(v_1,v_2)$ is a seed pair for $h_0$}}\ssp
		\Delta_{v_1,v_2}h_0(0,0)
	\right],
\end{equation}
where
$\Delta_{v_1,v_2}h_0(0,0)$ 
is the (signed) change of the height function
at $(0,0)$
triggered by initiating the jump at the vertical edge $v_1-v_2$.
In taking the expectation, we assume that 
$h_0$ is distributed according to $\pi(\mathsf{s})$.
Recall 
$\mathfrak{R}_{v_1,v_2}(u)$
is equal to either of the quantities
$\mathfrak{a}(u),\mathfrak{b}(u)$, or $\mathfrak{c}(u)$
\eqref{eq:abc_rates}
depending
on the surrounding paths, see
\Cref{fig:rates_for_G_quad}.

We are now in a position to compute the current:

\begin{theorem}
	\label{thm:current_computation}
	We have
	\begin{equation}
		\label{eq:current_final_formula_in_the_text}
		J(\mathsf{s},\varphi(\mathsf{s}))
		=
		-\frac{\mathsf{s}\ssp(1 - \mathsf{s})}
		{(\mathsf{s} + u - \mathsf{s} u)^2}.
	\end{equation}
\end{theorem}
Using the function $\varphi(\mathsf{s})$, see
\eqref{eq:phi_of_rho_define}, we can also write
$J(\mathsf{s},\varphi(\mathsf{s})) =
-\frac{(1-\mathsf{s})(\varphi(\mathsf{s}))^2 }{\mathsf{s}} =
\frac{\partial}{\partial u}\ssp \varphi(\mathsf{s}) $.
\begin{proof}[Proof of \Cref{thm:current_computation}]
	We use \eqref{eq:J_average_rate_of_change}
	and the description of 
	$\pi(\mathsf{s})$
	as a trajectory of the stationary 
	stochastic six vertex model,
	as discussed in
	\Cref{sub:pure_states}.
	Throughout this proof, we denote $\mathsf{t}\coloneqq \varphi(\mathsf{s})$,
	for short, and use the notation $\delta_1,\delta_2$ for vertex weights,
	see \eqref{eq:delta_1_2_through_t_u}.

	We compute the average change of height at the face $(0,0)$
	by assuming that this change comes from a jump of a horizontal path
	of a specified structure to the left of $(0,0)$.
	Namely, 
	to the left of $(0,0)$
	we find the rightmost pair of vertices of one of the following six 
	kinds (here we use the traditional names for the 
	six vertices, see \Cref{fig:6types}):
	\begin{align*}
		&(c_1,a_1), (b_2,c_2), (c_1,c_2)     \qquad \text{(for up jumps, which contribute $-1$ height change)};\\
		&(b_1,c_1), (c_2,a_2), (b_1,a_2)    \qquad \text{(for down jumps, which contribute $+1$ height change)},
	\end{align*}
	which is at position $((x,0),(x,1))$ for some $x \leq 0$.
	(Note that $(c_1,c_2)$ is not a seed pair, but the others are).
	Our horizontal path starts
	from this initial pair of vertices, 
	crosses $n\ge0$ vertical paths (i.e., formed by three occupied vertical
	edges in a single column),
	and contains $n+1$ uninterrupted horizontal strings of $b_2$ vertices of lengths
	$k_0,k_1,\ldots,k_n \ge0$ in between the vertical paths. 
	The contribution from vertices 
	in the region $\{x+1, x+2,\dots,0 \} \times \{0,1 \}$ to
	the 
	probability of such a two-layer path configuration in $\mathbb{Z} \times \left\{ 0,1 \right\}$
	is
	$(\mathsf{s}\delta_1)^n \left( (1-\mathsf{s})\delta_2 \right)^{|\mathbf{k}|}$,
	where we denote $|\mathbf{k}|=k_0+k_1+\ldots+k_n $.
	Note that this contribution is the same in the two cases
	when the
	horizontal path goes through the bottom or the top layer.
	
	The rate of the height change contains 
	a term involving the vertex weights of the pair at $(x,0),(x,1)$ times the rate of initiating jump
	at this pair (this term is equal to zero for $(c_1,c_2)$ because it is not a 
	seed pair). Moreover, for up jumps, we also need to add 
	a term accounting for $|\mathbf{k}|$ extra seed pairs along the horizontal path 
	where a jump may also be initiated.
	
	Overall, we obtain the following expression
	for the current:
	\begin{align*}
		J(\mathsf{s},\varphi(\mathsf{s}))
		&=
		\sum_{n=0}^\infty\ssp\ssp
		\sum_{\mathbf{k} = (k_0,k_1,\dots,k_n) \in \mathbb{Z}_{\geq 0}^{n+1}}
		(\mathsf{s}\delta_1)^n \left( (1-\mathsf{s})\delta_2 \right)^{|\mathbf{k}|}
		\\
		&\hspace{20pt}\times
		\biggl[
				-\mathsf{s} (1-\mathsf{t})^2 (1-\delta_1)
				(\mathfrak{a}(u) + |\mathbf{k}| \mathfrak{c}(u) ) 
				-
				(1 - \mathsf{s}) \mathsf{t}^2 \delta_2
				(1-\delta_2)(\mathfrak{b}(u) + |\mathbf{k}| \mathfrak{c}(u))   
				\\
				&\hspace{50pt}
				- 
				\mathsf{s} \mathsf{t} (1-\mathsf{t})(1-\delta_1)
				(1-\delta_2) |\mathbf{k}| \mathfrak{c}(u)
				+
				\mathsf{s}
				(1-\mathsf{t})^2 \delta_1 (1-\delta_1) \mathfrak{b} (u)
				\\&\hspace{50pt}+ 
				(1 - \mathsf{s}) \mathsf{t}^2
				(1-\delta_2)\mathfrak{a}(u) 
				+
				\mathsf{s} \mathsf{t} (1-\mathsf{t}) \delta_1 \mathfrak{c}(u)
		\biggr].
	\end{align*}
	First one can compute the sum over $\mathbf{k}$, using
	\begin{equation*}
		\sum_{\mathbf{k} = (k_0,k_1,\dots,k_n) \in \mathbb{Z}_{\geq 0}^{n+1}}
		\xi^{|\mathbf{k}|}=
		\frac{1}{(1-\xi)^{n+1}},
	\end{equation*}
	and
	\begin{equation*}
		\xi\ssp\frac{\partial}{\partial\xi}
		\sum_{\mathbf{k} = (k_0,k_1,\dots,k_n) \in \mathbb{Z}_{\geq 0}^{n+1}}
		\xi^{|\mathbf{k}|}=
		\sum_{\mathbf{k} = (k_0,k_1,\dots,k_n) \in \mathbb{Z}_{\geq 0}^{n+1}}
		|\mathbf{k}|\ssp\xi^{|\mathbf{k}|}=
		\frac{(n+1) \xi}{(1-\xi )^{n+2}}.
	\end{equation*}
	Employing these identities leaves only the summation over $n$,
	which is readily computed. After necessary simplifications, 
	we arrive at the desired formula 
	\eqref{eq:current_final_formula_in_the_text}.
\end{proof}

\subsection{Heuristic hydrodynamics in the quadrant}
\label{sub:hydro_heuristic_text}

This subsection presents a heuristic discussion of 
some hydrodynamic Burgers type equations 
in one and two space dimensions related to the stochastic six vertex model
in the quadrant.

Recall the Markov dynamics $\mathcal{Q}(\tau)$
with the infinitesimal generator
$G^{\mathrm{quad}}_{u,\eta,\tau}$
\eqref{eq:quad_generator}. 
The dynamics $\mathcal{Q}(\tau)$
acts on path configurations
in the quadrant $\mathbb{Z}_{\ge0}\times
\mathbb{Z}_{\ge1}$, with 
step-$\uplambda$ or empty-$\uplambda$
boundary conditions. 
Recall that the subset 
$\uplambda\subset\mathbb{Z}_{\ge0}$
encodes the locations of incoming vertical arrows
along the bottom boundary,
and this subset stays fixed throughout the dynamics. 
By \Cref{thm:quad_process_exists},
$\mathcal{Q}(\tau)$
changes (in distribution)
the Gibbs property of the stochastic six vertex model
by continuously increasing the spectral parameter 
from $u$ to some terminal value $u+\eta\in(0,1)$.
Namely, we have
$\mathbb{P}_u^{\mathrm{s6v,\,hom}}\ssp
\mathcal{Q}(\tau)
=
\mathbb{P}_{u+(1-e^{-\tau})\ssp\eta}^{\mathrm{s6v,\,hom}}$,
where
$\mathbb{P}_u^{\mathrm{s6v,\,hom}}$
denotes the homogeneous
stochastic six vertex model with spectral
parameter $u$ and our fixed 
step-$\uplambda$ or empty-$\uplambda$
boundary conditions.
Denote the evolving spectral parameter by
\begin{equation}
	\label{eq:evolving_spectral_parameter_for_hydrodynamics}
	u(\tau)\coloneqq u+(1-e^{-\tau})\ssp\eta.
\end{equation}

Now consider the 
limit when the lattice coordinates
are $(\lfloor x/\varepsilon \rfloor ,
\lfloor y/\varepsilon \rfloor )$
for some $(x,y)\in \mathbb{R}_{\ge0}^2$,
and $\varepsilon\searrow0$.
Assume that the subset $\uplambda=\uplambda^{(\varepsilon)}$ depends on
$\varepsilon$
and behaves
regularly in the sense that its
height function is
\begin{equation*}
	\#
	\bigl\{ 
		l\in \uplambda^{(\varepsilon)}
		\colon l\le \lfloor x/\varepsilon \rfloor
	\bigr\}=
	\lfloor 
		\varepsilon^{-1}\ssp
		\tilde\uplambda(x)
		\rfloor
\end{equation*}
for all $\varepsilon$ and $x\in \mathbb{R}_{\ge0}$,
where $\tilde \uplambda$ is a fixed 
nondecreasing function with slope $\le 1$.
For each $\tau$, consider the
(random) height function 
$h_\tau(k,l)$, where $k\ge0$, $l\ge1$,
of the stochastic six vertex model
$\mathbb{P}_{u(\tau)}^{\mathrm{s6v,\,hom}}$,
which is defined as the (signed) number of 
up-right paths crossed between $(k,l)$
and $(0,0)$.
The height function increases by crossing a path 
going north or west, and decreases otherwise.
See \Cref{fig:height_fn}
from the Introduction
for an illustration.

From \cite[Theorem 1.1]{aggarwal2020limit} we know that
the random height function $h_\tau(k,l)$ admits a 
limit shape
\begin{equation*}
	\lim_{\varepsilon\to0}
	\varepsilon\ssp
	h_\tau\left( \lfloor x/\varepsilon \rfloor ,\lfloor y/\varepsilon \rfloor  \right)
	=\mathcal{H}(\tau,x,y), \qquad x,y\in \mathbb{R}_{\ge0},
\end{equation*}
with convergence in probability.
Here $\mathcal{H}(\tau,x,y)$ is a nonrandom function
with boundary conditions for all $\tau$:
\begin{equation*}
	\mathcal{H}(\tau,x,0)=\tilde \uplambda(x),\qquad 
	\mathcal{H}(\tau,0,y)=
	\begin{cases}
		0,&\textnormal{empty-$\uplambda$ boundary conditions};\\
		y,&\textnormal{step-$\uplambda$ boundary conditions}.
	\end{cases}
\end{equation*}
Also denote 
\begin{equation}
	\label{eq:rho_via_H_definition}
	\rho(\tau,x,y)\coloneqq -\frac{\partial}{\partial x}\mathcal{H}(\tau,x,y),
\end{equation}
this is the density of the occupied vertical edges near the global location $(x,y)$.
As our discussion in the current \Cref{sub:hydro_heuristic_text}
is heuristic, we assume that the derivative 
\eqref{eq:rho_via_H_definition} exists in a suitable sense (and similarly for all
other derivatives below).

There are two types of differential equations the function $\mathcal{H}(\tau,x,y)$
should satisfy:
\begin{itemize}
	\item 
		For each fixed $\tau$, the density $\rho(\tau,x,y)$
		should satisfy
		a version of the Burgers equation in (1+1) dimensions
		\cite[Theorem 1.1]{aggarwal2020limit}:
		\begin{equation}
			\label{eq:6v_1d_Burgers}
			\frac{\partial}{\partial y}\ssp 
			\rho(\tau,x,y)+
			\frac{\partial}{\partial x}
			\bigl( \varphi(\rho(\tau,x,y)\mid u(\tau)) \bigr)=0.
		\end{equation}
		This equation corresponds to the slice by slice 
		evolution under the transfer matrix
		of the stochastic six vertex model. Here 
		$y$ plays the role of time, and
		$\varphi(\rho\mid u(\tau))$ is the particle current 
		in stationarity on $\mathbb{Z}$ at density $\rho$.
	\item The height function should satisfy:
		\begin{equation}
			\label{eq:6v_2d_Burgers}
			\frac{\partial}{\partial \tau}
			\mathcal{H}
			(\tau,x,y)=e^{-\tau}\eta \ssp y\ssp J\left( -\frac{\partial}{\partial x}\mathcal{H}(\tau,x,y),
			\varphi\Bigl( -\frac{\partial}{\partial x}\mathcal{H}(\tau,x,y)\mid u(\tau) \Bigr)\right).
		\end{equation}
		This equation corresponds to the 
		fact that the $\mathcal{H}$'s are the limit shapes of the random height functions
		$h_\tau(k,l)$. Indeed, the latter are obtained from $h_0(k, l)$
		(by means of increasing $\tau$)
		using the Markov dynamics $\mathcal{Q}(\tau)$.
		Finally, the average velocity of the height function 
		in the bulk around $(x,y)$ under $\mathcal{Q}(\tau)$ is 
		$e^{-\tau} \eta \ssp y\ssp J(\rho,\varphi(\rho\mid u(\tau)))$,
		where $\rho=\rho(\tau,x,y)$ is the density of the occupied vertical edges, and 
		the factor $e^{-\tau}\eta\ssp y$ is due to the inhomogeneity of the edge Poisson 
		clocks, cf. \eqref{eq:quad_generator}.
\end{itemize}

The two equations
\eqref{eq:6v_1d_Burgers}--\eqref{eq:6v_2d_Burgers}
are consistent in the following sense:
\begin{proposition}\label{conj:hydro}
	Let $\mathcal{H}_\tau(x,y)$ be a family of limiting height functions 
	for the stochastic six vertex models
	$\mathbb{P}_{u+(1-e^{-\tau})\ssp\eta}^{\mathrm{s6v,\,hom}}$
	in the quadrant.
	If 
	\begin{equation}
		\label{eq:6v_2d_Burgers_with_tilde}
		\frac{\partial}{\partial\tau}\ssp \mathcal{H}(\tau,x,y)
		=
		e^{-\tau}\eta\ssp y
		\ssp\tilde J( \rho(\tau,x,y) \mid u(\tau))
	\end{equation}
	for some function $\tilde J(\rho\mid u)$
	(and with $\rho$ given by 
	\eqref{eq:rho_via_H_definition}), then we must have
	\begin{equation}
		\label{eq:hydro_theorem}
		\frac{\partial}{\partial \rho}\ssp\tilde J(\rho\mid u)=
		\frac{\partial}{\partial\rho}\frac{\partial}{\partial u}\ssp
		\varphi(\rho\mid u)
	\end{equation}
	for all $\tau,x,y$ for which
	$\frac{\partial}{\partial x}\rho(\tau,x,y)\ne 0$.
\end{proposition}
In words, if the time-dependent
limiting height functions satisfy any (2+1)-dimensional
hydrodynamic equation of a certain form
(corresponding to inhomogeneous edge rates), 
then the right-hand side is the same as in
\eqref{eq:6v_2d_Burgers}, up to a constant depending on $\rho$.
Clearly, the velocity $J$ given by
\eqref{eq:current_final_formula_in_the_text}
satisfies \eqref{eq:hydro_theorem}.
\begin{proof}[Proof of \Cref{conj:hydro}]
	Throughout the proof we denote derivatives 
	by lower indices like $\rho_\tau$,
	and also sometimes by $\partial_\tau\rho$ when convenient.
	Differentiating \eqref{eq:6v_2d_Burgers_with_tilde} in $x$, we get
	\begin{equation}
		\label{eq:hydro_proof_1}
		\rho_\tau=
		-
		e^{-\tau}\eta\ssp y\ssp
		\rho_x\ssp
		\tilde J_\rho( \rho \mid u(\tau))
		.
	\end{equation}
	Assuming that the Burgers equation \eqref{eq:6v_1d_Burgers}
	holds at time $\tau$, we use
	\eqref{eq:hydro_proof_1}
	to write down the same equation at time $\tau+\Delta\tau$:
	\begin{equation}
		\label{eq:hydro_proof_2}
		\partial_y
		\left( \rho-
		e^{-\tau}\eta\ssp y\ssp
		\rho_x\ssp
		\tilde J_\rho( \rho \mid u(\tau))\Delta \tau
		\right)
		+
		\partial_x
		\left( \varphi
		\left(  
			\rho-
			e^{-\tau}\eta\ssp y\ssp
			\rho_x\ssp
			\tilde J_\rho( \rho \mid u(\tau))\Delta \tau
		\mid u(\tau+\Delta \tau)\right)\right)
		=
		0.
	\end{equation}
	In \eqref{eq:hydro_proof_1}, \eqref{eq:hydro_proof_2} we use the shorthand $\rho=\rho(\tau,x,y)$, so
	this is the quantity at time $\tau$.
	Equating the coefficient by $\Delta\tau$ in \eqref{eq:hydro_proof_2} to zero,
	and multiplying by $e^{\tau}$,
	we obtain 
	\begin{equation}
		\label{eq:hydro_proof_3}
		-\partial_y
		\left( 
		\eta\ssp y\ssp
		\rho_x\ssp
		\tilde J_\rho( \rho \mid u(\tau))
		\right)
		+
		\partial_x
		\left( 
			-\varphi_\rho\ssp
			\eta\ssp y\ssp
			\rho_x\ssp
			\tilde J_\rho( \rho \mid u(\tau))
			+
			e^{\tau}
			\varphi_u u_\tau
		\right)=0.
	\end{equation}
	We have
	$e^{\tau} u_\tau=\eta$.
	Dividing \eqref{eq:hydro_proof_3} by $(-\eta)$, we continue as
	\begin{equation}
		\label{eq:hydro_proof_4}
		0=\tilde J_\rho \rho_x+y \tilde J_{\rho \rho}\rho_x \rho_y+y \tilde J_\rho \rho_{xy}
		+
		y\varphi_{\rho\rho}\rho_x^2 \tilde J_\rho
		+
		y \varphi_\rho \rho_{xx}\tilde J_\rho
		+
		y \varphi_\rho \rho_x^2 \tilde J_{\rho\rho}
		-
		\varphi_{u\rho}\rho_x.
	\end{equation}
	Substituting the expression for $\varphi$
	\eqref{eq:phi_of_rho_define}
	and using the Burgers equation \eqref{eq:6v_1d_Burgers} at time $\tau$
	to express $\rho_y$ through $\rho_x$, 
	we see that \eqref{eq:hydro_proof_4} reduces to
	\begin{equation*}
		\tilde J_\rho\rho_x-\varphi_{\rho u}\rho_x=0,
	\end{equation*}
	as desired.
\end{proof}

\section{Dynamics on the torus}
\label{sec:torus}

In this section 
we define an analogue of the full plane dynamics 
$\mathcal{C}(t)$ constructed in \Cref{sec:Markov_full_plane_process}
which acts on up-right path configurations on the torus.
Our dynamics preserve the six vertex model Gibbs measures
with arbitrary slopes $(\mathsf{s},\mathsf{t})$. 
We present
two proofs.
The first immediately follows from the Yang--Baxter equation and its bijectivisation,
and involves a certain discrete twist of the six vertex graph on the torus.
For simplicity, in the first proof we only consider the stochastic six vertex weights
for which we already have explicit jump rates.
For the second proof, we mimic the 
argument of
\cite{BorodinBufetov2015} involving symmetry of jump rates,
and this 
allows to generalize our torus dynamics
to arbitrary six vertex weights $a_1,a_2,b_1,b_2,c_1,c_2$.

\subsection{Bijectivisation on the torus}

Suppose we have an $M \times N$ torus 
$\mathbb{Z}/M \mathbb{Z} \times \mathbb{Z}/N \mathbb{Z}$
denoted by $\mathbb{T}=\mathbb{T}_{M,N}$. 
Consider the set
$\mathcal{S}_{k_1,k_2}$ of configurations 
of up-right paths on the torus
with fixed
overall height change $k_1$ and $k_2$ in the $x$ and $y$
directions, respectively. 
Let $\mu_{k_1,k_2}$ denote the Gibbs measure
on $\mathcal{S}_{k_1,k_2}$ given by a choice of 
six vertex 
Boltzmann
weights $a_1,a_2,b_1,b_2,c_1,c_2$
(see \Cref{fig:6types} for an illustration of the weights).
We index the vertices by
$(x,y)$ where $x \in \{0, \dots, M-1\}$, $y \in \{0, \dots,
N-1\}$. If $x$ or $y$ is outside of this range, we 
reduce these coordinates modulo $M$ or $N$, respectively.

First, we consider the special stochastic case 
$a_1=a_2=1$, $b_1=1-c_1=\delta_1$, $b_2=1-c_2=\delta_2$,
where $\delta_1,\delta_2$ depend on $q$ and $u$, see \eqref{eq:delta_1_2_through_t_u}.
Let us define a continuous time Markov chain 
on up-right path configurations on the torus.
We employ the notion of seed pairs
(\Cref{def:seed_pair}) and
the jump rates $\mathfrak{R}_{v_1,v_2}(u)$
from \Cref{fig:rates_for_G_quad}.

\begin{definition}[Continuous time Markov dynamics on the torus]
\label{def:torus_jump_rates}
	Each vertical edge  $(v_1,v_2)$
	which is a seed pair has 
	an exponential clock with rate $\mathfrak{R}_{v_1,v_2}(u)$. 
	When the clock at $(v_1,v_2)$ rings,
	the horizontal path passing through this seed pair
	jumps up or down 
	depending on the local path configuration
	around the edge $v_1-v_2$.
	This jump then instantaneously 
	propagates to the right.
	The jump propagation may stop in two ways:
	\begin{itemize}
		\item Either there exists a configuration to the right 
			where the jump may stop 
			in the same way as in the full plane dynamics,
			according to the rules described in 
			\Cref{sub:sample}. See
			\Cref{fig:jump_prop} for an illustration.
		\item Or
			there is no stopping configuration,
			and the
			jump propagation has to make
			a loop around the torus, leading to a jump of a full straight horizontal path.
			See \Cref{fig:torus_jump} for an illustration.
	\end{itemize}
	We denote 
	by $\mathcal{L}(\tau)$
	the continuous time Markov semigroup 
	of thus defined process.
\end{definition}

\begin{figure}[htpb]
	\centering
	\includegraphics[width=.7\textwidth]{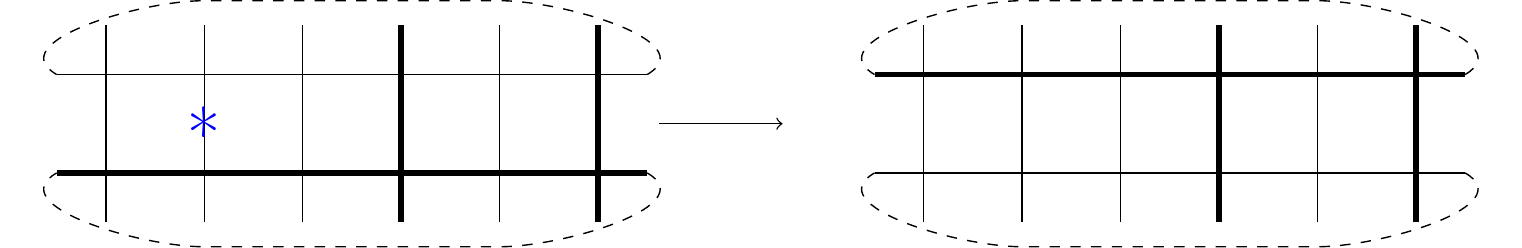}
	\caption{An example of a jump where 
	the propagation goes all the way 
	around the torus, and the full straight horizontal path
	jumps up. The identification of the horizontal edges is also shown.}
	\label{fig:torus_jump}
\end{figure}

We will show that 
in the case of 
the stochastic six vertex weights,
the Markov chain 
$\mathcal{L}(\tau)$ preserves the Gibbs measure
$\mu_{k_1,k_2}$ for any $k_1,k_2$.
We achieve this by constructing $\mathcal{L}(\tau)$ as a
Poisson type continuous time limit of a discrete time Markov chain coming from the bijectivisation
of the Yang--Baxter equation for the stochastic six vertex model 
defined in \Cref{sec:bijectivisation}.
In the torus case it turn out to be very convenient to define the
discrete
Markov chain not on the straight torus 
$\mathbb{T}$,
but on its suitably twisted version.

\begin{definition}
	Let the (\emph{discretely}) \emph{twisted torus}
	$\tilde{\mathbb{T}}=\tilde{\mathbb{T}}_{M,N}$
	be the graph displayed in
	\Cref{fig:twisted_torus}.
	We associate the spectral parameters
	$u,u+\varepsilon,\ldots,u+\varepsilon $ with the 
	horizontal strands as shown in this figure, 
	where $0<u<1$ and $0<\varepsilon<1-u$.
\end{definition}

\begin{figure}[htpb]
	\centering
	\includegraphics[height=.25\textwidth]{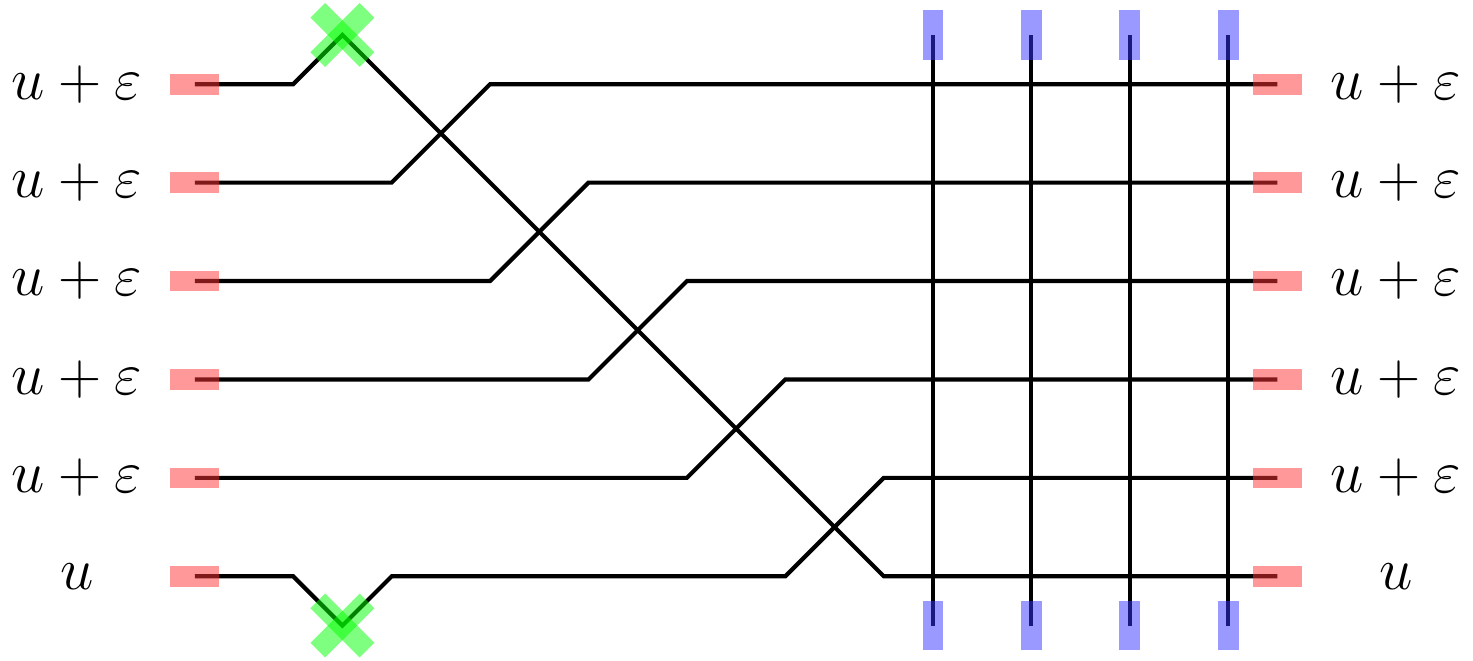}
	\caption{The twisted six vertex graph on the torus with $M=4$ and $N=6$. The identified edges are
	indicated in bold, and note that on the very left we 
	also identify the pieces of a diagonal cross. The spectral parameters associated to the 
	horizontal strands are $u,u+\varepsilon,u+\varepsilon,\ldots,u+\varepsilon$.}
	\label{fig:twisted_torus}
\end{figure}

Since the graph $\tilde{\mathbb{T}}$
is embedded into the usual torus,
on $\tilde{\mathbb{T}}$ the height function is still well-defined
(up to a constant) on the faces of the fundamental domain.
The horizontal and the vertical 
height change
along a generator of each homology class are also well
defined. See \Cref{fig:twisted_torus_hf} for an example of a path configuration and the height function.
Denote by $\mathcal{S}^{\mathrm{twist}}_{k_1,k_2}$
the six vertex configurations on $\tilde{\mathbb{T}}$ where
the horizontal height change is $k_1$, and the vertical
height change is $k_2$. Denote by $\mu_{\varepsilon,k_1,k_2}$
the Gibbs measure restricted to six vertex configurations in
$\mathcal{S}^{\mathrm{twist}}_{k_1,k_2}$, where we use
stochastic vertex weights parameterized by the spectral parameters
$u$ on the bottom row and $u+\varepsilon$ on every other
row. For cross vertices, we take their weights
equal to $X_{u,u+\varepsilon}=w_{u/(u+\varepsilon)}$
\eqref{eq:X_vertex}, which are the weights entering the 
Yang--Baxter equation.
Note that all these vertex weights on $\tilde{\mathbb{T}}$ are nonnegative.

\begin{figure}[h]
	\centering
	\includegraphics[height=.25\textwidth]{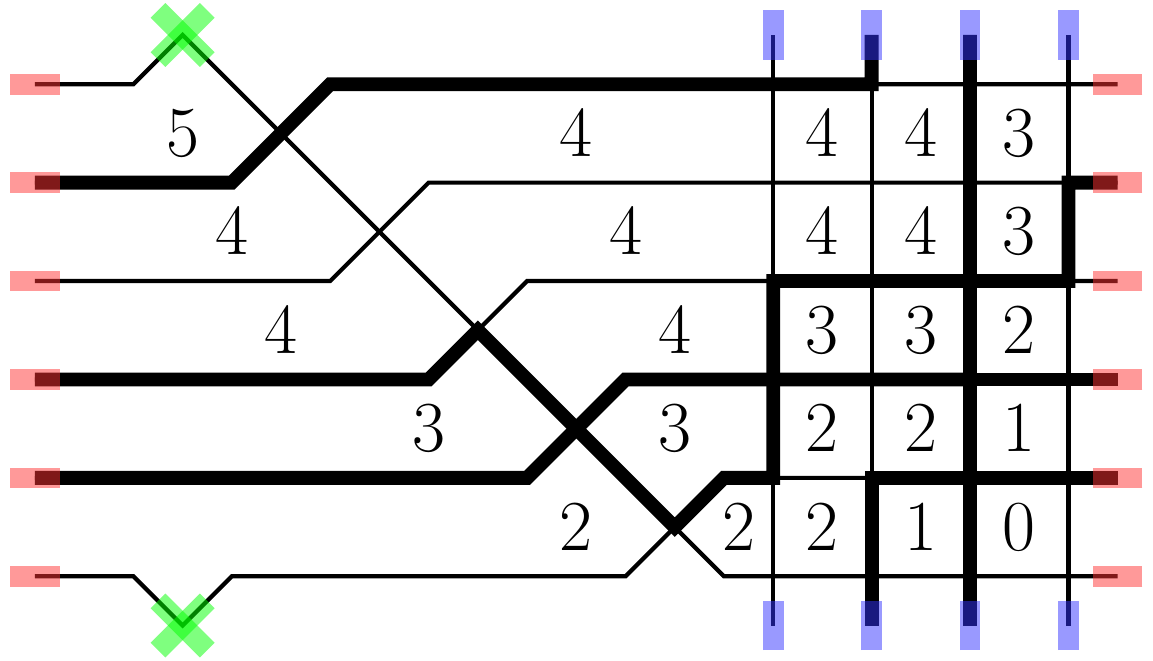}
	\caption{
		An up-right path configuration and the corresponding height function which is well-defined 
		(up to a constant)
		on faces of the fundamental domain.
		Here $k_1=2$, $k_2=3$.}
	\label{fig:twisted_torus_hf}
\end{figure}

\begin{definition}
	\label{def:discrete_chain}
	For $s,s'\in \mathcal{S}^{\mathrm{twist}}_{k_1,k_2}$
	we define the Markov transition probability
	$L_{\varepsilon}(s\to s')$ 
	by starting from $s$ and performing the following sequence of $N$ random updates:
	\begin{itemize}
	\item First, drag the cross vertex between rows $0$ and $1$
		through the lattice until it is to the right of $x =
		M-1$. Each step of dragging the cross is a 
		random update coming from the bijectivisation
		of the Yang--Baxter equation, see \Cref{fig:bijProbs}.
		After this, the spectral parameters on rows
		$0,\dots, N-1$ are $u+\varepsilon, u, u + \varepsilon,
		\dots, u + \varepsilon$.
	\item Then drag the cross 
		between rows $1$ and $2$
		through the lattice 
		using bijectivisation
		until it is to the right of $x =
		M-1$. After this, the spectral parameters are
		$u+\varepsilon, u + \varepsilon,u, u+\varepsilon, \dots,
		u + \varepsilon$.\\
	\vdots
	\item At the last step, drag the cross between rows $N-1$ and $0$ through
		the lattice. After this step, the 
		lattice returns to the original state, and 
		the spectral parameters are back
		to $u,u+\varepsilon,u+\varepsilon,\ldots,u+\varepsilon $.
	\end{itemize}
\end{definition}
 
The following statement ensures that 
$L_{\varepsilon}$ is well-defined:
\begin{lemma}
	The random updates described in \Cref{def:discrete_chain}
	preserve 
	$\mathcal{S}^{\mathrm{twist}}_{k_1,k_2}$.
\end{lemma}
\begin{proof}
	At each step of dragging the cross the path configuration
	and the associated height function only change locally, so
	the height change over the whole torus is preserved.
\end{proof}

\begin{proposition}
	\label{prop:torus_discrete_preserve}
		The Markov chain $L_\varepsilon$ preserves the measure
		$\mu_{\varepsilon,k_1,k_2}$.
\end{proposition}
\begin{proof}
	This follows from the fact that each 
	step of dragging the cross through the whole torus 
	maps the current Gibbs measure into a Gibbs
	measure with swapped spectral parameters
	(see \Cref{prop:U_swap_uv_tworow,prop:action_of_L_k_on_s6v}).
	After all $N$ steps, the spectral parameters
	are back to the original sequence $u,u+\varepsilon,u+\varepsilon,\ldots,u+\varepsilon $,
	and hence the Gibbs measure is preserved.
\end{proof}

Let us now discuss the limit as $\varepsilon\to 0$.

\begin{proposition}
	In the limit as $\varepsilon\to 0$, the Gibbs measure
	$\mu_{\varepsilon,k_1,k_2}$ 
	on the twisted graph
	can be identified with 
	the measure $\mu_{k_1,k_2}$ on the straight torus $\mathbb{T}$.
	\label{prop:torus_measure_epsilon_to_zero}
\end{proposition}
\begin{proof}
	As $\varepsilon\to0$, the cross vertex weights
	become $X_{u,u}=w_1$, and place zero weight onto the 
	vertices $(1,0;1,0)$ and $(0,1;0,1)$, see
	\eqref{eq:w_u_weights}.
	With this restriction on the cross vertex types, 
	we may identify path configurations on $\tilde{\mathbb{T}}$
	with those on $\mathbb{T}$, see \Cref{fig:twisted_torus_epsilon}
	for an illustration. Therefore,
	the measure $\mu_{\varepsilon,k_1,k_2}$ for $\varepsilon=0$
	is determined only by the usual vertices and not the cross vertices, 
	and thus coincides with $\mu_{k_1,k_2}$.
\end{proof}

\begin{figure}[htpb]
	\centering
	\includegraphics[height=.25\textwidth]{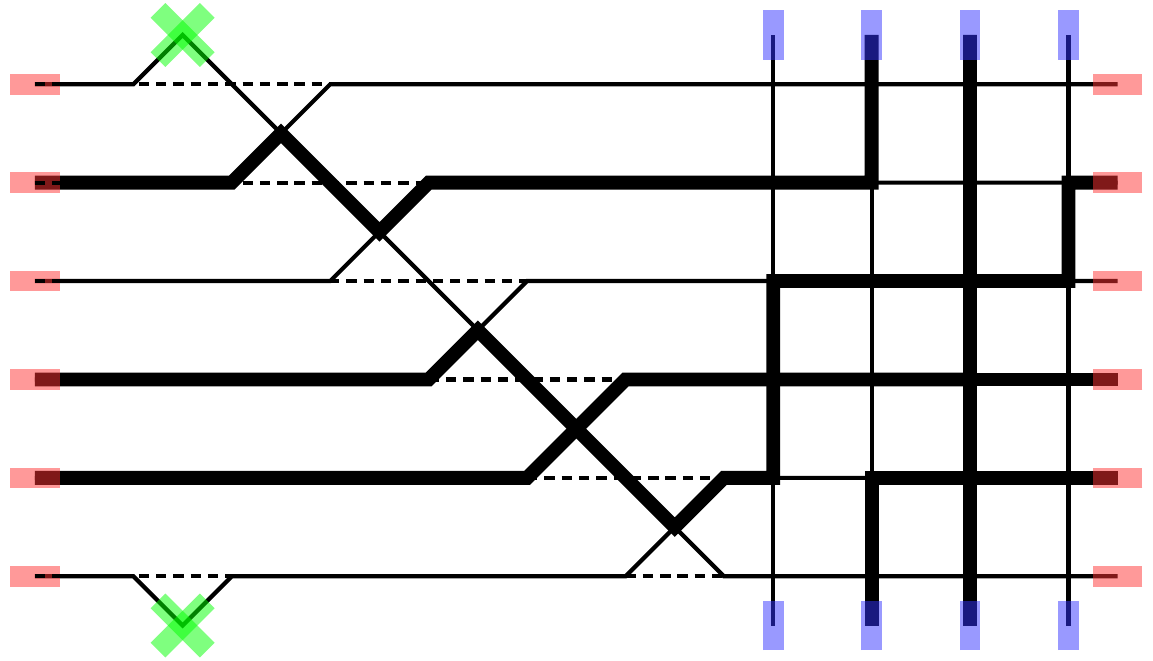}
	\caption{A path configuration under $\mu_{\varepsilon,k_1,k_2}$ for 
	$\varepsilon=0$. In this case, the state of a cross vertex is
	completely determined by the paths to the right of it.
	Cutting the cross vertices out and identifying the strands as shown by the dashed lines
	leads to the straight torus graph $\mathbb{T}$.}
	\label{fig:twisted_torus_epsilon}
\end{figure}

Consider the Poisson type continuous time limit
as $\varepsilon\to0$ of the iterated Markov transition operators
$L_\varepsilon^{\lfloor \tau/\varepsilon \rfloor }$, 
where $\tau\in \mathbb{R}_{\ge0}$.
We know from
\Cref{sec:continuous_time_limit}
that muliple dragging of the crosses should be replaced by jumps
of the horizontal paths initiated by Poisson clocks
placed onto vertical edges.
One readily sees that in this limit we have
\begin{equation*}
	\lim_{\varepsilon\to 0}
	L_\varepsilon^{\lfloor \tau/\varepsilon \rfloor }
	=
	\mathcal{L}(\tau),
\end{equation*}
where $\mathcal{L}(\tau)$ is the Markov transition operator
(over time $\tau$) from \Cref{def:torus_jump_rates}.
This convergence together with 
\Cref{prop:torus_discrete_preserve,prop:torus_measure_epsilon_to_zero}
implies the following result:
\begin{theorem}
	\label{thm:torus_preserves}
	The continuous time Markov process
	$\mathcal{L}(\tau)$
	on stochastic six vertex configurations on the 
	torus $\mathbb{T}$ preserves the measure
	$\mu_{k_1,k_2}$ for arbitrary $k_1,k_2$.
\end{theorem}

\begin{remark}[Comparison with a similar process from \cite{BorodinBufetov2015}]
	In \cite{BorodinBufetov2015}, the authors
	introduce a continuous time Markov process (denote it by $\hat {\mathcal{L}}$)
	preserving the measure $\mu_{k_1,k_2}$ on the torus.
	Rotating $\hat{\mathcal{L}}$ by
	$\pi/2$
	counterclockwise and reflecting along the $x$ direction
	makes the jumps in $\hat{\mathcal{L}}$ move the horizontal paths by jumps triggered by vertical edges.
	That is, we may describe both 
	$\mathcal{L}$ and $\hat{\mathcal{L}}$
	in similar terms.
	Moreover, after a scalar time renormalization, 
	some of the jumps and their jump rates exactly
	coincide in
	$\mathcal{L}$ and $\hat{\mathcal{L}}$.
	However, this transformation does not 
	make all the rates in both processes equal,
	which shows that the two processes are not the same. 

	The processes 
	$\mathcal{L}$ and $\hat{\mathcal{L}}$
	share another common feature, namely, that one can prove the 
	preservation of the measure $\mu_{k_1,k_2}$
	under both using a certain symmetry of the jump rates. We 
	present such an argument for our processes $\mathcal{L}$
	in \Cref{sub:torus_general} below.
\end{remark}

\subsection{Dynamics on the torus for general six vertex model}
\label{sub:torus_general}

Up to an overall constant $\upeta>0$, the jump rates 
$\mathfrak{a}(u)$, $\mathfrak{b}(u)$, and $\mathfrak{c}(u)$
\eqref{eq:abc_rates}
in the Markov chain $\mathcal{L}(\tau)$ from \Cref{def:torus_jump_rates}
can be
written in terms of the general six vertex weights 
$a_1,a_2,b_1,b_2,c_1,c_2$
(recall \Cref{fig:6types}) as follows:
\begin{equation}
	\label{eq:abc_general_parameters_rates}
	\mathfrak{c} = \upeta \frac{c_1 c_2}{\sqrt{b_1 b_2 a_1 a_2} },\qquad 
	\mathfrak{a} = \upeta \frac{\sqrt{b_1 b_2}}{\sqrt{a_1 a_2}},\qquad 
	\mathfrak{b} = \upeta \frac{\sqrt{a_1 a_2}}{\sqrt{b_1 b_2}} .
\end{equation}
Extending the definition, let 
$\mathcal{L}(\tau)$ 
denote the Markov process on configurations on the torus
with the jump rates depending on 
the generic six vertex weights
$a_1,a_2,b_1,b_2,c_1,c_2$
as in \eqref{eq:abc_general_parameters_rates}.
Similarly, let $\mu_{k_1,k_2}$ be the
six vertex model on the torus $\mathbb{T}$
determined by these generic vertex weights
(and horizontal and vertical height changes $k_1,k_2$).
We can extend
\Cref{thm:torus_preserves}
to the general weights.
\begin{theorem}
	\label{thm:torus_preserves_generic}
	The process $\mathcal{L}(\tau)$ preserves the measure $\mu_{k_1,k_2}$
	for arbitrary $k_1,k_2$, and for general six vertex weights
	$a_1,a_2,b_1,b_2,c_1,c_2$.
\end{theorem}
\begin{proof}
	For brevity, we will not reproduce here the details
	from
	\cite{BorodinBufetov2015}, and simply follow the argument and notation of that paper.
	The main ingredient is to
	check that \cite[Lemma 2]{BorodinBufetov2015}
	(a symmetry of the jump rates under the flip transformation)
	applies to our Markov process $\mathcal{L}(\tau)$. 
	This is indeed the case, as can be seen by a direct inspection
	of all the jump rates.

	Then the preservation of $\mu_{k_1,k_2}$ under $\mathcal{L}(\tau)$ follows similarly to
	\cite[Theorem 5]{BorodinBufetov2015}. We need to show
	\begin{multline*}
		\sum_s \mu_{k_1,k_2}(s)\ssp 
		\mathrm{Rate}(s \rightarrow s_0)  
		-
		\mu_{k_1,k_2}(s_0) 
		\sum_{s_2} 
		\mathrm{Rate}(s_0 \rightarrow s_2) 
		\\= 
		\mu_{k_1,k_2}(s_0) 
		\left( \sum_s \mathrm{Rate}(\overline{s_0} \rightarrow \overline{s})
		-
		\mathrm{Rate}(s_0 \rightarrow s) \right).
	\end{multline*}
	The sum on the right is equal to
	\begin{multline*}
		\mathfrak{c} (N(a_1, b_2) - N(b_2, a_1)) + \mathfrak{a}(N(a_1, c_2) - N(c_1, a_1)) + 
		\mathfrak{b}(N(c_1, b_2) - N(b_2, c_2))  
		\\+
		\mathfrak{c} (N(a_2, b_2) - N(b_2, a_2)) + \mathfrak{a}(N(a_2, c_1) - N(c_2, a_2)) + 
		\mathfrak{b}(N(c_2, b_1) - N(b_1, c_1)),
	\end{multline*}
	$N(X,Y)$ is the number of vertically adjacent vertices in the torus
	where the lower is of type $X$ and the upper is of type $Y$.
	By \cite[Lemma 4]{BorodinBufetov2015}, we see that this sum vanishes,
	which completes the proof.
\end{proof}

\subsection{Degeneration to five vertex model and lozenge tilings}
\label{sub:lozenges}

Let us set the weight $a_2$ of the vertex $(1,1;1,1)$ to zero.
This turns the six vertex model
into the five vertex model, which may be viewed as a certain 
model of nonintersecting (but interacting) paths,
or, equivalently, lozenge tilings on the triangular lattice (with interacting lozenges).
The five vertex model admits a more detailed asymptotic analysis than the general six 
vertex one by means of the Bethe Ansatz, for example, see the recent work 
\cite{deGierKenyon2021limit}.

Further letting $b_1b_2=c_1c_2$ 
makes
the five vertex model
free fermion
by removing the interaction.
The free fermion five vertex model is equivalent to a dimer model,
and may be analyzed asymptotically
through determinantal point processes, see, for example, \cite{ABPW2021free}.

Let us consider the degeneration of our Markov dynamics on the torus at $a_2=0$.
Under a suitable renormalization,
the rates \eqref{eq:abc_general_parameters_rates}
for $a_2=0$
reduce to
\begin{equation}
	\label{eq:abc_general_parameters_rates_five_vertex}
	\mathfrak{c}=\frac{c_1c_2}{\sqrt{b_1b_2}},\qquad 
	\mathfrak{a}=\sqrt{b_1b_2},\qquad 
	\mathfrak{b}=0.
\end{equation}
Because $a_2=0$ and $\mathfrak{b}=0$, out of six possible seed pairs
in
\Cref{fig:rates_for_G_quad} only two may lead to a jump. Both these
seed pairs correspond to up jumps, so our Markov process
becomes \emph{totally asymmetric}. 
By \Cref{thm:torus_preserves_generic},
the dynamics on the torus with rates 
\eqref{eq:abc_general_parameters_rates_five_vertex}
preserves the five vertex model (in particular, the one considered in \cite{deGierKenyon2021limit}).
Determining the particle
current of the dynamics preserving the five vertex model
could be simpler than in the general six vertex case, but this is outside the 
scope of the present work.

In the free fermion case 
$b_1b_2=c_1c_2$ we see that $\mathfrak{c}=\mathfrak{a}$.
After mapping nonintersecting paths of the free fermion five vertex model
to lozenge tilings, one readily sees that our dynamics
reduces to the totally asymmetric case of the interacting Hammersley processes
studied in \cite{Toninelli2015-Gibbs},
\cite{chhita2017combinatorial}
\cite{CFT2019}.


\medskip

\textsc{M. Nicoletti, Massachusetts Institute of Technology, Cambridge, MA}

E-mail: \texttt{lenia.petrov@gmail.com}

\medskip

\textsc{L. Petrov, Mathematical Sciences Research Institute, Berkeley, CA, University of Virginia, Charlottesville, VA, and Institute for Information Transmission Problems, Moscow, Russia}

E-mail: \texttt{lenia.petrov@gmail.com}

\end{document}